\documentclass[a4paper,11pt]{amsart}

\usepackage[left=2.7cm,right=2.7cm,top=3.5cm,bottom=3cm]{geometry}
\usepackage{amsthm,amssymb,amsmath,amsfonts,mathrsfs,amscd}
\usepackage[latin1]{inputenc}
\usepackage[all]{xy}

\theoremstyle{definition}
\newtheorem{defi}{Definition}[section]
\newtheorem{assumption}[defi]{Assumption}

\theoremstyle{remark}
\newtheorem{rem}[defi]{Remark}

\theoremstyle{plain}
\newtheorem{lemma}[defi]{Lemma}
\newtheorem{coro}[defi]{Corollary}
\newtheorem{teo}[defi]{Theorem}
\newtheorem{prop}[defi]{Proposition}
\newtheorem{conj}[defi]{Conjecture}

\DeclareMathOperator{\Pic}{Pic}
\DeclareMathOperator{\End}{End}

\DeclareMathOperator{\Frob}{Frob}

\DeclareMathOperator{\Norm}{Norm}

\DeclareMathOperator{\Hom}{Hom}
\DeclareMathOperator{\Div}{Div}
\DeclareMathOperator{\Gal}{Gal}
\DeclareMathOperator{\GL}{GL}

\DeclareMathOperator{\SL}{SL}
\DeclareMathOperator{\Sel}{Sel}
\DeclareMathOperator{\M}{M}

\newcommand{\res}{\mathrm{res}}

\newcommand{\cyc}{{\rom{cyc}}}
\newcommand{\cor}{\mathrm{cor}}

\newcommand{\ord}{\mathrm{ord}}
\newcommand{\new}{\mathrm{new}}
\newcommand{\ab}{\mathrm{ab}}
\newcommand{\tame}{\mathrm{tame}}
\newcommand{\wild}{\mathrm{wild}}

\newcommand{\wt}{\widetilde}
\newcommand{\T}{\mathbb T}
\newcommand{\p}{\boldsymbol p}

\newcommand{\Ta}{\mathrm{Ta}}

\newcommand{\D}{\mathbf D}

\newcommand{\B}{\mathbb B}

\newcommand{\Q}{\mathbb Q}
\newcommand{\N}{\mathbb N}
\newcommand{\Z}{\mathbb Z}
\newcommand{\R}{\mathbb R}
\newcommand{\C}{\mathbb C}

\newcommand{\rom}{\mathrm}
\newcommand{\fr}{\mathfrak}
\newcommand{\cl}{\mathcal}

\newcommand{\longmono}{\mbox{$\lhook\joinrel\longrightarrow$}}

\newcommand{\longepi}{\mbox{$\relbar\joinrel\twoheadrightarrow$}}

\newcommand{\smallmat}[4]{\bigl(\begin{smallmatrix}#1&#2\\#3&#4\end{smallmatrix}\bigr)}
\newcommand{\dirlim}{\mathop{\varinjlim}\limits}
\newcommand{\invlim}{\mathop{\varprojlim}\limits}

\begin{document}

\title{Quaternion algebras, Heegner points and the arithmetic of Hida families}

\author{Matteo Longo and Stefano Vigni}

\address{M. L.: Dipartimento di Matematica Pura e Applicata, Universit{\`a} di Padova, Via Trieste 63, 35121 Padova, Italy}
\email{mlongo@math.unipd.it}
\address{S. V.: Departament de Matem{\`a}tica Aplicada II, Universitat Polit{\`e}cnica de Catalunya, C. Jordi Girona 1-3, 08034 Barcelona, Spain}
\email{stefano.vigni@upc.edu}


\begin{abstract}
Given a newform $f$, we extend Howard's results on the variation of Heegner points in the Hida family of $f$ to a general quaternionic setting. More precisely, we build big Heegner points and big Heegner classes in terms of compatible families of Heegner points on towers of Shimura curves. The novelty of our approach, which systematically exploits the theory of optimal embeddings, consists in treating both the case of \emph{definite} quaternion algebras and the case of \emph{indefinite} quaternion algebras in a uniform way. We prove results on the size of Nekov\'a\v{r}'s extended Selmer groups attached to suitable big Galois representations and we formulate two-variable Iwasawa main conjectures both in the definite case and in the indefinite case. Moreover, in the definite case we propose refined conjectures \emph{{\`a} la} Greenberg on the vanishing at the critical points of (twists of) the $L$-functions of the modular forms in the Hida family of $f$ living on the same branch as $f$.
\end{abstract}

\maketitle

\tableofcontents

\section{Introduction}

The purpose of this work is to extend Howard's results on the variation of Heegner points in Hida families of modular forms (\cite{ho}) to a general quaternionic setting. Analogues of the constructions by Howard of systems of big Heegner points on towers of classical modular curves have been proposed by Fouquet (\cite{fouquet}, \cite{fouquet2}) for Shimura curves attached to \emph{indefinite} quaternion algebras over totally real number fields; on the contrary, the case where modular curves need to be replaced by Shimura curves coming from \emph{definite} quaternion algebras has never been investigated. However, the philosophy behind the so-called ``parity conjectures'' suggests that the definite and indefinite cases are equally significant from an arithmetic point of view, so it would be desirable to have both sides of the quaternionic setting well understood and developed. With this goal in mind, in this article we offer a systematic construction of big Heegner points and classes attached to Hida families which treats simultaneously both the definite case and the indefinite case over $\Q$, and we study the arithmetic of the relevant extended Selmer groups as defined by Nekov\'a\v{r}. Now let us describe the subject of the paper more in detail.

Fix an integer $N$, a prime $p\nmid 6N$ and an ordinary $p$-stabilized newform
\[ f(q)=\sum_{n=1}^\infty a_nq^n\in S_k\bigl(\Gamma_0(Np),\omega^j\bigr) \]
where $\omega$ is the Teichm\"uller character and $j\equiv k\pmod{2}$. Let $F$ be a finite extension of $\Q_p$ containing all the eigenvalues of the Hecke operators acting on $f$ and let $\cl O_F$ denote its ring of integers. Assume also that the residual representation attached to $f$ is absolutely irreducible.

Fix an imaginary quadratic field $K$ of discriminant prime to $Np$ and consider the factorization $N=N^+N^-$ induced by $K$: a prime number $\ell$ divides $N^+$ (respectively, $N^-$) if and only if $\ell$ splits (respectively, is inert) in $K$. Assume throughout that $N^-$ is square-free and say that we are in the \emph{definite} (respectively, \emph{indefinite}) case if the number of primes dividing $N^-$ is odd (respectively, even). For simplicity, in this introduction we suppose that $p$ does not divide the class number of $K$.

Hida's theory (\cite{h-iwasawa}, \cite{h-galois}) incorporates the modular form $f$ and the $p$-adic Galois representation $\rho_f:G_\Q:=\Gal(\bar\Q/\Q)\rightarrow\GL_2(F)$ attached to $f$ by Deligne into an analytic family of modular forms and Galois representations. More precisely, Hida defines the \emph{universal ordinary Hecke algebra} $\fr h_\infty$ by taking the inverse limit over $m$ of the (classical) Hecke algebras $\fr h_m$ over $\cl O_F$ acting on weight $2$ cusp forms with coefficients in $\cl O_F$ of level $\Gamma_1(Np^m)$ and then projecting to the ordinary part. Out of $\fr h_\infty$ one then constructs a local domain $\cl R$, finite and flat over the Iwasawa algebra $\Lambda:=\cl O_F[\![1+p\Z_p]\!]$, such that certain prime ideals $\fr p$ of $\cl R$ (called \emph{arithmetic}) correspond to modular forms $f_\fr p$ of suitable weight $k_\fr p$, level $\Gamma_1(Np^{m_\fr p})$ and character $\psi_\fr p$ with coefficients in the residue field $F_\fr p$ of the localization of $\cl R$ at $\fr p$; moreover, $f_{\bar{\fr p}}=f$ for a certain arithmetic prime $\bar{\fr p}$ of weight $k$. Finally, taking inverse limits over $m$ of the $p$-adic Tate modules of the Jacobian varieties of the modular curves $X_1(Np^m)$ one can introduce a $G_\Q$-representation ${\bf T}$ which is free of rank two over $\cl R$ and has the property that $V_\fr p:={\bf T}\otimes_{\cl R}F_\fr p$ is a twist of the representation $V(f_\fr p)$ associated with $f_\fr p$.

\subsection{Big Selmer groups}

In recent years, the systematic study of certain Selmer groups attached to the $G_\Q$-representation ${\bf T}$ has been pursued, among others, by Nekov\'a\v{r} and Plater (\cite{np}), Nekov\'a\v{r} (\cite{Ne-selmer}), Ochiai (\cite{ochiai}), Howard (\cite{ho}) and Delbourgo (\cite{delbourgo}). More precisely, the $G_\Q$-representation ${\bf T}$ admits a twist ${\bf T}^\dagger$ which has a perfect alternating pairing ${\bf T}^\dagger\times{\bf T^\dagger}\rightarrow \cl R(1)$, and for every arithmetic prime $\fr p$ of $\cl R$ the representation $V_\fr p^\dagger:={\bf T}^\dagger\otimes_\cl RF_\fr p$ is a self-dual twist of $V(f_\fr p)$. Then, using Nekov\'a\v{r}'s theory of Selmer complexes (\cite{Ne-selmer}), for any number field $L$ one can define extended Selmer groups $\wt H^1_f\bigl(L,{\bf T}^\dagger\bigr)$ and $\wt H^1_f\bigl(L,V_\fr p^\dagger\bigr)$, whose arithmetic is the main theme of the present paper.

Now we briefly sketch the work of Howard which was the original inspiration for our article. In order to study the arithmetic of Nekov\'a\v{r}'s Selmer groups, when all primes $\ell|N$ split in $K$ (i.e., when $N^-=1$) Howard introduced in \cite{ho} canonical cohomology classes
\[\fr X_c\in\wt H^1_f\bigl(H_c,{\bf T}^\dagger\bigr), \]
which he calls ``big Heegner points'', where $c\geq1$ is an integer prime to $N$ and $H_c$ is the ring class field of $K$ of conductor $c$. These classes are constructed by taking an inverse limit of cohomology classes arising from Heegner points in the Jacobians of classical modular curves via Kummer maps, and satisfy suitable Euler system relations in the sense of Kolyvagin: see \cite[\S\S 2.2--2.4]{ho}. These objects are used to obtain various results on the arithmetic of the above-mentioned Selmer groups; in particular, a vertical nonvanishing theorem (generalizing results of Cornut and Vatsal in \cite{cv}) is proved in \cite[\S 3.1 and \S 3.2]{ho}, while an horizontal nonvanishing conjecture is formulated in \cite[\S 3.4]{ho}. Moreover, in \cite[Conjecture 3.3.1]{ho} Howard proposes a two-variable Iwasawa main conjecture for $\wt H^1_{f,{\rm Iw}}\bigl(K_\infty,{\bf T}^\dagger\bigr)$ which extends the Heegner point main conjecture formulated by Perrin-Riou in \cite{p-r}. Here
\[ \wt H^1_{f,{\rm Iw}}\bigl(K_\infty,{\bf T}^\dagger\bigr):=\invlim\wt H^1_f\bigl(K_n,{\bf T}^\dagger\bigr) \]
is a module over the Iwasawa algebra $\cl R_\infty:=\mathcal R[\![G_\infty]\!]$ attached to the Galois group $G_\infty$ of the anticyclotomic $\Z_p$-extension $K_\infty$ of $K$ as described in \cite[\S 3.3]{ho}, and $K_n$ is the $n$-th layer of $K_\infty$, i.e. the subfield of $K_\infty$ such that $\Gal(K_n/K)\simeq\Z/p^n\Z$.

In this paper we are interested in results and conjectures of the type described above in the more general case where one allows for the existence of primes dividing $N$ which are inert in $K$. In other words, the integer $N^-$ is not necessarily equal to $1$. In the \emph{indefinite} case (i.e., when the number of primes dividing $N^-$ is even) our constructions and results should be compared with those obtained by Fouquet in \cite{fouquet} and \cite{fouquet2} for Shimura curves over totally real fields; on the contrary, as far as we know the \emph{definite} case (corresponding to an odd number of primes dividing $N^-$) is considered here for the first time. The ability of performing constructions which apply equally well to both the two cases is the most significant novelty in our approach, and we hope that this represents a first step towards the development in a Hida context of a theory of Bertolini--Darmon type (\cite{bd*}, \cite{BD-annals-146}, \cite{BD}, \cite{BD-inv2}), where the interplay between definite and indefinite settings (manifesting itself, for example, via \v{C}erednik's interchange of invariants, congruences between special values, explicit reciprocity laws) plays a crucial role for studying the arithmetic of modular forms.

In the rest of the introduction we give a brief description of the paper, referring to the main body of the text for all details.

\subsection{Families of optimal embeddings on Shimura curves}

Let $B$ denote the quaternion algebra over $\Q$ of discriminant $N^-$ (thus $B$ is split at the archimedean place $\infty$ of $\Q$ in the indefinite case and is ramified at $\infty$ in the definite case) and for every integer $m\geq 0$ choose an Eichler order $R_m$ of $B$ of level $N^+p^m$ such that $R_m\subset R_{m-1}$ for all $m\geq 1$. If the hat denotes adelizations, one then defines open compact subgroups $U_m\subset\widehat R_m^\times$ by imposing an extra $\Gamma_1(p^m)$-level structure on $\widehat R_m$ and considers the Shimura curves $\wt X_m$ associated with $U_m$ (precise definitions in terms of double cosets are given in \S \ref{section-definite-shimura} and \S \ref{section-indefinite-shimura}). In the definite case these are disjoint unions of genus $0$ curves defined over $\Q$, while in the indefinite case they are compact Riemann surfaces admitting canonical models over $\Q$. For any integer $c\geq1$ prime to $N$ and the discriminant of $K$ we define the Heegner points of conductor $c$ on $\wt X_m$ as those pairs $[(g,f)]$ in the subset
\[ \wt X_m^{(K)}:=U_m\big\backslash\bigl(\widehat B^\times\times\Hom(K,B)\bigr)\big/B^\times\subset\wt X_m(\C) \]
such that $f$ is an {optimal} embedding of the order $\cl O_c$ of $K$ of conductor $c$ into the Eichler order $B\cap(g^{-1}\widehat R_mg)$ of $B$ and the local component $f_p$ of $f$
takes optimally the elements of $\mathcal O_c\otimes\Z_p$ congruent to 1 modulo $p^m\mathcal O_K\otimes\Z_p$ to
the local component of $U_m$ at $p$
(see Definition \ref{def-Heegner-points} for a more precise statement). In \S \ref{section-fields-of-rationality} we prove that in the indefinite case these Heegner points are rational over $H_c(\boldsymbol\mu_{p^m})$, where $H_c$ is the ring class field of $K$ of conductor $c$ and $\boldsymbol\mu_{p^m}$ is the group of $p^m$-th roots of unity. If $\fr a\in\widehat K^\times$ and $\hat f:\widehat K^\times\rightarrow\widehat B^\times$ is the adelization of $f$, in both the definite and the indefinite cases the map
\[ [(g,f)]\longmapsto\bigl[\bigl(g\hat f(\fr a),f\bigr)\bigr] \]
induces a (free) action of $\Gal(H_c({\boldsymbol \mu}_{p^m})/K)$ on the set of Heegner points of conductor $c$. Furthermore, the group $\Div\bigl(\wt X_m\bigr)$ of divisors on $\wt X_m$ is endowed with an action of the usual Hecke operators $T_\ell$ for primes $\ell\nmid Np^m$ and $U_\ell$ for primes $\ell|Np^m$ and of diamond operators $\langle\ell\rangle$ for $\ell\in(\Z/p^m\Z)^\times$. In Section \ref{section-explicit-construction} we provide an explicit construction of suitably compatible families of Heegner points on our tower of Shimura curves. The main features of our system of points are summarized by the following

\begin{teo} \label{thm-A}
For every integer $m\geq0$ and every integer $c\geq1$ prime to $N$ and the discriminant of $K$ there is a Heegner point $\wt P_{c,m}\in\wt X_m^{(K)}$ of conductor $cp^m$, rational over $H_{cp^m}({\boldsymbol \mu}_{p^m})$ in the indefinite case, such that the following conditions are satisfied.
\begin{enumerate}
\item\emph{Vertical compatibility.} If $m\geq2$ then the equality
\[U_p\bigl(\wt P_{c,m-1}\bigr)=\wt\alpha_{m,\ast}\bigl(\mathrm{tr}_{H_{cp^m}(\boldsymbol\mu_{p^m})/H_{cp^{m-1}}(\boldsymbol\mu_{p^m})}(\wt P_{c,m})\bigr) \]
holds in $\Div\bigl(\wt X_{m-1}\bigr)$, where $\wt\alpha_{m,\ast}$ is obtained from the covering map
$\wt\alpha_m:\wt X_m\rightarrow \wt X_{m-1}$.
\item\emph{Horizontal compatibility.} Let $m\geq1$ and $n\geq1$ be integers. Then the equality
\[ U_p\bigl(\wt P_{cp^{n-1},m}\bigr)=\mathrm{tr}_{H_{cp^{m+n}}(\boldsymbol\mu_{p^{m+n}})/H_{cp^{m+n-1}}(\boldsymbol\mu_{p^{m+n}})}\bigl(\wt P_{cp^n,m}\bigr) \]
holds in $\Div\bigl(\wt X_m\bigr)$. Furthermore, assuming $\mathcal O_{cp^m}^\times=\{\pm 1\}$, for primes $\ell\nmid cNp$ which are inert in $K$ one has
\[ T_\ell\bigl(\wt P_{c,m}\bigr)=\mathrm{tr}_{H_{c\ell p^m}(\boldsymbol{\mu}_{p^m})/H_{cp^m}(\boldsymbol{\mu}_{p^m})}\bigl(\wt P_{c\ell,m}\bigr). \]
\item \emph{Galois compatibility.} Set $p^\ast:=(-1)^{(p-1)/2}p$, let $\epsilon_\cyc:G_\Q\rightarrow\Z_p^\times$ be the $p$-adic cyclotomic character and let $\vartheta:\Gal\bigl(\bar\Q/\Q(\sqrt{p^\ast})\bigr)\rightarrow\Z_p^\times/\{\pm1\}$
be the unique continuous homomorphism such that $\vartheta^2$ coincides with the restriction of
$\epsilon _\cyc$. Then for all $\sigma\in\Gal(\bar\Q/H_{cp^m})$ the equality
\[ \wt P_{c,m}^\sigma=\langle\vartheta(\sigma)\rangle\wt P_{c,m} \]
holds in $\Div\bigl(\wt X_m\bigr)$.
\end{enumerate}
\end{teo}

\begin{proof} Part (1) is Proposition \ref{Hecke-relation-for-X}, part (2) is Proposition \ref{horizontal-hecke-X} plus Proposition \ref{T-operator} and part (3) is equality \eqref{eq5}. \end{proof}

The existence of compatible sequences of CM points as in Theorem \ref{thm-A} was shown by Howard in \cite{ho} when the $\wt X_m$ are classical modular curves (using the interpretation of modular curves as moduli spaces for elliptic curves with suitable level structures) and by Fouquet in \cite{fouquet} for Shimura curves attached to indefinite quaternion algebras over totally real fields having exactly one split archimedean place.

As remarked before, in this paper our families of CM points are introduced via a systematic use of the theory of optimal embeddings as described in \cite{gross-2} and \cite{bd*}, and this approach (although technically more intricate than those of Howard and Fouquet) has the advantage of offering a uniform setting for dealing with both the definite case and the indefinite case, as in \cite{bd*}. The importance of dealing with the definite case as well when studying the representation associated with a Hida family stems from the fact that, conjecturally, the indefinite case should take care of situations in which the rank of the Selmer group is \emph{odd}, while the definite case should describe \emph{even} rank settings. Observe, moreover, that the definite case cannot be treated by means of the tools developed in \cite{ho}, \cite{fouquet} and \cite{fouquet2}.

We use these families to construct big Heegner points and classes that are the counterparts of those defined in \cite{ho}, and then we prove results and formulate conjectures which generalize those obtained in \emph{loc. cit.} by Howard. In the rest of the introduction we focus our attention on the main results obtained in our work. For clarity of exposition, it will be convenient to treat the definite case and the indefinite case separately.

\subsection{The definite case}

As already observed, an adequate setting for dealing with arbitrary quaternion algebras represents the newest contribution of the paper.

The literature on the arithmetic of (extended) Selmer groups attached to modular forms of \emph{arbitrary} weight which are associated with forms on definite quaternion algebras via the Jacquet--Langlands correspondence is not so vast as that on the indefinite case (for instance, no analogue in rank $0$ of the results of Nekov\'a\v{r} in \cite{nk-inv} is available); as a consequence, the applications of our Theorem \ref{thm-A} we can presently offer are either conjectural or conditional.

Let $w\in\{\pm1\}$ be the common root number of the $L$-functions of the twisted forms $f_\fr p^\dagger$ (see \S \ref{section-bounding-def} for the definition) for all but finitely many arithmetic primes $\fr p$ of $\cl R$, and for every arithmetic prime $\fr p$ let $F_\fr p$ be the residue field of the localization of $\cl R$ at $\fr p$. It is expected that for almost all arithmetic primes $\fr p$ of $\cl R$ the dimension over $F_\fr p$ of $\wt H^1_f\bigl(K,V_\fr p^\dagger\bigr)$ is $0$ if $w=1$ and is $2$ if $w=-1$. Let us focus our attention on the case $w=1$. We consider the Hecke modules
\begin{equation} \label{eq-intro-2}
J_m:=\cl O_F\bigl[U_m\backslash\widehat B^\times/B^\times\bigr]\simeq\Pic\bigl(\wt X_m\bigr)\otimes_\Z\cl O_F
\end{equation}
and define the inverse limit $J_\infty:=\invlim J_m$ with respect to the canonical projection maps. By the Jacquet--Langlands correspondence, the ordinary part $J_\infty^\ord$ of this $\cl O_F$-module is endowed with an action of the $N^-$-new quotient $\T_\infty$ of the universal ordinary Hecke algebra $\fr h_\infty^\ord$. One can then introduce the finitely generated $\cl R$-module ${\bf J}:=J_\infty^\ord\otimes_{\T_\infty}\cl R$. Under a reasonable hypothesis (Assumption \ref{ass-def}), we prove that ${\bf J}$ is free of rank one over $\cl R$ and fix an isomorphism
\begin{equation} \label{eq-intro-1}
{\bf J}\simeq\cl R.
\end{equation}
The compatible sequence of Heegner points on the tower of definite Shimura curves whose existence is guaranteed by Theorem \ref{thm-A} can be combined with isomorphisms \eqref{eq-intro-2} to produce canonical elements in ${\bf J}$. By isomorphism \eqref{eq-intro-1} one then obtains an element $\cl J_0\in\cl R$. In light of the conjectural formulas for the dimension of Selmer groups that we recalled above, we predict that if $w=1$ then $\cl J_0\not=0$ (Conjecture \ref{conj-def-1}). In fact, the element $\cl J_0$ is the counterpart in our context of the divisor introduced by Gross in \cite[\S 11]{gross-2}, hence it is naturally expected to be related to the special values of the $L$-functions over $K$ of the forms $f_\fr p^\dagger$. When $w=1$ the functional equations suggest that these special values are non-zero for almost all arithmetic primes $\fr p$, whence the (conjectural) non-triviality of $\cl J_0$. We elaborate on this circle of ideas in \S \ref{section-vanishing}, where we propose general conjectures on the vanishing of the special values of twists of the (classical) $L$-functions over $K$ of the modular forms living on the Hida branch of $f$. These conjectures can be viewed as a refinement of the conjectures on the generic analytic rank of the forms in the Hida family made by Greenberg in \cite{greenberg}.

For every arithmetic prime $\fr p$ of $\cl R$ let $\pi_\fr p:\cl R\rightarrow F_\fr p$ be the canonical map. The following is (part of) Conjecture \ref{conj-def}.

\begin{conj} \label{conj-B}
Suppose that $w=1$. If $\pi_\fr p(\cl J_0)\not=0$ then $\wt H^1_f\bigl(K,V_\fr p^\dagger\bigr)=0$.
\end{conj}

We expect that Conjecture \ref{conj-B} can be proved (at least for arithmetic primes of weight $2$) by suitably extending the arguments of \cite{BD} and \cite{lv} to the case of forms with non-trivial character. In general, we can say that this conjecture plays in our definite setting a similar role to one of Nekov\'a\v{r}'s theorems (\cite{nk-inv}) in the context considered by Howard. More precisely, just as Nekov\'a\v{r}'s work extends the classical results of Kolyvagin (see, e.g., \cite{gross}), Conjecture \ref{conj-B} should be viewed as a generalization of (a portion of) the theory of Bertolini and Darmon (\cite{BD-annals-146}, \cite{BD}) to modular forms of higher weight. In this direction, see work in progress by Chida (\cite{chida1}, \cite{chida2}). Now Theorem \ref{teo-def} can be stated as follows.
\begin{teo} \label{thm-C}
Suppose that $w=1$ and assume Conjecture \ref{conj-B}. If $\cl J_0\not=0$ then $\wt H^1_f\bigl(K,{\bf  T}^\dagger\bigr)$ is a torsion $\cl R$-module.
\end{teo}

Another application of Theorem \ref{thm-A} is to the formulation of a conjecture in the Iwasawa theory of our Hida family. Set $G_n:=\Gal(K_n/K)$ for all integers $n\geq1$; essentially by corestricting to the finite layers of $K_\infty$, one also gets elements $\cl Q_n^\sigma\in\cl R$ for all $\sigma\in G_n$ and all $n\geq1$. The compatibility properties of these points allow us to define
\[ \theta_n:=\alpha_p^{-n}\sum_{\sigma\in G_n}\cl Q_{n}^\sigma\otimes\sigma^{-1}\in\cl R[G_n],\qquad\theta_\infty:=\invlim_n\theta_n\in\cl R_\infty \]
where $\alpha_p\in\cl R^\times$ is the image of the Hecke operator $U_p$ under the natural map $\fr h_\infty^\ord\rightarrow\cl R$. As in \S \ref{section-iwasawa-def}, one introduces an $\cl R_\infty$-module $\wt H^1_{f,{\rm Iw}}\bigl(K_\infty,{\bf A}^\dagger\bigr)$ where ${\bf A}^\dagger:=\Hom\bigl({\bf T}^\dagger,{\boldsymbol \mu}_{p^\infty}\bigr)$. Finally, write $x\mapsto x^*$ for the involution of $\cl R_\infty$ given by $\sigma\mapsto\sigma^{-1}$ on group-like elements. The following statement (which is Conjecture \ref{main-def-conj}) must be seen as a main conjecture of Iwasawa theory in the definite setting.
\begin{conj} \label{conj-D}
Assume that the local ring $\cl R$ is regular. The group $\wt H^1_{f,{\rm Iw}}\bigl(K_\infty,{\bf A}^\dagger\bigr)$ is a finitely generated torsion module over $\cl R_\infty$ and there is an equality
\[ \bigl(\theta_\infty\cdot\theta_\infty^*\bigr)={\rm Char}_{\cl R_\infty}\Big(\wt H^1_{f,{\rm Iw}}\bigl(K_\infty,{\bf A}^\dagger\bigr)^\vee\Big) \]
of ideals of $\cl R_\infty$.
\end{conj}
Here the symbol ${}^\vee$ denotes the Pontryagin dual and the product $\theta_\infty\cdot\theta_\infty^*$ is interpreted as a $p$-adic $L$-function. Note that these definitions are reminiscent of the constructions performed by Bertolini and Darmon in, e.g., \cite{bd*} and \cite{BD}.

\subsection{The indefinite case}

This is the direct generalization of the classical modular curves setting originally studied by Howard in \cite{ho}, and has also been considered, along a different line of investigation, in \cite{fouquet} and \cite{fouquet2} by Fouquet (who works in the broader context of Shimura curves attached to indefinite quaternion algebras over totally real fields). The reader is suggested to compare our approach to Fouquet's, since the goals and results of his work and of ours are of different (and, in many respects, complementary) natures.

By taking the inverse limit of the $p$-adic Tate modules of the Jacobian varieties of $\wt X_m$ as in \cite{h-galois}, we construct a $G_\Q$-representation ${\bf T}_{\rm Sh}$ which is free of rank two over $\cl R$, and prove isomorphisms of $G_\Q$-modules ${\bf T}\simeq{\bf T}_{\rm Sh}$ and ${\bf T}^\dagger\simeq{\bf T}_{\rm Sh}^\dagger$. Following \cite{ho}, the compatible sequence of Heegner points of Theorem \ref{thm-A} can then be used to define cohomology classes $\kappa_c\in H^1\bigl(H_c,{\bf T}^\dagger\bigr)$. In this general quaternionic setting, the problem of showing that these classes belong to Nekov\'a\v{r}'s Selmer group presents extra complications. More precisely, due to the possible presence of primes dividing $N$ which are inert in $K$ and so split completely in $H_c/K$ for all $c$ (prime to $N$), we are only able to show that $\lambda\cdot\kappa_c\in\wt H^1_f\bigl(H_c,{\bf T}^\dagger\bigr)$ for any choice of $\lambda\in\cl R$ in the annihilator of the $\cl R$-torsion module $\prod_{\ell\mid N^-}H^1\bigl(K_\ell,{\bf T}^\dagger\bigr)_{\rm tors}$. We fix once and for all a non-zero $\lambda$ in this annihilator and define, in analogy with \cite{ho}, the classes
\[ \fr X_c:=\lambda\cdot\kappa_c\in\wt H^1_f\bigl(H_c,{\bf T}^\dagger\bigr),\qquad\fr Z_0:=\cor_{H_1/K}(\fr X_1)\in\wt H^1_f\bigl(K,{\bf T}^\dagger\bigr). \]
The following two results generalize theorems of Howard (for the definition of ``non-exceptional primes'', in the sense of Mazur--Tate--Teitelbaum, see \S \ref{section-Selmer-groups}).
\begin{teo} \label{thm-E}
Let $\fr p$ be a non-exceptional arithmetic prime of $\cl R$ with trivial character and even weight. If $\fr Z_0$ has non-trivial image in $\wt H^1_f\bigl(K,V_\fr p^\dagger\bigr)$ then $\dim_{F_\fr p}\wt H^1_f\bigl(K,V_\fr p^\dagger\bigr)=1$.
\end{teo}
This is Theorem \ref{teo-final-indefinite} in the text.
\begin{teo} \label{thm-F}
If $\fr Z_0$ is not $\cl R$-torsion then $\wt H^1_f\bigl(K,{\bf T}^\dagger\bigr)$ is an $\cl R$-module of rank one.
\end{teo}
We prove this statement in Theorem \ref{teo-final-indefinite2}, and we expect the condition on the class $\fr Z_0$ to be always true. We finally remark that a main conjecture of Iwasawa theory is proposed in Conjecture \ref{main-indef-conj}; this can be viewed as the counterpart of Conjecture \ref{conj-D} in the indefinite setting and extends both the conjecture \cite[Conjecture 3.3.1]{ho} of Howard and the classical Heegner point main conjecture for elliptic curves formulated by Perrin-Riou in \cite{p-r}.

\subsubsection*{Acknowledgements} We would like to thank Massimo Bertolini, Ga\"{e}tan Chenevier, Haruzo Hida, Ben Howard and Jan Nekov\'a\v{r} for helpful discussions and correspondence on some of the topics of this paper. We are also grateful to Olivier Fouquet for his interest in our work and his helpful remarks. Last but not least, we especially wish to express our gratitude to the anonymous referee for the extremely careful reading of earlier versions of this article: his or her constructive criticism and several valuable suggestions led us to correct a few mistakes, completely rethink some parts and improve the overall exposition in a significant way. 

\section{Towers of Shimura curves} \label{towers-sec}

For any ring $A$ denote by $\widehat A:=A\otimes_\Z\prod_{\ell}\Z_\ell$ its profinite completion, where the product is over all prime numbers $\ell$, by $A_\ell:=A\otimes\Z_\ell$ its $\ell$-adic completion at a prime number $\ell$ and by $A_\infty:=A\otimes \R$ its archimedean completion. An element $x\in\widehat A$ is denoted by $(x_\ell)_\ell$.

Let $N^-$ be a positive square free integer and $N^+$ a positive integer prime to $N^-$. Define
\[ N:=N^+N^- \]
and let $p\nmid N$ be an odd prime number. Denote by $B$ the (unique, up to isomorphism) quaternion algebra over $\Q$ of discriminant $N^-$. If the number of primes dividing $N^-$ is odd (respectively, even) then $B$ is definite (respectively, indefinite), that is, $B_\infty$ is isomorphic to the Hamilton skew field (respectively, to the matrix algebra $\rom{M}_2(\R)$). Fix once and for all an isomorphism
\[ \phi_p:B_p\overset\simeq\longrightarrow\rom{M}_2(\Q_p) \]
of $\Q_p$-algebras. Moreover, for every integer $m\geq0$ let $R_m\subset B$ be an Eichler order of level $N^+p^m$ such that $R_{j+1}\subset R_j$ for all $j\geq0$ and
\[ \phi_p(R_m\otimes\Z_p)=\left\{\begin{pmatrix}a&b\\c&d\end{pmatrix}\in\M_2(\Z_p)\:\Big\vert\:c\equiv0\pmod{p^m}\right\}. \]
Finally, for all $m\geq 0$ let $U_m\subset \widehat R_m^\times$ be the subgroup of elements $(x_\ell)_\ell$ with $\phi_p(x_p)\equiv\bigl(\begin{smallmatrix}1&b\\0&d\end{smallmatrix}\bigr)\pmod{p^m}$ for some $b\in\Z_p$ and some $d\in\Z_p^\times$.

\bigskip

\noindent\emph{Convention.} In order not to burden the notation, in the rest of the paper we will often identify $B_p$ with $\mathrm{M}_2(\Q_p)$ via the isomorphism $\phi_p$ -- we will do so according to convenience, without explicit warning. Thus the reader should always bear in mind that when we write, for example, ``the adele $b\in\widehat B$ has $p$-component $b_p$ equal to $\smallmat\alpha\beta\gamma\delta\in\mathrm{M}_2(\Q_p)$'' we really mean that $b_p$ is equal to $\phi_p^{-1}\bigl(\smallmat\alpha\beta\gamma\delta\bigr)$.

\subsection{Definite Shimura curves} \label{section-definite-shimura}

Let $B$ be \emph{definite}. Denote by $\mathbb P=\mathbb P_{N^-}$ the curve of genus $0$ defined over $\Q$ by setting
\[ \mathbb P(A):=\bigl\{x\in B\otimes_\Q A\mid x\not=0,\;\rom{Norm}(x)=\rom{Trace}(x)=0\bigr\}\big/A^\times \]
for any $\Q$-algebra $A$, where $\Norm$ and $\rom{Trace}$ are the reduced norm and trace of $B\otimes_\Q A$. The group $B^\times$ acts on $\mathbb P$ by conjugation and this action is algebraic and defined over $\Q$. Note that $\mathbb P(\C)$ is canonically identified with $\Hom_\R(\C,B_\infty)$, where $\Hom_\R$ denotes homomorphisms of $\R$-algebras. Explicitly, if $z\mapsto\bar z$ denotes complex conjugation then with each embedding $f:\C\rightarrow B_\infty$ one associates the image $x_f$ of the unique $\C$-line on the quadric $\bigl\{x\in B\otimes\C\mid\rom{Norm}(x)=\rom{Trace}(x)=0\}$ on which $f(\C^\times)$ acts via the character $y\mapsto\bar y/y$. Observe that $x_f$ is one of the two fixed points of $f(\C^\times)$ acting on $\mathbb P(\C)$. In fact, this recipe allows one to identify $\mathbb P(K)$ with $\Hom_\Q(K,B)$ for any imaginary quadratic field $K$ (cf. \cite[p. 131]{gross-2}). Define the \emph{definite Shimura curve of level $R_m$} (respectively, $U_m$) \emph{and discriminant $N^-$} to be the double coset space
\[ X_m:=\widehat R^\times_m\backslash(\widehat B^\times\times\mathbb P)/B^\times\qquad\text{(respectively, $\widetilde X_m:=U_m\backslash(\widehat B^\times\times\mathbb P)/B^\times$)}, \]
where $\widehat R^\times_m$ and $U_m$ act by left multiplication on $\widehat B^\times$ and trivially on $\mathbb P$, while $B^\times$ acts by conjugation on $\mathbb P$ and by right multiplication on $\widehat B^\times$.

If $K$ is an imaginary quadratic field write
\[ X_m^{(K)}:=\widehat R^\times_m\big\backslash\bigl(\widehat B^\times\times\mathbb P(K)\bigr)\big/B^\times,\qquad\widetilde X_m^{(K)}:=U_m\big\backslash\bigl(\widehat B^\times\times\mathbb P(K)\bigr)\big/B^\times. \]
As remarked in \cite[p. 131]{gross-2}, $X_m^{(K)}=X_m(K)$ and $\wt X_m^{(K)}=\wt X_m(K)$. However, in the following we use the above symbols in order to keep our notation uniform with the one adopted in the indefinite case (see below), where the points in $X_m^{(K)}$ or $\wt X_m^{(K)}$ are in general rational only over (abelian) extensions of $K$.

Choose representatives $g_1,\dots,g_{h(m)}$ and $\tilde g_1,\dots,\tilde g_{\tilde h(m)}$ of the double cosets $\widehat R^\times_m\backslash \widehat B^\times/B^\times$ and $U_m\backslash\widehat B^\times/B^\times$, respectively. Define the finite arithmetic groups
\[ \Gamma_m^i:=g_i^{-1}\widehat R^\times_mg_i\cap B^\times,\qquad \wt\Gamma_m^j:=\tilde g_j^{-1}U_m\tilde g_j\cap B^\times \]
with $i\in\{1,\dots,h(m)\}$ and $j\in\{1,\dots,\tilde h(m)\}$. Then $X_m$ and $\wt X_m$ can be expressed as disjoint unions
\[ X_m=\coprod_{i=1}^{h(m)}\mathbb P/\Gamma_m^i,\qquad\wt X_m=\coprod_{i=1}^{\tilde h(m)}\mathbb P/\wt\Gamma_m^i \]
of curves of genus $0$.

\subsection{Indefinite Shimura curves} \label{section-indefinite-shimura}

Let $B$ be \emph{indefinite}. In this case, for all $m\geq0$ both $\widehat R^\times_m\backslash\widehat B^\times/B^\times$ and $U_m\backslash\widehat B^\times/B^\times$ consist of a single class. Fix an isomorphism $\phi_\infty:B_\infty\overset{\simeq}{\longrightarrow}\rom{M}_2(\R)$; then $\phi_\infty(R_m^\times)$ is a discrete subgroup of $\GL_2(\R)$. Denote by $\Gamma_m$ the subgroup of $\phi_\infty(R^\times_m)$ consisting of matrices with determinant $1$ and let $\widetilde\Gamma_m$ be the analogous subgroup of $\phi_\infty(U_m\cap B)$. Define the Riemann surfaces
\[ Y_m(\C):=\cl H/\Gamma_m, \qquad \widetilde Y_m:=\cl H/\widetilde \Gamma_m \]
where $\cl H$ is the complex upper half plane and the groups $\Gamma_m$ and $\widetilde \Gamma_m$ act on $\cl H$ by M\"obius
(i.e., fractional linear) transformations. Let $X_m(\C)$ (respectively, $\widetilde X_m(\C)$) denote the Baily--Borel compactification of $Y_m(\C)$ (respectively, $\widetilde Y_m(\C)$). If $B\not=\rom M_2(\Q)$ then $X_m(\C)=Y_m(\C)$ and $\widetilde X_m(\C)=\widetilde Y_m(\C)$. The Riemann surface $X_m(\C)$ (respectively, $\widetilde X_m(\C)$) has a model over $\Q$ which will be denoted by $X_m$ (respectively, $\widetilde X_m$) and referred to as the \emph{indefinite Shimura curve of level $R_m$} (respectively, $U_m$) \emph{and discriminant $N^-$}. Setting $\mathbb P:=\C-\R$ for the union of the complex upper and lower half-planes yields
\[ Y_m(\C)=\widehat R^\times_m\backslash(\widehat B^\times\times\mathbb P)/B^\times,\qquad \widetilde Y_m(\C)=U_m\backslash(\widehat B^\times \times\mathbb P)/B^\times \]
where, as above, $\widehat R^\times_m$ and $U_m$ act by left multiplication on $\widehat B^\times$ and trivially on $\mathbb P$, while $B^\times$ acts by M\"{o}bius transformations via $\phi_\infty$ on $\mathbb P$ and by right multiplication on $\widehat B^\times$. Observe that there is a $B^\times$-equivariant identification $\mathbb P=\Hom_\R(\C,B_\infty)$ (here $B^\times$ acts on the homomorphisms by conjugation): similarly to the definite case, with an embedding $f:\C\rightarrow B_\infty$ we associate the unique fixed point of $f(\C^\times)$ lying in the upper half-plane, i.e., the fixed point $x_f$ such that the induced action of $f(\C^\times)$ on the cotangent space of $\mathbb P$ at $x_f$ is via the character $y\mapsto\bar y/y$. 
For any imaginary quadratic field $K$ fix an embedding $K\hookrightarrow\C$; so there are injections
\[ X_m^{(K)}:=\widehat R^\times_m\big\backslash\bigl(\widehat B^\times\times\Hom_\Q(K,B)\bigr)\big/B^\times\; \longmono\; X_m(\C), \]
\[ \wt X_m^{(K)}:=U_m\big\backslash\bigl(\widehat B^\times\times\Hom_\Q(K,B)\bigr)\big/B^\times\; \longmono\;\widetilde X_m(\C) \]
induced by the map $\Hom_\Q(K,B)\rightarrow\Hom_\R(\C,B_\infty)$ which is obtained by extending scalars from $\Q$ to $\R$. Actually, the subsets $X_m^{(K)}$ and $\wt X_m^{(K)}$ are contained in $X_m(\bar\Q)$ and $\wt X_m(\bar\Q)$, respectively, where $\bar\Q$ is the algebraic closure of $\Q$ in $\C$.

As a piece of notation, both in the definite case and in the indefinite case write $\Div(X_m)$ and $\Div(\wt X_m)$ for the groups of divisors on the Riemann surfaces $X_m(\C)$ and $\wt X_m(\C)$, respectively.

\subsection{The tower of curves} \label{tower-subsubsec}

The inclusions $R_{m+1}\subset R_m$, $U_{m+1}\subset U_m$ and $U_m\subset R_m$ for $m\geq 0$ yield a commutative diagram of curves
\begin{equation} \label{tower}
\xymatrix@C=40pt{\dots\ar[r]^-{\widetilde\alpha_{m+1}}&\widetilde X_m\ar[r]^-{\widetilde\alpha_m}\ar[d]^-{\beta_m}&\widetilde X_{m-1} \ar[r]^-{\widetilde\alpha_{m-1}} \ar[d]^-{\beta_{m-1}}&\dots\\
\dots\ar[r]^-{\alpha_{m+1}}&X_m\ar[r]^-{\alpha_m}&X_{m-1}\ar[r]^-{\alpha_{m-1}}&\dots}
\end{equation}
in which all maps are finite coverings that are defined over $\Q$.

\subsection{Hecke operators} \label{section-Hecke-operators}

We briefly review the standard description of the Hecke operators $T_\ell$ and $U_p$ in the case of our interest. Let $m\geq0$ be an integer and let $\ell$ be a prime number which does not divide $Np^m$. In particular, the case $m=0$ and $\ell=p$ is allowed. For all $a\in\{0,\dots,\ell-1\}$ denote by $\hat\lambda_a\in\widehat B^\times$ the idele whose $\ell$-component is equal to $\smallmat 1a0\ell$ and whose components at all other places are equal to $1$. Similarly, let $\hat\lambda_\infty$ be the idele whose $\ell$-component is equal to $\smallmat\ell001$ and all other components are $1$. Then
\[ \widehat R_m^\times\hat\lambda_0\widehat R_m^\times=\bigcup_{a=0}^{\ell-1}\widehat R_m^\times\hat\lambda_a\cup R_m^\times\hat\lambda_{\infty},\qquad U_m\hat\lambda_0 U_m=\bigcup_{a=0}^{\ell-1}U_m\hat\lambda_a\cup U_m\hat\lambda_\infty. \]
The action of $T_\ell$ on $\Div(X_m)$ and $\Div(\wt X_m)$ can be defined as
\[ T_\ell\bigl([(g,f)]\bigr):=\sum_{a=0}^{\ell-1}\bigl[\bigl(\hat\lambda_ag,f\bigr)\bigr]+\bigl[\bigl(\hat\lambda_\infty g,f\bigr)\bigr]. \]
The action of the Hecke operator $U_p$ on $\Div(X_m)$ and $\Div(\wt X_m)$ for $m\geq1$ will be especially important for us. For all $a\in\{0,\dots,p-1\}$ denote by $\hat\pi_a\in\widehat B^\times$ the idele whose $p$-component is equal to $\smallmat 1a0p$ and whose components at all other places are equal to $1$. Then
\[ \widehat R_m^\times\hat\pi_0\widehat R_m^\times=\bigcup_{a=0}^{p-1}\widehat R_m^\times\hat\pi_a, \qquad U_m\hat\pi_0U_m=\bigcup_{a=0}^{p-1}U_m\hat\pi_a. \]
The action of $U_p$ on $\Div(X_m)$ and $\Div(\wt X_m)$ is given by
\[ U_p\bigl([(g,f)]\bigr):=\sum_{a=0}^{p-1}[(\hat\pi_a g,f)]. \]
Observe that, as pointed out also in \cite[\S 1.5]{bd*}, the single terms in the sums expressing $T_\ell$ and $U_p$ depend on the choice of representative for $[(g,f)]$, but their collections do not.

\section{Heegner points}

Let $K$ be an imaginary quadratic field of discriminant $D=D_K$ prime to $pN$ and denote by $\cl O_K$ its ring of algebraic integers. Assume that the following \emph{Heegner hypothesis} is satisfied:
\begin{itemize}
\item a prime number $\ell$ divides $N^+$ (respectively, $N^-$) only if $\ell$ splits (respectively, is inert) in $K$.
\end{itemize}
No conditions are imposed on $p$.

\subsection{Heegner points} \label{Heegner-subsec}

Denote by $\mathcal O_K$ the ring of integers of $K$. For any integer $c\geq1$ prime to $N$ let $\mathcal O_c:=\Z+c\mathcal O_K$ be the order of $K$ of conductor 
$c$ and let $H_c$ denote the ring class field of $K$ of conductor $c$.  

For any order $\cl O\subset K$ and any Eichler order $R\subset B$, a morphism $f\in\Hom_\Q(K,B)$ is said to be an \emph{optimal embedding of $\cl O$ in $R$} if
\[ f(\cl O)=R\cap f(K) \qquad \text{(i.e., $f^{-1}(R)=\cl O$)}. \] 
We say  that a point $P=[(g,f)]\in X_m^{(K)}$ for some integer $m\geq0$ is a \emph{Heegner point of conductor $c$ on $X_m$} if $f$ is an optimal embedding of $\cl O_c$ into the Eichler order $g^{-1}\widehat R_mg\cap B$. Note that both in the definite and in the indefinite case Heegner points are contained in $X_m(\bar\Q)$. More precisely, suppose that $P$ is a Heegner point of conductor $c$ on $X_m$: if $B$ is definite then $P\in X_m(K)$, while if $B$ is indefinite then $P\in X_m(H_c)$. In the indefinite case Heegner points on $X_m$ are well known to satisfy Shimura's reciprocity law (\cite[Theorem 9.6]{Sh}) describing the action of $\Gal(H_c/K)$. 
 
For the next definition, for all integers $m\geq 0$ let $U_{m,p}$ denote the $p$-component of $U_m$.

\begin{defi} \label{def-Heegner-points}
We say that a point $\wt P=[(g,f)]\in \wt X_m^{(K)}$ is a \emph{Heegner point of conductor $c$ on $\wt X_m$} if $\beta_m\bigl(\wt P\bigr)\in X_m^{(K)}$ is a Heegner point of conductor $c$ and
\[ f_p^{-1}\bigl(f_p\bigl((\mathcal O_c\otimes\Z_p)^\times\bigr)\cap g_p^{-1}U_{m,p}g_p\bigr)=
(\mathcal O_c\otimes\Z_p)^\times\cap(1+p^m\mathcal O_K\otimes\Z_p)^\times. \]
\end{defi}

In other words, Heegner points on $\wt X_m$ are lifts of Heegner points on $X_m$ satisfying a suitable local condition at $p$. This condition will be used to study the field of rationality of Heegner points on $\widetilde X_m$. 

\subsection{Fields of rationality} \label{section-fields-of-rationality}

Define an action of $\Gal(K^{\rm ab}/K)\simeq\widehat K^\times/K^\times$ on $\wt X_m^{(K)}$ by the formula
\[ P^\fr a:=\bigl[\bigl(g\hat f(\fr a),f\bigr)\bigr] \]
for all $\fr a\in \widehat K^\times/K^\times$ and all $P=[(g,f)]\in\wt X_m^{(K)}$. Define
\[ Z_m:=\bigl\{a=(a_\ell)\in \widehat{\cl O}_{cp^m}^\times\mid a_p\equiv1\mod{p^m(\mathcal O_K\otimes\Z_p)}\bigr\}. \]
If $[(g,f)]\in \wt X_m^{(K)}$ is a Heegner point of conductor $cp^m$ then it follows directly from Definition \ref{def-Heegner-points} that
\begin{equation} \label{z-m-eq}
Z_m=\hat f^{-1}\bigl(\hat f(\widehat{\mathcal O}_{cp^m}^\times)\cap g^{-1}U_m g\bigr).
\end{equation}
For any number field $F$ denote by $I_F$ its idele group (so $\widehat F^\times$ is the finite part of $I_F$). Write $\wt H_{cp^m}$ for the class field of $Z_{m,\infty}:=Z_m\times\C^\times$, so that
\[ \Gal(\wt H_{cp^m}/K)\simeq\widehat K^\times/K^\times Z_m. \]

\begin{prop}
Let $P\in\wt X_m^{(K)}$ be a Heegner point of conductor $cp^m$. Then
\begin{enumerate}
\item $P \in H^0\bigl(\Gal(K^{\rm ab}/\wt H_{cp^m}),\wt X_m(K)\bigr)$ in the definite case;
\item $P\in\widetilde X_m\bigl(\wt H_{cp^m}\bigr)$ in the indefinite case.
\end{enumerate}
\end{prop}

\begin{proof} Use the fact that $P$ is fixed by the action of $\Gal(K^\rom{ab}/\wt H_{cp^m})$ and that in the indefinite case $P$ is rational over $K^\rom{ab}$ by, for example, \cite[Lemma 3.11]{cv}. \end{proof}

We give a more explicit description of $\wt H_{cp^m}$. As a general notation, for every integer $n\geq1$ let $\boldsymbol\mu_n$ be the $n$-th roots of unity. Set $p^\ast:=(-1)^{(p-1)/2}p$.

\begin{prop} \label{rationality-pro}
$\wt H_{cp^m}=H_{cp^m}(\boldsymbol\mu_{p^m})$.
\end{prop}

\begin{proof} Write $\Gal(K/\Q)\simeq I_\Q/\Q^\times C$ where $C:=\Norm_{K/\Q}I_K$ is the norm group of $K$. Define
\[ W_m:=\prod_{\ell\neq p}\Z_\ell^\times\times\{\alpha\in\Z_p^\times\mid\alpha\equiv 1\pmod{p^m}\} \]
and set $W_{m,\infty}:=W_m\times\R_+$ where $\R_+$ is the group of positive real numbers. The extension $K(\boldsymbol\mu_{p^m})/\Q$ is abelian, and since $\Gal(\Q(\boldsymbol\mu_{p^m})/\Q) \simeq I_\Q/\Q^\times W_{m,\infty}$ it follows by global class field theory (cf. \cite[Ch. IV, Theorem 7.1]{neu}) that
\[ \Gal(K(\boldsymbol\mu_{p^m})/\Q)\simeq I_\Q/\Q^\times (C\cap W_{m,\infty}). \]
Now $\Gal(K(\boldsymbol\mu_{p^m})/K)\simeq\widehat K^\times/K^\times V_m$ where $V_m$ denotes the finite part of the norm group $\Norm_{K(\boldsymbol\mu_{p^m})/K}\bigl(I_{K(\boldsymbol\mu_{p^m})}\bigr)$. Hence
\begin{equation} \label{eq0}
\Gal(H_{cp^m}(\boldsymbol\mu_{p^m})/K)\simeq\widehat K^\times/K^\times(V_m\cap\widehat{\cl O}_{cp^m}^\times).
\end{equation}
Since the finite part of $\Norm_{K(\boldsymbol\mu_{p^m})/\Q}\bigl(I_{K(\boldsymbol\mu_{p^m})}\bigr)$ equals $\Norm_{K/\Q}(V_m)$ and
\[ \Q^\times\Norm_{K(\boldsymbol\mu_{p^m})/\Q}\bigl(I_{K(\boldsymbol\mu_{p^m})}\bigr)=\Q^\times (C\cap W_{m,\infty}), \]
it follows that
\[ V_m\subset\bigl\{x\in \widehat K^\times\mid\Norm_{K/\Q}(x)\in \Q^\times W_m\bigr\}. \]
Let $x\in V_m\cap\widehat{\cl O}_{cp^m}^\times$ and write $x=\alpha+cp^m\beta$ with $\alpha\in\widehat\Z$ and $\beta\in\widehat{\cl O}_K$. Then $\Norm_{K/\Q}(x)\in \Q^\times W_m\cap\widehat\Z^\times=W_m$. On the other hand, locally at $p$ one has the congruence
\[ \Norm_{K_p/\Q_p}(x_p)\equiv\alpha_p^2\pmod{p^m}. \]
It follows that $\alpha_p\equiv\pm 1\pmod{p^m}$, and we get the inclusion $K^\times(V_m\cap\widehat{\cl O}_{cp^m}^\times)\subset K^\times Z_m$. Isomorphism \eqref{eq0} finally yields
\begin{equation} \label{eq1}
\wt H_{cp^m}\subset H_{cp^m}(\boldsymbol\mu_{p^m}).
\end{equation}
It is easily seen that the Galois group $\Gal(H_{cp^m}(\boldsymbol\mu_{p^m})/H_{cp^m})$ is isomorphic to $\widehat{\cl O}_{cp^m}^\times/{\cl O}_{cp^m}^\times Z_m$. Since $\widehat{\cl O}_{cp^m}^\times/Z_m$ is isomorphic to $(\Z/p^m\Z)^\times$ via the map which sends $a=(a_q)_q\in\widehat{\cl O}_{cp^m}^\times$ to $a_p\pmod{p^m}\in(\Z/p^m\Z)^\times$, and $\cl O_{cp^m}^\times=\{\pm 1\}$ for $m\geq 1$, we get that
\begin{equation} \label{eq2}
\bigl[\wt H_{cp^m}:H_{cp^m}\bigr]=\varphi(p^m)/2.
\end{equation}
The result follows from \eqref{eq1} and \eqref{eq2} upon noticing that  $[H_{cp^m}(\boldsymbol\mu_{p^m}):H_{cp^m}]\leq\varphi(p^m)/2$ because $\Q(\sqrt{p^\ast})\subset H_{cp^m}$. \end{proof}

In light of Proposition \ref{rationality-pro}, from now on we adopt the explicit notation $H_{cp^m}(\boldsymbol\mu_{p^m})$ in place of the shorthand $\wt H_{cp^m}$. The reason for doing so is that whenever $p|c$ we have $\wt H_{cp^m}\neq\wt H_{(c/p)p^{m+1}}$, so the previous notation would be ambiguous.

\subsection{Hecke relations on $\wt X_m$}

Let $r,s\geq1$ be integers. Then
\[ \Gal\bigl(H_{cp^s}(\boldsymbol\mu_{p^r})/K\bigr)\simeq\widehat K^\times\big/K^\times\bigl(V_r\cap\widehat{\cl O}_{cp^s}^\times\bigr) \]
where $V_r$ is the finite part of the norm group $\Norm_{K(\boldsymbol\mu_{p^r})/K}\bigl(I_{K(\boldsymbol\mu_{p^r})}\bigr)$. Hence for every pair of integers $t,u$ with $t\geq s$ and $u\geq r$ there is an isomorphism
\[ K^\times\bigl(V_r\cap\widehat{\cl O}_{cp^s}^\times\bigr)\big/K^\times\bigl(V_u\cap\widehat{\cl O}_{cp^t}^\times\bigr)\overset{\simeq}{\longrightarrow}\Gal\bigl(H_{cp^t}(\boldsymbol\mu_{p^u})/H_{cp^s}(\boldsymbol\mu_{p^r})\bigr). \]
As pointed out in the proof of Proposition \ref{rationality-pro}, every element $x=\alpha+cp^s\beta\in V_r\cap\widehat{\cl O}_{cp^s}^\times$ (with $\alpha\in\widehat\Z$ and $\beta\in\widehat{\cl O}_K$) satisfies the local conditions
\[ {\rm Norm}_{K_p/\Q_p}(x_p)\equiv\alpha_p^2\pmod{p^r},\qquad {\rm Norm}_{K_p/\Q_p}(x_p)\equiv 1\pmod{p^r}. \]
Let $\sigma\in\Gal\bigl(H_{cp^{n+1}}(\boldsymbol\mu_{p^{n+1}})/H_{cp^n}(\boldsymbol\mu_{p^{n+1}})\bigr)$ be represented by the idele $\fr a_\sigma\in\widehat{\cl O}_{cp^n}^\times$. By the above discussion, we have
\begin{equation} \label{cong-loc-p-eq}
\fr a_\sigma=\alpha+cp^n\beta,\qquad\alpha_p\equiv 1\pmod{p^n}.
\end{equation}

\begin{prop} \label{prop-Galois-Hecke-tilde-X}
Let $\wt P$ be a Heegner point of conductor $cp^n$ on $\wt X_m$ for some $n\geq m\geq 1$ and let $\wt Q\in\wt X_m^{(K)}$ belong to the support of $U_p(\wt P)$. Then
\[ U_p\bigl(\wt P\bigr)=\mathrm{tr}_{H_{cp^{n+1}}(\boldsymbol\mu_{p^{n+1}})/H_{cp^n}(\boldsymbol\mu_{p^{n+1}})}\bigl(\wt Q\bigr) \]
in $\Div\bigl(\wt X_m\bigr)$.
\end{prop}

\begin{proof} Let $\wt P=[(g,f)]$, so that $\wt Q=[(\hat\pi_ag,f)]$ for a certain $a\in\{0,\dots,p-1\}$. Hence, with notation as above, $\wt Q^\sigma=\bigl[\bigl(\hat\pi_ag\hat f(\fr a_\sigma),f\bigr)\bigr]$. Since $\wt P$ is a Heegner point of conductor $cp^n$ and $n\geq m$, it follows from \eqref{z-m-eq} and \eqref{cong-loc-p-eq} that $\hat f(\fr a_\sigma)\in g^{-1}U_mg$. Thus there esists $x\in U_m$ such that $\wt Q^\sigma=[(\hat\pi_axg,f)]$, whence
\begin{equation} \label{eq17}
\wt Q^\sigma\in U_p\bigl(\wt P\bigr).
\end{equation}
Since the fields $H_{cp^{n+1}}$ and $H_{cp^n}(\boldsymbol\mu_{p^{n+1}})$ are linearly disjoint over $H_{cp^n}$, the projection
\[ \Gal\bigl(H_{cp^{n+1}}(\boldsymbol\mu_{p^{n+1}})/H_{cp^n}(\boldsymbol\mu_{p^{n+1}})\bigr)\longrightarrow\Gal(H_{cp^{n+1}}/H_{cp^n}) \]
is an isomorphism. Therefore the order of $\Gal\bigl(H_{cp^{n+1}}(\boldsymbol\mu_{p^{n+1}})/H_{cp^n}(\boldsymbol\mu_{p^{n+1}})\bigr)$ is $p$, and the claim of the proposition follows from \eqref{eq17}. \end{proof}

For simplicity, for the next proposition assume that $\cl O_{cp^m}^\times=\{\pm1\}$ (this excludes only the cases where $c=1$, $m=0$ and $K=\Q(\sqrt{-1})$ or $K=\Q(\sqrt{-3})$). 

\begin{prop} \label{T-operator-1}
Let $m\geq0$. Fix a prime $\ell\nmid Np^mc$ which is inert in $K$. Let $\wt P$ be a Heegner point of conductor $c\ell p^m$ on $\wt X_m$ and let $\wt Q\in\wt X_m^{(K)}$ belong to the support of $T_\ell(\wt P)$. Then
\[ T_\ell\bigl(\wt P\bigr)=\mathrm{tr}_{H_{c\ell p^m}(\boldsymbol{\mu}_{p^m})/H_{cp^m}(\boldsymbol{\mu}_{p^m})}\bigl(\wt Q\bigr) \]
in $\Div\bigl(\wt X_m\bigr)$. 
\end{prop}

\begin{proof} Arguing exactly as in the proof of Proposition \ref{prop-Galois-Hecke-tilde-X}, it can be shown that if $\sigma\in\Gal\bigl(H_{c\ell p^m}(\boldsymbol{\mu}_{p^m})/H_{cp^m}(\boldsymbol{\mu}_{p^m})\bigr)$ then
\begin{equation} \label{eq17bis}
\wt Q^\sigma\in T_\ell\bigl(\wt P\bigr).
\end{equation}
On the other hand, since the fields $H_{c\ell p^m}$ and $H_{cp^m}(\boldsymbol{\mu}_{p^m})$ are linearly disjoint over $H_{cp^m}$, the canonical projection induces an isomorphism
\begin{equation} \label{gal-T}
\Gal\bigl(H_{c\ell p^m}(\boldsymbol{\mu}_{p^m})/H_{cp^m}(\boldsymbol{\mu}_{p^m})\bigr)\overset\simeq\longrightarrow\Gal\bigl(H_{c\ell p^m}/H_{cp^m}\bigr).
\end{equation}
The claim of the proposition follows by combining \eqref{eq17bis} and \eqref{gal-T} because the two divisors are both sums of $\ell+1$ points. \end{proof}

\section{Families of Heegner points} \label{section-explicit-construction}

The purpose of this section is to construct a family of Heegner points on the tower of Shimura curves which satisfies suitable compatibility properties with respect to the natural covering maps in the tower. These points will be the building blocks in our definition of big Heegner points and classes that will be performed in Section \ref{big-heegner-section}. Unlike what is done in \cite{ho}, to achieve our goal we systematically adopt the language of optimal embeddings, and this approach allows us to treat in a uniform way both the definite and the indefinite case.

\subsection{Choice of local conditions} \label{heegner-points-explicit-subsec}

In order to introduce the systems of Heegner points that we shall work with, we need to recall some auxiliary results and definitions. As a preliminary remark, the Heegner hypothesis and \cite[Theorems 1 and 2]{ogg} ensure that the set of Heegner points of conductor $cp^m$ on $X_m$ is not empty.

Let $\cl O$ be an order of $K$ and $R$ an order of $B$. Let $\ell$ be a prime number. Define $K_\ell:=K\otimes_\Z \Z_\ell$ and $B_\ell:=B\otimes_\Z\Z_\ell$. An injective homomorphism $\varphi:K_\ell\hookrightarrow B_\ell$ of $\Q_\ell$-algebras is said to be an \emph{optimal embedding} of $\cl O\otimes \Z_\ell$ into $R\otimes \Z_\ell$ if
\[ \varphi(\cl O\otimes\Z_\ell)=\varphi(K_\ell)\cap(R\otimes\Z_\ell)\qquad (\text{i.e., $\varphi^{-1}(R\otimes\Z_\ell)=\cl O\otimes\Z_\ell$}). \]
Two optimal embeddings $\varphi$ and $\psi$ of $\cl O\otimes\Z_\ell$ into $R\otimes \Z_\ell$ are said to be \emph{equivalent} if there exists an element $u\in(R\otimes \Z_\ell)^\times$ such that $\varphi=u^{-1}\psi u$.

If $f:K\hookrightarrow B$ is an injective homomorphism of $\Q$-algebras and $\ell$ is a prime number, denote by $f_\ell=f\otimes\text{id}_{\Z_\ell}:K_\ell\hookrightarrow B_\ell$ the homomorphism which is obtained from $f$ by extension of scalars. The next lemma says that, for a global injection, the property of being an optimal embedding is local.

\begin{lemma} \label{lemma-locale-globale}
An injective homomorphism of $\Q$-algebras $f:K\hookrightarrow B$ is an optimal embedding of $\cl O$ into $R$ if and only if $f_\ell$ is an optimal embedding of $\cl O\otimes\Z_\ell$ into $R\otimes\Z_\ell$ for all primes $\ell$.
\end{lemma}

\begin{proof} A routine verification; see \cite[Lemma 4.9]{Pi2} for a quick proof using the elementary divisor theorem. \end{proof}

Let $R$ be an Eichler order of $B$, let $I_1,\dots,I_h$ be representatives of all the distinct classes of left $R$-ideals and denote by $R_i$ the right order of $I_i$ for $i=1,\dots,h$. The number $h$ depends only on the level of $R$ and the discriminant of the quaternion algebra, and the set $\{R_1,\dots,R_h\}$ is a system of representatives for all the $R$-conjugacy classes of Eichler orders in $B$ with the same level as $R$. For every $i\in\{1,\dots,h\}$ fix an element $\gamma_i\in\widehat B^\times$ such that $\widehat R_i=\gamma_i^{-1}\widehat R\gamma_i$ and write $\gamma_{i,\ell}$ for the $\ell$-component of $\gamma_i$ at a prime $\ell$.

\begin{prop} \label{loc-glob-emb-prop}
Let $\cl O$ be an order of $K$ and $R$ an Eichler order of $B$, and let $\{\varphi_\ell\}_\ell$ be a collection of optimal embeddings of $\cl O\otimes\Z_\ell$ into $R\otimes\Z_\ell$ for all primes $\ell$. Then there exists an optimal embedding $f:K\hookrightarrow B$ of $\cl O$ into $R_i$ for some $i\in\{1,\dots,h\}$ such that $\gamma_{i,\ell}f_\ell\gamma_{i,\ell}^{-1}$ is equivalent to $\varphi_\ell$ for all $\ell$.
\end{prop}

\begin{proof} This is essentially a consequence of Eichler's trace formula (\cite[Ch. III, Th{\'e}or{\`e}me 5.11]{vi}). For later use, we give here a direct proof (see \cite[Ch. III, \S 5]{vi} or \cite[\S 3]{Pi} for more details). By the construction of $B$, there exists an injective homomorphism $g:K\hookrightarrow B$ of $\Q$-algebras. By the Skolem--Noether theorem, for every prime $\ell$ there exists $a_\ell\in B_\ell^\times$ such that $g_\ell=a_\ell^{-1}\varphi_\ell a_\ell$. For almost all primes $\ell$ which do not divide the discriminant of $B$, the level of $R$ and the conductor of $\cl O$ the map $g_\ell$ is an optimal embedding of $\cl O\otimes\Z_\ell$ into $R\otimes\Z_\ell$: this is so because $g(\cl O)$ is contained in a maximal order whose $\ell$-adic completion is equal to $R\otimes\Z_\ell$ for almost all $\ell$. Hence we can assume that $a_\ell\in (R\otimes\Z_\ell)^\times$ for almost all $\ell$; in fact, by \cite[Ch.  II, \S 3]{vi}, if $\ell$ does not divide the discriminant of $B$ and the level of $R$ there is only one equivalence class of optimal embeddings of $\cl O\otimes\Z_\ell$ into $R\otimes\Z_\ell$. Write $a$ for the idele $(a_\ell)_\ell$. By the strong approximation theorem, there exist a unique index $i\in\{1,\dots,h\}$, a global element $b\in B^\times$ and a unit $u\in\widehat R^\times$ such that $a=u\gamma_ib$. Then $f:=bgb^{-1}$ is a global embedding of $K$ into $B$ such that $f_\ell$ is conjugate to $\varphi_\ell$ for all primes $\ell$. In fact, for every prime $\ell$ one has
\[ \gamma_{i,\ell}f_\ell\gamma_{i,\ell}^{-1}= \gamma_{i,\ell}bg_\ell b^{-1}\gamma_{i,\ell}^{-1}= (\gamma_{i,\ell}ba_\ell^{-1})\varphi_\ell(\gamma_{i,\ell}ba_\ell^{-1})^{-1}=u_\ell^{-1}\varphi_\ell u_\ell, \]
which shows that $\gamma_{i,\ell}f_\ell\gamma_{i,\ell}^{-1}$ is equivalent to $\varphi_\ell$. In particular, $f_\ell$ is an optimal embedding of $\cl O\otimes\Z_\ell$ into $\gamma_{i,\ell}^{-1}(R\otimes\Z_\ell)\gamma_{i,\ell}=R_i\otimes\Z_\ell$ for every prime $\ell$, hence $f$ is an optimal embedding of $\cl O$ into $R_i$ by Lemma \ref{lemma-locale-globale}. \end{proof}

For all integers $k,m\geq0$ define the Eichler order $R_k^{(m)}$ of level $N^+p^k$ by the following local conditions:
\[ R_k^{(m)}\otimes\Z_\ell=R_0\otimes\Z_\ell \qquad \text{for all $\ell\neq p$}; \]
\[ R_k^{(m)}\otimes\Z_p=\bigl(\begin{smallmatrix}\Z_p&p^{k-m}\Z_p\\p^m\Z_p&\Z_p\end{smallmatrix}\bigr)=\smallmat01{-p^m}0(R_k\otimes\Z_p)\smallmat0{-p^{-m}}10. \] 
In particular, we have $R_m=R_0\cap R_0^{(m)}$ and $R_m^{(m)}=R_m$.

Proposition \ref{loc-glob-emb-prop} reduces the construction of global optimal embeddings to that of local ones. In the following the local component at $p$ is studied. Let $p^h$ be the power of $p$ dividing $c$ exactly (i.e., $p^h|c$ but $p^{h+1}\nmid c$). Write $K=\Q(\sqrt{-D})$ with $D>0$, so $K_p=\Q_p(\sqrt{-D})$ if $p$ is inert in $K$ and $K_p=\Q_p\oplus\Q_p$ if $p$ is split in $K$. Then we consider the following embeddings $\psi_p^{(c)}:K_p\hookrightarrow B_p$:

\bigskip

\noindent 1) $p$ \underline{inert}
\[ \begin{array}{ccc}
   K_p & \longrightarrow & B_p\\[2mm]
   \alpha+\sqrt{-D}\beta & \longmapsto & \Big(\begin{smallmatrix}\alpha&-D\beta p^h\\\beta/p^h&\alpha\end{smallmatrix}\Big);
   \end{array} \]
\noindent 2) $p$ \underline{split}
\[ \begin{array}{ccc}
   K_p & \longrightarrow & B_p\\[2mm]
   (\alpha,\beta) & \longmapsto & \Big(\begin{smallmatrix}\alpha&0\\(\alpha-\beta)/p^h&\beta\end{smallmatrix}\Big).
   \end{array} \]
Recall that $U_{n,p}$ denotes the $p$-component of $U_n$. For all integers $n\geq0$ an easy calculation shows that
\begin{itemize}
\item $\psi^{(c)}_p$ is an optimal embedding of $\mathcal O_{cp^n}\otimes\Z_p$ into $R_n\otimes\Z_p$;
\item $(\psi_p^{(c)})^{-1}\bigl(\psi_p^{(c)}\bigl((\mathcal O_{cp^n}\otimes\Z_p)^\times\bigr)\cap U_{n,p}\bigr)=(\mathcal O_{cp^n}\otimes\Z_p)^\times\cap(1+p^n\mathcal O_K\otimes\Z_p)^\times$.
\end{itemize}
Define
\begin{equation} \label{def-varphi}
\varphi_p^{(c,m)}:=\smallmat01{-p^m}0\psi_p^{(c)}\smallmat0{-p^{-m}}10
\end{equation}
and
\[ U_{n,p}^{(m)}:=\smallmat 01{-p^m}0 U_{n,p}\smallmat 0{-p^{-m}}10. \]
For all integers $m,n\geq0$ it follows from the above equations for $\psi_p^{(c)}$ and the definition of $R_n^{(m)}$ that
\begin{itemize}
\item $\varphi^{(c,m)}_p$ is an optimal embedding of of $\mathcal O_{cp^n}\otimes\Z_p$ into $R_n^{(m)}\otimes\Z_p$;
\item $(\varphi_p^{(c,m)})^{-1}\bigl(\varphi_p^{(c,m)}\bigl((\mathcal O_{cp^n}\otimes\Z_p)^\times\bigr)\cap U_{n,p}^{(m)}\bigr)=(\mathcal O_{cp^n}\otimes\Z_p)^\times\cap(1+p^n\mathcal O_K\otimes\Z_p)^\times$.
\end{itemize}

\begin{lemma} \label{admissible-prop-2}
Fix an integer $m\geq0$. Then
\begin{enumerate}
\item $\varphi_p^{(c,m)}$ is an optimal embedding of $\cl O_{cp^m}\otimes\Z_p$ into $R_m\otimes\Z_p$.
\item $(\varphi_p^{(c,m)})^{-1}\bigl(\varphi_p^{(c,m)}\bigl((\mathcal O_{cp^m}\otimes\Z_p)^\times\bigr)\cap U_{m,p}\bigr)=(\mathcal O_{cp^m}\otimes\Z_p)^\times\cap(1+p^m\mathcal O_K\otimes\Z_p)^\times$.
\end{enumerate}
\end{lemma}

\begin{proof} Since $R_m^{(m)}=R_m$ and
\[ \varphi_p^{(c,m)}\bigl((\mathcal O_{cp^m}\otimes\Z_p)^\times\bigr)\cap U_{m,p}=\varphi_p^{(c,m)}\bigl((\mathcal O_{cp^m}\otimes\Z_p)^\times\bigr)\cap U_{m,p}^{(m)}, \] 
both claims are immediate consequences of the above formulas for $n=m$. \end{proof}

For every prime $\ell\not=p$ choose an optimal embedding $\varphi_{\ell,0}:K_\ell\hookrightarrow B_\ell$ of $\cl O_K\otimes\Z_\ell$ into $R_m\otimes\Z_\ell=R_0\otimes\Z_\ell$: this can be done by \cite[Theorem 2]{ogg}. If $\ell\nmid Np$ then for all integers $n\geq1$ fix also optimal embeddings $\varphi_{\ell,n}:K_\ell\hookrightarrow B_\ell$ of $\cl O_{\ell^n}\otimes\Z_\ell$ into $R_0\otimes\Z_\ell$. For any prime $\ell$ set $\pi_\ell:=\smallmat 100\ell$. If $\ell$ is inert in $K$ and $\ell\nmid Np$ then choose $\varphi_{\ell,0}$ and $\varphi_{\ell,1}$ in such a way that $\varphi_{\ell,1}=\pi_\ell\varphi_{\ell,0}\pi_\ell^{-1}$ (for example, this can explicitly be done by adopting definitions analogous to those of $\varphi_p^{(c,1)}$ and $\varphi_p^{(c,0)}$). Suppose that we have fixed $\varphi_{\ell,n}$ for prime $\ell\neq p$ and integers $n\geq0$ so that all the above conditions are fulfilled. Let $c\geq1$ be an integer prime to $N$ and the discriminant of $K$. For primes $\ell\neq p$ set $\varphi_\ell^{(c)}:=\varphi_{\ell,n(\ell)}$ where $\ell^{n(\ell)}$ is the maximal power of $\ell$ dividing $c$. We remark that, with these choices, $\varphi_\ell^{(c)}=\varphi_\ell^{(c')}$ whenever $\ell\nmid cc'$. 

\subsection{Compatible families of Heegner points} \label{compatible-Heegner-subsec}

We now use the proof of Proposition \ref{loc-glob-emb-prop} to globalize the local choices performed 
in \S \ref{heegner-points-explicit-subsec}. To begin with, fix an injection $g:K\hookrightarrow B$ of $\Q$-algebras. Choose elements $a_{\ell,n}\in B_\ell^\times$ and $a^{(c,m)}_p\in B_p^\times$ such that
\begin{itemize}  
\item $g_\ell=a_{\ell,n}^{-1}\varphi_{\ell,n} a_{\ell,n}$ for $\ell\neq p$ and $n\geq0$;
\item $g_p=(a^{(c,m)}_p)^{-1}\varphi_p^{(c,m)}a^{(c,m)}_p$.
\end{itemize} 
We can perform the above choices at $p$ and at primes $\ell\nmid Npc$ which are inert in $K$ as follows. First, assume that $p\nmid c$. Let $a^{(c,0)}_p$ be an arbitrary element satisfying the above relation for $m=0$. Since $\varphi_p^{(c,m)}=\pi_p\varphi_p^{(c,m-1)}\pi^{-1}_p$, we can define inductively
\begin{equation} \label{+1} 
a^{(c,m)}_p:=\pi_pa^{(c,m-1)}_p
\end{equation} 
for all $m\geq1$. Keeping the condition $p\nmid c$, since $\varphi_p^{(cp^{h-1},m+1)}=\varphi_p^{(cp^h,m)}$ for all integers $m\geq0$ and $h\geq1$, we can also define inductively 
\begin{equation} \label{+2} 
a^{(cp^h,m)}_p:=a^{(cp^{h-1},m+1)}_p
\end{equation} 
for all $m\geq0$ and $h\geq1$. The two conditions above define $a_p^{(c,m)}$ for all $c\geq 1$ and $m\geq 0$. Next, as in the rest of the paper, let $c\geq1$ be an integer prime to $N$. Suppose the prime $\ell$ is inert in $K$ and $\ell\nmid Npc$. In this case, recall that we have chosen $\varphi_{\ell,1}$ and $\varphi_{\ell,0}$ so that $\varphi_{\ell,1}=\pi_\ell\varphi_{\ell,0}\pi^{-1}_\ell$. Therefore we can fix an arbitrary $a_{\ell,0}$ so that $g_\ell=a_{\ell,0}^{-1}\varphi_{\ell,0} a_{\ell,0}$ and define $a_{\ell,1}:=\pi_\ell a_{\ell,0}$.  

Fix an integer $c\geq1$ as before and recall that if $\ell$ is a prime then $\ell^{n(\ell)}$ denotes the exact power of $\ell$ dividing $c$. Fix also an integer $m\geq1$ and define global elements $a^{(c,m)}\in\widehat B^\times$ by the following local conditions:
\begin{itemize}
\item the $\ell$-component of $a^{(c,m)}$ is equal to $a_{\ell,n(\ell)}$ for all $\ell\not=p$; 
\item the $p$-component of $a^{(c,m)}$ is equal to the $a_p^{(c,m)}$ chosen before. 
\end{itemize}
Thus for primes $\ell\nmid Npc$ which are inert in $K$ there is an equality 
\begin{equation} \label{+3}
a^{(c\ell,m)}_\ell=\pi_\ell a^{(c,m)}_\ell.
\end{equation} 
Observe that if $\ell\nmid cc'$ or $\ell$ divides both $c$ and $c'$ to the same power then $a^{(c,m)}_\ell=a^{(c',m)}_\ell$ for all $m\geq0$. Recall the idele $\hat\pi_0\in\widehat B^\times$ introduced in \S \ref{section-Hecke-operators}, with $p$-component equal to $\pi_p$ and all other components equal to $1$; then combining the definition of $a^{(c,m)}$ with \eqref{+1} and \eqref{+2} yields equalities
\begin{equation} \label{a-1}
a^{(c,m)}=\hat\pi_0a^{(c,m-1)}
\end{equation}
and
\begin{equation} \label{a-2}
a^{(cp^h,m)}=a^{(cp^{h-1},m+1)}
\end{equation}
for all $m\geq0$ and $h\geq1$. Similarly, if $\ell\nmid Npc$ and $\hat\lambda_0\in\widehat B^\times$ is the idele of \S \ref{section-Hecke-operators} with all components equal to $1$ except the $\ell$-component which is equal to $\pi_\ell$ then \eqref{+3} gives the equality
\begin{equation} \label{a-3}
a^{(c\ell,m)}=\hat\lambda_0a^{(c,m)}
\end{equation}
for all $m\geq0$. 

Denote by $R_{m,1},\dots,R_{m,h(m)}$ the right orders of a set of representatives of the left $R_m$-ideals and take $\gamma_{m,1},\dots,\gamma_{m,h(m)}\in\widehat B^\times$ such that $\widehat R_{m,i}=\gamma_{m,i}^{-1}\widehat R_m\gamma_{m,i}$ for all $i\in\{1,\dots,h(m)\}$. The set 
\[\mathfrak S_m:=\left\{\gamma_{m,1},\dots,\gamma_{m,h(m)}\right\}\] is a complete set of representatives for the double coset space $\widehat R_m^\times\backslash\widehat B^\times/B^\times$. 
Write 
\begin{equation} \label{+5}
a^{(c,m)}=u_{c,m}\gamma^{(c,m)}b_{c,m}
\end{equation} 
with $u_{c,m}\in\widehat R_m^\times$, $\gamma^{(c,m)}\in\mathfrak S_m$ and $b_{c,m}\in B^\times$. Then 
define \[f^{(c,m)}=b_{c,m}gb_{c,m}^{-1}.\] Note in particular that we obtain the equality
\begin{equation} \label{+4}  
g=b^{-1}_{c,m}f^{(c,m)}b_{c,m}=b^{-1}_{c',m'}f^{(c',m')}b_{c',m'}
\end{equation} 
for all $c,c',m,m'$.  Now the local embedding 
\[ {\gamma^{(c,m)}_\ell}f^{(c,m)}_\ell(\gamma^{(c,m)}_\ell)^{-1} \] 
is equivalent to $\varphi_\ell^{(c)}$ for every $\ell\neq p$ and to $\varphi^{(c,m)}_p$ for $\ell=p$.  

\begin{prop} \label{prop-globale}
Fix an integer $m\geq0$. Then
\begin{enumerate}
\item $f ^{(c,m)}$ is an optimal embedding of $\cl O_{cp^m}$ into $B\cap(\gamma^{(c,m)})^{-1}\widehat R_m\gamma^{(c,m)}$;
\item We have the equality: 
\[(f_p^{(c,m)})^{-1}\bigl(f_p^{(c,m)}\bigl((\mathcal O_{cp^m}\otimes\Z_p)^\times\bigr)\cap(\gamma_p^{(c,m)})^{-1}U_{m,p}\gamma^{(c,m)}_p\bigr)=(\mathcal O_{cp^m}\otimes\Z_p)^\times\cap(1+p^m\mathcal O_K\otimes\Z_p)^\times.\]
\end{enumerate}
\end{prop}

\begin{proof} Part (1) is just a restatement of the definition of $f^{(c,m)}$. For (2), observe that there exists an element $u\in{(\gamma^{(c,m)}_p)}^{-1}(R_m\otimes\Z_p)^\times{\gamma^{(c,m)}_p}$ such that
\begin{equation} \label{descr-f}
f^{(c,m)}_p=u^{-1}{(\gamma^{(c,m)}_p)}^{-1}\varphi^{(c,m)}_p{\gamma^{(c,m)}_p}u.
\end{equation}
The second statement in the proposition follows from part (2) of Lemma \ref{admissible-prop-2} combined with equation \eqref{descr-f} and the fact that $U_{m,p}$ is a normal subgroup of $(R_m\otimes\Z_p)^\times$. \end{proof}

\begin{coro} \label{Heegner}
The class $\bigl[\bigl(\gamma^{(c,m)},f^{(c,m)}\bigr)\bigr]$ is a Heegner point of conductor $cp^m$ both on $X_m^{(K)}$ and on $\wt X_m^{(K)}$.
\end{coro}

\begin{proof} Both statements are immediate consequences of Proposition \ref{prop-globale}. \end{proof}

Define the family of points 
\[ P_{c,0}=\wt P_{c,0}:=\bigl[\bigl(\gamma^{(c,0)},f^{(c,0)}\bigr)\bigr]\in X_0^{(K)}=\wt X_0^{(K)} \]
\[ P_{c,m}:=\bigl[\bigl(\gamma^{(c,m)},f^{(c,m)}\bigr)\bigr]\in X_m^{(K)},\qquad\wt P_{c,m}:=\bigl[\bigl(u_{c,m}\gamma^{(c,m)},f^{(c,m)}\bigr)\bigr]\in\wt X_m^{(K)}. \]
The point $\wt P_{c,m}$ is a suitable lift of $P_{c,m}$ to $\wt X_m^{(K)}$. We first note the following property enjoyed by these points.

\begin{prop} \label{prop-Heegner-I}
The point $P_{c,m}$ (respectively, $\wt P_{c,m}$) is a Heegner point of conductor $cp^m$ on $X_m$ (respectively, $\wt X_m$).
\end{prop}
\begin{proof} A direct consequence of Corollary \ref{Heegner}, where for $\wt P_{m,c}$ we use again the fact that $U_{m,p}$ is a normal subgroup of $(R_m\otimes\Z_p)^\times$. \end{proof}

\subsection{Hecke relations in compatible families} \label{compatible-Hecke-families}

The results we prove in this subsection justify our choice of the lifts $\wt P_{c,m}$ of the points $P_{c,m}$. Write
\[ \wt\alpha_{m,\ast}:\mathrm{Div}(\wt X_m)\longrightarrow\mathrm{Div}(\wt X_{m-1}) \]
for the map between divisor groups induced by $\wt\alpha_m$ by covariant functoriality. In other words, $\wt\alpha_{m,\ast}(P_1+\dots+P_s):=\wt\alpha_m(P_1)+\dots+\wt\alpha_m(P_s)$ for all points $P_1,\dots,P_s$ on $\widetilde X_m$. 

\begin{prop} \label{Hecke-relation-for-X}
Let $m\geq2$. Then
\[ U_p\bigl(\wt P_{c,m-1}\bigr)=\wt\alpha_{m,\ast}\bigl(\mathrm{tr}_{H_{cp^m}(\boldsymbol\mu_{p^m})/H_{cp^{m-1}}(\boldsymbol\mu_{p^m})}(\wt P_{c,m})\bigr) \] 
in $\Div\bigl(\wt X_{m-1}\bigr)$.
\end{prop}

\begin{proof} The image of $\wt P_{c,m}\in\wt X_m^{(K)}$ in $\wt X_{m-1}^{(K)}$ is given by 
\[ \wt\alpha_m(\wt P_{c,m})=\bigl[\bigl(\hat\pi_0u_{c,m-1}\gamma^{(c,m-1)}b_{c,m-1}b_{c,m}^{-1},f^{(c,m)}\bigr)\bigr]=\bigl[\bigl(\hat\pi_0u_{c,m-1}\gamma^{(c,m-1)},f^{(c,m-1)}\bigr)\bigr], \] 
where the first equality comes from \eqref{a-1} and \eqref{+5} and the second from \eqref{+4}. So $\wt\alpha_m(\wt P_{c,m})$ belongs to the support of $U_p(\wt P_{c,m-1})$, and the result follows from Proposition \ref{prop-Galois-Hecke-tilde-X}. \end{proof}

\begin{prop} \label{horizontal-hecke-X}
Let $m,r\geq1$. Then
\[ U_p\bigl(\wt P_{cp^{r-1},m}\bigr)=\mathrm{tr}_{H_{cp^{m+r}}(\boldsymbol\mu_{p^{m+r}})/H_{cp^{m+r-1}}(\boldsymbol\mu_{p^{m+r}})}\bigl(\wt P_{cp^r,m}\bigr) \] 
in $\Div\bigl(\wt X_m\bigr)$.
\end{prop}

\begin{proof} We compute:
\[ \begin{split}
   \bigl[\bigl(\hat\pi_0u_{cp^{r-1},m}\gamma^{(cp^{r-1},m)},f^{(cp^{r-1},m)}\bigr)\bigr]&=\bigl[\bigl(u_{cp^{r-1},m+1}\gamma^{(cp^{r-1},m+1)}b_{cp^{r-1},m+1}b_{cp^{r-1},m}^{-1},f^{(cp^{r-1},m)}\bigr)\bigr]\\
   &=\bigl[\bigl(u_{cp^r,m}\gamma^{(cp^r,m)}b_{cp^r,m}b_{cp^{r-1},m}^{-1},f^{(cp^{r-1},m)}\bigr)\bigr]\\
   &=\bigl[\bigl(u_{cp^r,m}\gamma^{(cp^r,m)},f^{(cp^r,m)}\bigr)\bigr]=\wt P_{cr^r,m},
   \end{split} \]
where the first equality comes from \eqref{a-1} and \eqref{+5}, the second from \eqref{a-2} and the third from \eqref{+4}. We conclude that $\wt P_{cr^r,m}$ belongs to the support $U_p(\wt P_{cp^{r-1},m})$, and the result follows from Proposition \ref{prop-Galois-Hecke-tilde-X}. \end{proof}

For the next proposition assume that $\cl O_{cp^m}^\times=\{\pm1\}$. 

\begin{prop} \label{T-operator}
Let $m\geq0$ and fix a prime $\ell\nmid Np^mc$ which is inert in $K$. Then
\[ T_\ell\bigl(\wt P_{c,m}\bigr)=\mathrm{tr}_{H_{c\ell p^m}(\boldsymbol{\mu}_{p^m})/H_{cp^m}(\boldsymbol{\mu}_{p^m})}\bigl(\wt P_{c\ell,m}\bigr) \]
in $\Div\bigl(\wt X_m\bigr)$.
\end{prop}

\begin{proof} Observe that 
\[ \bigl[\bigl(\hat\lambda_0u_{c,m}\gamma^{(c,m)},f^{(c,m)}\bigr)\bigr]=\bigl[\bigl(u_{c\ell,m}\gamma^{(c\ell,m)} b_{c\ell,m}b_{c,m}^{-1},f^{(c,m)}\bigr)\bigr]=\wt P_{c\ell,m}, \] 
where the first equality follows from \eqref{a-3} and \eqref{+5} and the second from \eqref{+4}. Hence $\wt P_{c\ell,m}$ belongs to the support of $T_\ell(\wt P_{c,m})$, and the result follows from Proposition \ref{T-operator-1}. \end{proof}

\subsection{Galois relations in compatible families} \label{compatible-Galois-families}

Set $G_\Q:=\Gal(\bar\Q/\Q)$ and let
\[ \epsilon_\cyc:G_\Q\longrightarrow\Z_p^\times \]
be the $p$-adic cyclotomic character. Since the restriction of $\epsilon_\cyc$ to $\Gal(\bar\Q/\Q(\sqrt{p^\ast}))$ takes values in $(\Z_p^\times)^2$ (recall that $\Q(\sqrt{p^\ast})$ is the unique quadratic extension of $\Q$ contained in the $p$-adic cyclotomic field), there is a unique continuous homomorphism
\[ \vartheta:\Gal\bigl(\bar\Q\big/\Q(\sqrt{p^\ast})\bigr)\longrightarrow\Z_p^\times/\{\pm1\} \]
such that $\vartheta^2=\epsilon_\cyc$. Fix $\sigma\in\Gal(H_{cp^m}(\boldsymbol\mu_{p^m})/H_{cp^m})$ with $m\geq1$. The Galois group $\Gal(H_{cp^m}(\boldsymbol\mu_{p^m})/H_{cp^m})$ is isomorphic to $\widehat{\cl O}_{cp^m}^\times/\cl O_{cp^m}^\times Z_m$ via the Artin map, so $\sigma$ can be represented by an element $x\in\widehat{\cl O}_{cp^m}^\times$ such that $x_\ell=1$ for $\ell\neq p$. Write $x_p=\alpha+p^m\beta$ with $\alpha\in\Z_p^\times$ and $\beta\in\cl O_K\otimes\Z_p$. The image $\bar\sigma$ of $\sigma$ via the natural map $\Gal(H_{cp^m}(\boldsymbol\mu_{p^m})/H_{cp^m})\rightarrow\Gal\bigl(\Q(\boldsymbol\mu_{p^m})/\Q(\sqrt{p^\ast})\bigr)$ is represented via the Artin map by $\Norm_{K/\Q}(x)$ and, by class field theory, we have $\Norm_{K_p/\Q_p}(x_p)^{-1}=\epsilon_\cyc(\bar\sigma)$. Hence $\epsilon_\cyc(\bar\sigma)\equiv\alpha^{-2}\pmod{p^m}$ and $\vartheta(\bar\sigma)\equiv\pm\alpha^{-1}\pmod{p^m}$. Thus, observing that $\vartheta(\sigma)\equiv1\pmod{p^m}$ if $\sigma\in\Gal(\bar\Q/\Q(\boldsymbol{\mu}_{p^m}))$, we may write 
\[ \wt P_{c,m}^\sigma=\langle\pm\vartheta(\sigma)\rangle\wt P_{c,m} \] 
for every $\sigma\in\Gal(\bar\Q/H_{cp^m})$, where, for any $a\in\Z_p^\times$, the symbol $\langle a\rangle$ denotes the diamod operator and, with a slight abuse of notation, we understand that $\Gal(\bar\Q/K^{\rm ab})$ acts trivially on $\wt P_{c,m}$ in the definite case. Since the action of $\langle-1\rangle$ on $\Div(\widetilde X_m)$ is trivial, it follows that for all $\sigma\in\Gal(\bar\Q/H_{cp^m})$ there is an equality
\begin{equation} \label{eq5}
\wt P_{c,m}^\sigma=\langle\vartheta(\sigma)\rangle\wt P_{c,m}
\end{equation}
in $\Div(\wt X_m)$.

\section{Hida theory on $\GL_2$} \label{section-Hida}

Throughout this paper we choose an (algebraic) isomorphism $\C\simeq\C_p$ where $\C_p$ is the completion of an algebraic closure of $\Q_p$, and view any subring of $\C_p$ as a subring of $\C$ via this fixed isomorphism.

\subsection{Ordinary Hecke algebras} \label{subsection-ordinary-hecke}

In the next few lines we use Shimura's notations $T(n)$ and $T(n,n)$ (with $n$ an integer) for the (abstract) Hecke operators defined as in \cite[\S 3.1--\S 3.3]{Sh} by double cosets.

Define $\Delta:=\boldsymbol\mu_{p-1}$ and $\Gamma:=1+p\Z_p$, so that we have a  
canonical isomorphism $\Z_p^\times\simeq\Gamma\times\Delta$.  
Define the two Iwasawa algebras
\[ \Lambda:=\cl O_F[\![\Gamma]\!],\qquad\tilde\Lambda:=\cl O_F[\![\Z_p^\times]\!] \]
where $F$ is a finite extension of $\Q_p$ (which will eventually contain the Fourier coefficients of our modular form $f$) and $\cl O_F$ is its ring of integers, so that we have a natural inclusion $\Lambda\subset\tilde\Lambda$. Finally, denote by $z\mapsto[z]$ the inclusions of group-like elements $\Gamma\hookrightarrow\Lambda$ and $\Z_p^\times\hookrightarrow\tilde\Lambda$.

For any ring $A$, any congruence subgroup $G\subset\SL_2(\Z)$ and any character $\psi:G\rightarrow\bar\Q_p^\times$ let $S_k(G,\psi,A)$ be the $A$-module of cusp forms of level $G$, weight $k$ and character $\psi$ with coefficients in $A$. We follow \cite{ho} for the presentation of Hida's Hecke algebras. Define
\[ \Gamma_{0,1}(N,p^m):=\Gamma_0(N)\cap\Gamma_1(p^m) \]
and write $\fr h_{k,m}$ for the Hecke algebra with $\cl O_F$-coefficients acting on $S_k\bigl(\Gamma_{0,1}(N,p^m),\C\bigr)$. The $\mathcal O_F$-algebra $\fr h_{k,m}$ is a finite product of complete local rings. Let $\fr h_{k,m}^\ord$ be the ordinary part of $\fr h_{k,m}$, i.e., the product of those local factors on which the image of $U_p$ is a unit. Define the Hecke algebras of weight $k$ as $\fr h_{k,\infty}:=\invlim_m\fr h_{k,m}$ and $\fr h_{k,\infty}^\ord:=\invlim_m\fr h^\ord_{k,m}$, the projective limits being taken with respect to the canonical maps. The $\cl O_F$-algebras $\fr h_{k,\infty}$ and $\fr h_{k,\infty}^\ord$ can be endowed with structures of $\tilde\Lambda$ and $\Lambda$-algebras in such a way that if $a$ is an integer prime to $Np$ and $T(a,a)_k$ denotes the image of $T(a,a)$ in $\fr h_{k,m}$ then the image of $[a]$ in $\fr h_{k,m}$ is the diamond operator $\langle a\rangle_k$ defined by the formula $T(a,a)_k=a^{k-2}\langle a\rangle_k$ (here we adopt the conventions of \cite{ho} rather than those of \cite{h-iwasawa}). The $\Lambda$-algebra $\fr h^\ord_{k,\infty}$ is finite and flat over $\Lambda$. For all weights $k,k'$ there is a unique isomorphism of algebras
\begin{equation} \label{univ-hecke-algebra}
\rho_{k,k'}:\fr h_{k,\infty}^\ord\overset{\simeq}{\longrightarrow}\fr h_{k',\infty}^\ord
\end{equation}
taking the images of $T(\ell)$ and $T(\ell,\ell)$ in $\fr h_{k,\infty}^\ord$ to the images of the same operators in $\fr h_{k',\infty}^\ord$. It will usually be convenient to identify the Hecke algebras $\fr h_{k,\infty}^\ord$ for all weights $k$ by means of the isomorphisms \eqref{univ-hecke-algebra}, so we simply set $\fr h_\infty^\ord:=\fr h_{2,\infty}^\ord$.

\subsection{New quotients} \label{section-p-adic-modular-forms}

We are especially interested in the $\C$-vector space $S_k^\new\bigl(\Gamma_{0,1}(N,p^m),\C\bigr)$ consisting of those forms which are new at all the primes dividing $N^-$. Write $\T_{k,m}$ for the image of $\fr h_{k,m}$ in the endomorphisms ring $\End\left(S_k^\new\bigl(\Gamma_{0,1}(N,p^m),\C\bigr)\right)$ and set
\[ \T_{k,\infty}:=\invlim_m\T_{k,m},\qquad\T_{k,m}^\ord:=e_m^\ord\cdot\T_{k,m},\qquad\T_{k,\infty}^\ord:=e^\ord\cdot\T_{k,\infty}=\invlim_m\T_{k,m}^\ord \]
where $e_m^\ord$ and $e^\ord$ are Hida's ordinary idempotent projectors. Isomorphisms \eqref{univ-hecke-algebra} yield isomorphisms of $\Lambda$-modules $\T_{k,\infty}^\ord\simeq\T_{k^\prime,\infty}^\ord$ for all weights $k,k'$, so we identify the algebras $\T_{k,\infty}^\ord$ for all weights $k$ and set $\T_\infty^\ord:=\T_{2,\infty}^\ord$.

\subsection{Maximal ideals of Hecke algebras} \label{section-max-min}

Following \cite[\S 1.4.4]{np}, we briefly describe the decompositions of our Hecke algebras into products of local components. Since $\fr h_{\infty}^\ord$ and $\T_{\infty}^\ord$ are finitely generated $\Lambda$-modules, they split as finite products
\begin{equation} \label{hecke-decompositions-eq}
\fr h_{\infty}^\ord=\prod_{\wt{\fr m}}\fr h_{\infty,\wt{\fr m}}^\ord,\qquad \T_{\infty}^\ord=\prod_{\fr m}\T_{\infty,\fr m}^\ord
\end{equation}
of their localizations at their (finitely many) maximal ideals $\wt{\fr m}$ and $\fr m$. Every summand appearing in these decompositions is a complete local ring, finite over $\Lambda$. If $\cl L$ is the fraction field of $\Lambda$ then $\fr h_{\infty,\wt{\fr m}}^\ord\otimes_\Lambda\cl L$ and $\T_{\infty,{\fr m}}^\ord\otimes_\Lambda\cl L$ are finite-dimensional artinian algebras over $\cl L$, so they are sums of local artinian algebras. If $\wt{\fr m}$ (respectively, $\fr m$) is a maximal ideal of $\fr h_\infty^\ord$ (respectively, of $\T_\infty^\ord$) then $\fr h_{\infty,\wt{\fr m}}^\ord\otimes_\Lambda\cl L$ (respectively, $\T_{\infty,\fr m}^\ord\otimes_\Lambda\cl L$) is a direct factor of $\fr h_\infty^\ord\otimes_\Lambda\cl L$ (respectively, of $\T_\infty^\ord\otimes_\Lambda\cl L$).
There are splittings of $\cl L$-algebras
\begin{equation} \label{deco}
\fr h_\infty^\ord\otimes_\Lambda\cl L=\Bigg(\prod_{i\in I}\cl F_i\Bigg)\bigoplus\cl M,\qquad\T_\infty^\ord\otimes_\Lambda\cl L=\Bigg(\prod_{j\in J}\cl K_j\Bigg)\bigoplus\cl N
\end{equation}
where $\cl F_i$ and $\cl K_j$ are finite field extensions of $\cl L$ while $\cl M$ and $\cl N$ are nonreduced. In Hida's terminology, the $\cl F_i$ and the $\cl K_j$ are called the \emph{primitive components} of $\fr h_{\infty}^\ord\otimes_\Lambda\cl L$ and $\T_\infty^\ord\otimes_\Lambda\cl L$, respectively. As explained in \cite[\S 3]{h-iwasawa}, one has $I=J$ and there are canonical isomorphisms
\begin{equation} \label{deco-iso}
\cl F_i\overset{\simeq}{\longrightarrow}\cl K_i
\end{equation}
for all $i\in I$. We say that $\cl F_i$ (respectively, $\cl K_i$) \emph{belongs to} $\wt{\fr m}$ (respectively, $\fr m$) if it is a direct summand of $\fr h_{\infty,\wt{\fr m}}^\ord\otimes_\Lambda\cl L$ (respectively, of $\T_{\infty,\fr m}^\ord\otimes_\Lambda\cl L$).

Now fix a modular form
\begin{equation} \label{fixed-modular-form}
f=\sum_{n\geq 1}a_nq^n\in S_k\bigl(\Gamma_0(Np),\omega^j,\cl O_F\bigr)
\end{equation}
with $j\equiv k\mod 2$, where $\omega:(\Z/p\Z)^\times\to\boldsymbol\mu_{p-1}$ is the Teichm\"uller character (here $\boldsymbol\mu_{p-1}$ is the group of $(p-1)$-st roots of unity). Assume that $f$ is a normalized eigenform for the Hecke operators $T_\ell$ (with $\ell\nmid Np$) and $U_\ell$ (with $\ell|Np$). Here, as before, $F$ is a finite extension of $\Q_p$ and $\cl O_F$ is its ring of integers. Let $\rho_f:G_\Q\rightarrow\GL_2(F)$ be the $p$-adic Galois representation attached to $f$ by Deligne.
\begin{assumption} \label{f-assumption}
Throughout this article we assume that
\begin{itemize}
\item[i)] the modular form $f$ is an \emph{ordinary $p$-stabilized newform} in the sense that $a_p\in\cl O_F^\times$ and the conductor of $f$ is divisible by $N$ (cf. \cite[Definition 2.5]{GS}), i.e., $f$ arises from a newform of level $N$ or $Np$ (this implies, in particular, that $\rho_f$ is ramified at all the primes dividing $N$);
\item[ii)] the residual representation $\bar\rho_f$ is $p$-distinguished and absolutely irreducible.
\end{itemize}
\end{assumption}
Here we recall that $\bar\rho_f$ is said to be \emph{$p$-distinguished} if its restriction to the decomposition group $G_{\Q_p}:=\Gal(\bar\Q_p/\Q_p)$ at $p$ can be put in the shape $\bar\rho_f|_{G_{\Q_p}}=\bigl(\begin{smallmatrix}\varepsilon_1&\ast\\0&\varepsilon_2\end{smallmatrix}\bigr)$ for characters $\varepsilon_1\not=\varepsilon_2$ (see, e.g., \cite[\S 2]{ghate}).

Duality between modular forms and Hecke algebras yields morphisms
\[ \theta_f:\T_{\infty}^\ord\longrightarrow\cl O_F,\qquad\tilde\theta_f:\fr h_{\infty}^\ord\longrightarrow\cl O_F \]
such that $\theta_f$ factors through $\T^\ord_{k,1}$ and is characterized by $\theta_f(T(\ell))=a_\ell$ for all primes $\ell$, $\theta_f([\delta])=\delta^{k+j-2}$ for $\delta\in\Delta$, $\theta_f([\gamma])=\gamma^{k-2}$ for $\gamma\in\Gamma$, while $\tilde\theta_f$ is the composition of the canonical projection $\fr h_{\infty}^\ord\rightarrow\T_{\infty}^\ord$ with $\theta_f$. Let $\wt{\fr m}_f$ and $\fr m_f$ be the maximal ideals corresponding to the unique local factors of $\fr h_{\infty}^\ord$ and $\T_{\infty}^\ord$ through which $\tilde\theta_f$ and $\theta_f$ factor. Since $f$ satisfies i) in Assumption \ref{f-assumption}, we can consider the unique primitive component $\cl K$ of $\fr h_{\infty,{\wt{\fr m}_f}}^\ord\otimes_\Lambda\cl L$ appearing in \eqref{deco} to which $f$ \emph{belongs} in the sense of \cite[Corollary 3.7]{h-iwasawa} or \cite[pp. 316--317]{Hida-annals} (see also \cite[p. 95]{ho} for the more general type of arithmetic groups we are working with here). Thanks to isomorphisms \eqref{deco-iso}, there is a unique primitive component of $\T_{\infty,{{\fr m}_f}}^\ord\otimes_\Lambda\cl L$ (appearing in \eqref{deco}) which is isomorphic to $\cl K$: denote this component by the same symbol $\mathcal K$. Finally, let $\mathcal R$ be the integral closure of $\Lambda$ in $\cl K$.

\begin{prop} \label{A-ring-prop}
The ring $\cl R$ is a complete local noetherian domain which is finitely generated as a $\Lambda$-module.
\end{prop}
\begin{proof} See, e.g., \cite[Theorem 4.3.4]{hs}. \end{proof}

Observe that $\cl R$ is an $\fr h_{\infty,{\wt{\fr m}_f}}^\ord$-algebra. Indeed, the field $\cl K$ is an $\fr h_{\infty,{\wt{\fr m}_f}}^\ord$-algebra; moreover,  $\fr h_{\infty,{\wt{\fr m}_f}}^\ord$ identifies by \eqref{hecke-decompositions-eq} with a $\Lambda$-subalgebra of $\fr h_\infty^\ord$, hence it is integral over $\Lambda$ by \eqref{univ-hecke-algebra}, and this implies that $\fr h_{\infty,{\wt{\fr m}_f}}^\ord$ preserves the subring $\cl R$ of $\cl K$. Analogous arguments show that $\cl R$ is a $\T_{\infty,{{\fr m}_f}}^\ord$-algebra. Now consider the composition
\[ f_\infty:\fr h_\infty^\ord\;\longepi\;\fr h_{\infty,{\wt{\fr m}_f}}^\ord\longrightarrow\cl R \]
in which the first arrow is the natural projection and the second arrow is the structure map of $\cl R$ as an $\fr h_{\infty,{\wt{\fr m}_f}}^\ord$-algebra. The map $f_\infty$ is uniquely determined by the primitive component $\cl K$ to which $f$ belongs.

\begin{defi} \label{hida-family-defi}
The local $\Lambda$-algebra $\fr h_{\infty,{\wt{\fr m}_f}}^\ord$ is the \emph{Hida family} of $f$ and $\cl R$ is the \emph{branch} of the Hida family on which $f$ lives. We call the map $f_\infty$ the \emph{primitive morphism} associated with $f$.
\end{defi}

\subsection{Critical characters} \label{critical-subsection}

Factor $\epsilon_\cyc:G_\Q\to\Z_p^\times$ as a product $\epsilon_\cyc=\epsilon_\tame\epsilon_\wild$ with $\epsilon_\tame:G_\Q\rightarrow\boldsymbol\mu_{p-1}$ and
$\epsilon_\wild:G_\Q\rightarrow\Gamma$ and define the \emph{critical character} $\Theta:G_\Q\to\Lambda^\times$ by
\[ \Theta:=\epsilon_\tame^{(k+j-2)/2}\bigl[\epsilon_\wild^{1/2}\bigr] \]
where $\epsilon_\wild^{1/2}$ is the unique square root of $\epsilon_\wild$ taking values in $\Gamma$. If $i\in\Z/(p-1)\Z$ then the idempotent
\[ e_i:=\frac{1}{p-1}\sum_{\delta\in\Delta}\omega^{-i}(\delta)[\delta]\in \cl O_F[\![\Z_p^\times]\!] \]
satisfies the relation
\begin{equation} \label{fund-rel-in-Lambda}
e_i[\zeta]=\zeta^ie_i\qquad\text{for all $\zeta\in\boldsymbol\mu_{p-1}$}.
\end{equation}
Since $f(e_i)=0$ if $i\neq k+j-2$, we have $e_{k+j-2}(\fr h^\ord_{k,\infty})_{\wt{\fr m}_f}=(\fr h_{k,\infty}^\ord)_{\wt{\fr m}_f}$. Therefore
\begin{equation} \label{e-T-eq}
e_{k+j-2}(\T^\ord_{k,\infty})_{\fr m_f}=(\T_{k,\infty}^\ord)_{{\fr m}_f},
\end{equation}
and it follows that in $(\T^\ord_{k,\infty})_{\fr m_f}$ we have
\[ [\epsilon_\tame(\sigma)]=\epsilon_\tame^{k+j-2}(\sigma) \]
for all $\sigma\in G_\Q$. Furthermore, by definition of $\Theta$, in $(\T^\ord_{k,\infty})_{\fr m_f}$ there are also equalities
\[ \Theta^2(\sigma)=\epsilon_\tame^{k+j-2}(\sigma)[\epsilon_\wild(\sigma)]=[\epsilon_\cyc(\sigma)] \]
for all $\sigma\in G_\Q$.

\subsection{Arithmetic primes and Galois representations} \label{arithmetic-primes-subsec}

For every integer $m\geq0$ denote by $X_{0,1}(N,p^m)$ the compactified modular curve of level structure $\Gamma_{0,1}(N,p^m)$, viewed as a scheme over $\Q$, by ${\rm Jac}\bigl(X_{0,1}(N,p^m)\bigr)$ its Jacobian variety and by ${\rm Ta}_p\bigl({\rm Jac}\bigl(X_{0,1}(N,p^m)\bigr)\bigr)$ the $p$-adic Tate module of the Jacobian. As in \cite[\S 2.1]{ho}, for every integer $m\geq1$ we define the $\fr h_\infty^\ord$-modules
\[ {\rm Ta}_{p,m}^\ord:=e^\ord_m\Big({\rm Ta}_p\bigl({\rm Jac}\bigl(X_{0,1}(N,p^m)\bigr)\bigr)\otimes_{\Z_p}\cl O_F\Big),\qquad {\bf Ta}^\ord:=\invlim_m{\rm Ta}_{p,m}^\ord, \]
\[ {\bf Ta}^\ord_{\wt{\fr m}_f}:={\bf Ta}^\ord\otimes_{\fr h_\infty^\ord}{\fr h^\ord_{\infty,\wt{\fr m}_f}},\qquad{\bf T}:= {\bf Ta}^\ord_{\wt{\fr m}_f}\otimes_{\fr h^\ord_{\infty,\wt{\fr m}_f}}\cl R. \]
All these modules are endowed with $\fr h_\infty^\ord$-linear actions of the Galois group $G_\Q$. The $\cl R$-module ${\bf T}$ is free of rank two. Let $\cl R^\dagger$ denote $\cl R$ viewed as a module over itself with $G_\Q$ acting through $\Theta^{-1}$ and define the \emph{critical twist} of ${\bf T}$ to be the $G_\Q$-module
\[ {\bf T}^\dagger:={\bf T}\otimes_\cl R\cl R^\dagger={\bf Ta}^\ord_{\wt{\fr m}_f}\otimes_{{\fr h^\ord_{\infty,\wt{\fr m}_f}}}\cl R^\dagger. \]
The $G_\Q$-module ${\bf T}$ is unramified outside $Np$ and the arithmetic Frobenius at a prime $\ell\nmid Np$ acts with characteristic polynomial $X^2-T_\ell X+[\ell]\ell$. For a proof of these facts see, e.g., \cite[Proposition 2.1.2]{ho} and \cite[Th{\'e}or{\`e}me 7]{mt}.

Write $\fr m_\cl R$ for the maximal ideal of the local ring $\cl R$ and set
\[ \mathbb F_\cl R:=\cl R/\fr m_\cl R,\qquad \mathbb F_{\fr h^\ord_{\infty,\wt{\fr m}_f}}:=\fr h^\ord_{\infty,\wt{\fr m}_f}\big/\wt{\fr m}_f\fr h^\ord_{\infty,\wt{\fr m}_f} \]
for the residue fields of $\cl R$ and $\fr h^\ord_{\infty,\wt{\fr m}_f}$. Since $\cl R$ is finite over $\Lambda$ by Proposition \ref{A-ring-prop}, the map $\fr h^\ord_{\infty,\wt{\fr m}_f}\rightarrow\cl R$ is also finite, hence integral. Thus $\mathbb F_\cl R$ is naturally a finite extension of $\mathbb F_{\fr h^\ord_{\infty,\wt{\fr m}_f}}$ and hence of $\mathbb F_p$. The next result will be exploited in Section \ref{section-indefinite}.
\begin{prop} \label{T-residue-prop}
The residual $G_\Q$-representation ${\bf T}/\fr m_\cl R{\bf T}$ is equivalent up to finite base change to the residual representation $\bar\rho_f$ of $f$. In particular, it is absolutely irreducible.
\end{prop}
\begin{proof} First of all, if the claimed equivalence of representations is true then the absolute irreducibility follows from condition ii) in Assumption \ref{f-assumption}. With notation as above, there is a canonical isomorphism of $G_\Q$-modules
\[ {\bf T}/\fr m_\cl R{\bf T}\simeq\bigl({\bf Ta}^\ord_{\wt{\fr m}_f}\big/\wt{\fr m}_f{\bf Ta}^\ord_{\wt{\fr m}_f}\bigr)\otimes_{\mathbb F_{\fr h^\ord_{\infty,\wt{\fr m}_f}}}\mathbb F_{\cl R}. \]
As explained in \cite[p. 251]{h-iwasawa}, all modular forms in the Hida family $\fr h^\ord_{\infty,\wt{\fr m}_f}$ have residual representation equivalent to $\bar\rho_f$. On the other hand, by \cite[Proposition 2.1.2]{ho}, the local ring $\fr h^\ord_{\infty,\wt{\fr m}_f}$ is a Gorenstein $\Lambda$-algebra, and then \cite[\S 9]{h-galois} (see also \cite[\S 3]{mt}) shows that the residual $G_\Q$-representation ${\bf Ta}^\ord_{\wt{\fr m}_f}\big/\wt{\fr m}_f{\bf Ta}^\ord_{\wt{\fr m}_f}$ is equivalent (up to finite base change) to $\bar\rho_f$. \end{proof}

Now recall that $\Gamma:=1+p\Z_p$. If $A$ is a finitely generated commutative $\Lambda$-algebra then a homomorphism of $\cl O_F$-algebras $\kappa:A\rightarrow\bar\Q_p$ is said to be \emph{arithmetic} if the composition of $\kappa$ with the canonical map $\Gamma\rightarrow A^\times$ has the form $\gamma\mapsto\psi(\gamma)\gamma^{r-2}$ for some integer $r\geq2$ and some finite order character $\psi$ of $\Gamma$. The kernel of an arithmetic homomorphism, which is a prime ideal of $A$, is said to be an \emph{arithmetic prime} of $A$. If $\fr p$ is an arithmetic prime of $A$ and, as usual, $A_\fr p$ is the localization of $A$ at $\fr p$ then the residue field $F_\fr p:=A_\fr p/\fr p A_\fr p$ is a finite extension of $F$. The composition $\Gamma\rightarrow A^\times\rightarrow F_\mathfrak p^\times$ has the form $\gamma\mapsto \psi_\mathfrak p(\gamma)\gamma^{r_\mathfrak p-2}$ for a finite order character $\psi_\mathfrak p:\Gamma\rightarrow F_\mathfrak p^\times$ and an integer $r_\mathfrak p$, and we call $\psi_\mathfrak p$ and $r_\mathfrak p$ the \emph{wild character} and the \emph{weight} of the arithmetic prime $\mathfrak p$, respectively. The homomorphisms of $\cl O_F$-algebras $\tilde\theta_f$ and $\theta_f$ that were attached in \S \ref{section-max-min} to the modular form $f$ are arithmetic.

Let $\fr p$ be an arithmetic prime of $\cl R$ of weight $r_\fr p$ and character $\psi_\fr p$, and set
\begin{equation} \label{m-integer-eq}
m_\fr p:=\max\bigl\{1,\ord_p\bigl({\rm cond}(\psi_\fr p)\bigr)\bigr\}.
\end{equation}
By \cite[Corollary 1.3]{h-galois}, the morphism obtained by composing the maps
\[ \fr h^\ord_\infty\xrightarrow{f_\infty}\cl R\longrightarrow F_\fr p \]
factors through $\fr h_{r_\fr p}^\ord$ and determines, by duality, an ordinary $p$-stabilized newform
\begin{equation} \label{form-ass-arith-prime}
f_\fr p=\sum_{n\geq 1}a_n(f_\fr p)q^n\in S_{r_\fr p}\bigl(\Gamma_{0,1}(N,p^{m_\fr p}),\zeta_\fr p,F_\fr p\bigr)
\end{equation}
where, for simplicity, we put $\zeta_\fr p:=\psi_\fr p\omega^{k+j-r_\fr p}$.

Denote by $V(f_\fr p)$ the $G_\Q$-representation over $F_\fr p$ attached to $f_\fr p$ by Deligne. Thanks to a result of Ribet (\cite[Theorem 2.3]{ri2}), it is known that $V(f_\fr p)$ is (absolutely) irreducible. Define the $G_\Q$-modules ${\bf T}_\fr p:={\bf T}\otimes_{\cl R}\cl R_\fr p$ and $V_\fr p:={\bf T}\otimes_\cl R F_\fr p={\bf T}_\fr p/\fr p{\bf T}_\fr p$ and their critical twists ${\bf T}_\fr p^\dagger:={\bf T}^\dagger\otimes_{\cl R}\cl R_\fr p$ and $V_\fr p^\dagger:={\bf T}^\dagger\otimes_\cl R F_\fr p$. Then $V_\fr p^\dagger$ is a twist of the classical representation attached to $f_\fr p$. See \cite[\S 2.1]{ho} and \cite[\S 1.5 and \S 1.4]{np} for details. Let now $v$ be a place of $\bar\Q$ above $p$, let $D_v\subset G_\Q$ be a decomposition group at $v$ and let $I_v\subset D_v$ be the inertia subgroup. Denote by $\eta_v:D_v/I_v\rightarrow\cl R^\times$ the character defined by sending the arithmetic Frobenius to $U_p$. Then \cite[Proposition 2.4.1]{ho} ensures that there is a short exact sequence of $\cl R[D_v]$-modules
\[ 0\longrightarrow F^+_v({\bf T})\longrightarrow {\bf T}\longrightarrow F_v^-({\bf T})\longrightarrow 0 \]
where $F^+_v({\bf T})$ and $F^-_v({\bf T})$ are free $\cl R$-modules of rank one and $D_v$ acts on $F^+({\bf T})$ (respectively, on $F_v^-({\bf T})$) through $\eta_v^{-1}\epsilon_\cyc[\epsilon_\cyc]$ (respectively, through $\eta_v$). Furthermore, if $M$ denotes one of ${\bf T}_\fr p$, $V_\fr p$, ${\bf T}_\fr p^\dagger$ and $V_\fr p^\dagger$ then twisting by $\Theta^{-1}$ and tensoring the above exact sequence by $\cl R_\fr p$ or $F_\fr p$ yields another short exact sequence of $D_v$-modules
\[ 0\longrightarrow F_v^+(M)\longrightarrow M\longrightarrow F_v^-(M)\longrightarrow 0, \]
where $F_v^+(M)$ and $F^-_v(M)$ are free modules of rank one over either $\cl R_\fr p$ or $F_\fr p$, depending on whether $M\in\bigl\{{\bf T}_\fr p,{\bf T}_\fr p^\dagger\bigr\}$ or $M\in\bigl\{V_\fr p,V_\fr p^\dagger\bigr\}$, respectively. Finally, the Galois group $G_\Q$ acts on $F_v^+(M)$ and $F^-_v(M)$ either by $\eta_v^{-1}\epsilon_\cyc[\epsilon_\cyc]$ and $\eta_v$ or by $\Theta^{-1}\eta_v^{-1}\epsilon_\cyc[\epsilon_\cyc]$ and $\Theta^{-1}\eta_v$, depending on whether $M\in\bigl\{{\bf T}_\fr p,V_\fr p\bigr\}$ or $M\in\bigl\{{\bf T}_\fr p^\dagger,V_\fr p^\dagger\bigr\}$, respectively.

\subsection{Selmer groups} \label{section-Selmer-groups}

We recall the definitions of the various Selmer groups that are relevant for our purposes. The reader may also wish to consult \cite[\S 2.4]{ho} and \cite[\S 2.1]{np}.

Let $L$ be a number field and for any prime $v$ of $L$ let $L_v$ the completion of $L$ at $v$ and $L_v^{\rm unr}$ the maximal unramified extension of $L_v$. Let $M$ be one of the left $R[G_\Q]$-modules ${\bf T}$, ${\bf T}^\dagger$, ${\bf T}_\fr p$, ${\bf T}_\fr p^\dagger$, $V_\fr p$, $V_\fr p^\dagger$ where $R$ denotes the ring $\cl R$ in the first two cases, the ring $\cl R_\fr p$ in the middle two cases and the field $F_\fr p$ in the last two cases. Fix a prime $v$ of $L$ and define the \emph{Greenberg local subgroup} at $v$ by
\[ H^1_{\rm Gr}(L_v,M):=\begin{cases} \ker\bigl(H^1(L_v,M)\longrightarrow H^1(L_v^{\rm unr},M)\bigr)&\text{if $v\nmid p$},\\[3mm] \ker\bigl(H^1(L_v,M)\longrightarrow H^1(L_v,F^-_v(M))\bigr)&\text{if $v\mid p$}.\end{cases} \]
Then the \emph{Greenberg Selmer group} is by definition the group
\[ \Sel_{\rm Gr}(L,M):=\ker\Big(H^1(L,M)\longrightarrow\prod_v H^1(L_v,M)/H^1_{\rm Gr}(L_v,M)\Big). \]
Let ${\bf A}^\dagger:=\Hom_{\Z_p}\bigl({\bf T}^\dagger,{\boldsymbol \mu}_{p^\infty}\bigr)$. For $M={\bf T}^\dagger$, ${\bf A^\dagger}$ or $V_\fr p^\dagger$ one can also consider the \emph{Nekov\'a\v{r} Selmer group} $\wt H^1_f(L,M)$, which for $M={\bf T}^\dagger$ or $V_\fr p^\dagger$ sits in the short exact sequence
\begin{equation} \label{nekovar}
0\longrightarrow\bigoplus_{v\mid p}H^0\bigl(L_v,F^-_v(M)\bigr)\longrightarrow\wt H^1_f(L,M)\longrightarrow\Sel_{\rm Gr}(L,M)\longrightarrow 0,
\end{equation}
the direct sum being extended over the primes of $L$ above $p$. The reader is referred to \cite[Ch. 6]{Ne-selmer} for definitions and to \cite[Lemma 9.6.3]{Ne-selmer} for a proof of \eqref{nekovar}.

If $M=V_\fr p$ or $V_\fr p^\dagger$ one has the \emph{Bloch--Kato Selmer group} $H^1_f(L,M)$ as well, whose definition (in terms of Fontaine's ring $B_{cris}$) can be found in \cite[\S 3 and \S 5]{BK}. If $M=V_\fr p^\dagger$ and $\fr p$ has even weight then, by \cite[Proposition 12.5.9.2]{Ne-selmer}, this group fits into the short exact sequence
\begin{equation} \label{bloch-kato}
0\longrightarrow\bigoplus_{v\mid p}H^0\bigl(L_v,F^-_v(V_\fr p^\dagger)\bigr)\longrightarrow\wt H^1_f\bigl(L,V_\fr p^\dagger\bigr)\longrightarrow H^1_f\bigl(L,V_\fr p^\dagger\bigr)\longrightarrow 0.
\end{equation}
An arithmetic prime $\fr p$ of $\cl R$ is said to be \emph{exceptional} if $r_{\fr p}=2$, the character $\psi_\fr p$ is trivial and the image of $U_p$ under the map $\cl R\rightarrow F_\fr p$ is equal to $\pm1$. The relations between the Selmer groups that we introduced above are then summarized by exact sequences \eqref{nekovar} and \eqref{bloch-kato} and the following result.

\begin{prop} \label{prop-Selmer}
\begin{enumerate}
\item If $\fr p$ is a non-exceptional arithmetic prime of $\cl R$ then $\wt H^1_f\bigl(L,V_\fr p^\dagger\bigr)$ and $\Sel_{\rm Gr}\bigl(L,V_\fr p^\dagger\bigr)$ are isomorphic.
\item If $\fr p$ is an arithmetic prime of even weight then $H^1_f\bigl(L,V_\fr p^\dagger\bigr)$ and $\Sel_{\rm Gr}\bigl(L,V_\fr p^\dagger\bigr)$ are isomorphic.
\end{enumerate}
\end{prop}
\begin{proof} The first assertion is \cite[(22)]{ho}, which follows from \cite[Lemma 2.4.4]{ho}. The last claim is \cite[(23)]{ho}, which is immediate from \eqref{nekovar} and \eqref{bloch-kato}. \end{proof}

\section{Hida theory on quaternion algebras}

Recall the quaternion algebra $B$ over $\Q$ of discriminant $N^-$ introduced at the beginning of the paper. In the following (in slight conflict with the conventions of Section \ref{section-Hida}) we use Hida's notations $T(n)$ and $T(n,n)$ for the (abstract) Hecke operators defined as in \cite[p. 309]{Hida-annals} by double cosets. In this section we always assume that $B$ is a division algebra, the theory for the split case $B\simeq\M_2(\Q)$ having been considered earlier.

Fix an integer $m\geq 0$. Recall that $\Div(\wt X_m)$ and $\Div^0(\wt X_m)$ denote the groups of divisors and of degree zero divisors, respectively, on $\wt X_m$. Let ${\rm Pr}(\wt X_m)$ be the group of principal divisors on $\wt X_m$ and define, as usual, the Picard groups
\[ \Pic(\wt X_m):=\Div(\wt X_m)/{\rm Pr}(\wt X_m),\qquad\Pic^0(\wt X_m):=\Div^0(\wt X_m)/{\rm Pr}(\wt X_m). \]
The groups $\Pic(\wt X_m)$ and $\Pic^0(\wt X_m)$ are connected by the short exact sequence
\begin{equation} \label{exact-sequence-Pic}
0\longrightarrow \Pic^0(\wt X_m)\longrightarrow\Pic(\wt X_m)\xrightarrow{\deg}\Z\longrightarrow 0,
\end{equation}
where $\deg$ is the degree map.

\subsection{Hecke modules in the definite case} \label{section-hecke-mod}

Assume that $B$ is \emph{definite} and fix an integer $m\geq 0$. As pointed out in \cite[\S 1.4]{bd*} and \cite[\S 4]{gross-2}, in this case $\Pic(\wt X_m)$ can be identified with the free abelian group $\Z\bigl[U_m\backslash\widehat B^\times/B^\times\bigr]$ on the finite set of double cosets $U_m\backslash\widehat B^\times/B^\times$ and $\Pic^0(\wt X_m)$ corresponds to the degree zero elements in this group. With notation as in \S \ref{section-definite-shimura}, the sequence $m\mapsto\tilde h(m)$ is unbounded because the same is true, by \cite[Theorem 16]{Pi}, of the sequence $m\mapsto h(m)$, and $h(m)\leq\tilde h(m)$. Hence the ranks of the free abelian groups $\Pic(\wt X_m)$ and $\Pic^0(\wt X_m)$ are unbounded as $m$ varies. Now define
\[ J_m:=\Pic(\wt X_m) \otimes_\Z \cl O_F, \qquad J^0_m :=\Pic^0(\wt X_m) \otimes_\Z \cl O_F. \]
Tensoring \eqref{exact-sequence-Pic} with $\cl O_F$ over $\Z$ yields a short exact sequence of $\cl O_F$-modules
\begin{equation} \label{exact-seq-J}
0\longrightarrow J^0_m \longrightarrow J_m \xrightarrow{\deg}\cl O_F\longrightarrow0.
\end{equation}
By what has been said a few lines before, there is an identification of $\cl O_F$-modules
\begin{equation} \label{J-m-eq}
J_m=\cl O_F\bigl[U_m\backslash\widehat B^\times/B^\times\bigr],
\end{equation}
which will usually be viewed as an equality. The abelian group $\Pic(\wt X_m)$ is finitely generated, so, by \cite[Theorem 7.11]{mat}, it follows that $\End_\Z\bigl(\Pic(\wt X_m)\bigr)\otimes_\Z\cl O_F$ is canonically isomorphic to $\End_{\cl O_F}(J_m)$. A similar remark also applies to $J_m^0$.

In \cite[\S 1.4]{bd*} (see also \cite[\S 4]{gross-2}) it is explained how equality \eqref{J-m-eq} can be used to define an Hecke algebra with $\cl O_F$-coefficients, which we denote by $\B_m$, acting (via Brandt matrices) on $J_m$ and $J_m^0$.

Let us now assume $m\geq1$. Since $\B_m$ is a finitely generated $\cl O_F$-module, we can define an idempotent $e_m^\ord\in\B_m$ attached to the Hecke operator $U_p$ and introduce the ordinary parts $\mathcal X^\ord:=e_m^\ord\cdot\mathcal X$ for $\mathcal X\in\bigl\{\B_m,J_m,J_m^0\bigr\}$. So $J_m^\ord$ is a 
$\B_m^\ord$-module. Since $U_p$ has degree $p$, exact sequence \eqref{exact-seq-J} implies that there is an isomorphism of $\B_m^\ord$-modules
\[ J_m^\ord\simeq J_m^{0,\ord}. \]
The maps $\wt\alpha_m:\wt X_m\rightarrow \wt X_{m-1}$ induce (by covariant functoriality) maps $\wt\alpha_{m,\ast}:J_m\rightarrow J_{m-1}$ and $\wt\alpha_{m,\ast}:J_m^0\rightarrow J_{m-1}^0$ preserving the ordinary parts, so one can consider the projective limits
\[ J_\infty:=\invlim_m J_m,\qquad J^0_\infty:=\invlim_m J_m^0,\qquad J^\ord_\infty:=\invlim_m J_m^{\ord} \]
with respect to these maps. Define
\[\B_\infty:=\invlim_m \B_m,\qquad \B_\infty^\ord:=\invlim_m \B_m^\ord\]
with respect to the canonical maps. Then $J_\infty^\ord$ is a $\B_\infty^\ord$-module, while $J_\infty$ and $J^0_\infty$ are $\B_\infty$-modules.

\subsection{Hecke modules in the indefinite case}

Now suppose that $B$ is \emph{indefinite} and fix an integer $m\geq0$. Then $\Pic^0(\wt X_m)$ can be identified with the Jacobian variety ${\rm Jac}(\wt X_m)$ of $\wt X_m$, which is an abelian variety defined over $\Q$ whose dimension equals the genus of $\wt X_m$, while \eqref{exact-sequence-Pic} shows that $\Pic(\wt X_m)$ is an extension of $\Z$ by $\Pic^0(\wt X_m)$. More precisely, $\Pic(\wt X_m)$ identifies with the $\bar\Q$-points of the Picard scheme of $\wt X_m$ and $\Pic^0(\wt X_m)$ with the identity component of this scheme. If $L$ is an extension of $\Q$ then we denote by $\Pic(\wt X_m)(L)$ and $\Pic^0(\wt X_m)(L)$ the $L$-rational points of $\Pic(\wt X_m)$ and $\Pic^0(\wt X_m)$, respectively. Unlike what was done in the definite case, in the indefinite case by $J_m$ and $J_m^0$ we mean the functor from the category 
of $\Q$-algebras to the category of $\cl O_F$-modules
 which associate with any field extension $L/\Q$ the $\cl O_F$-modules
\[ J_m(L):=\Pic(\wt X_m)(L)\otimes_\Z\cl O_F,\qquad J_m^0(L):={\rm Jac}(\wt X_m)(L)\otimes_\Z\cl O_F, \]
respectively. These modules are endowed with a canonical action of a Hecke $\cl O_F$-algebra $\B_m$, which is induced by the Hecke action on divisors via Albanese functoriality. 

Let $L$ be an algebraic extension of $\Q$ and set $G_L:=\Gal(\bar\Q/L)$. Since ${\rm Jac}(\wt X_m)$ is defined over $\Q$, the $\cl O_F$-module $J_m^0(L)$ has a natural left $G_L$-action and so is canonically a left $\B_m[G_L]$-module.
Furthermore, the ordinary part
\[ J_m^{0,\ord}(L):=e_m^\ord\cdot J_m^0(L) \]
inherits a canonical structure of left $\B_m^\ord[G_L]$-module, where $e^\ord_m$ denotes as above Hida's ordinary projector.

Suppose that $L$ is a number field. Tensoring \eqref{exact-sequence-Pic} (with values in $L$) by $\cl O_F$ over $\Z$ yields, as above, a short exact sequence of left $\cl O_F[G_L]$-modules
\begin{equation} \label{exact-seq-J-II}
0\longrightarrow J^0_m(L)\longrightarrow J_m(L)\xrightarrow{\deg}\cl O_F\longrightarrow0.
\end{equation}
Since $J_m(L)$ is a finitely generated $\cl O_F$-module, it makes sense to  introduce the idempotent $e_m^\ord$ in $\B_m$ attached to the Hecke operator $U_p$ and define the ordinary part of $J_m(L)$ to be
\[ J_m^\ord(L):=e_m^\ord\cdot J_m(L). \]
Now observe that, since $U_p$ has degree $p$, sequence \eqref{exact-seq-J-II} shows that
\begin{equation} \label{J-m-ord-eq}
J_m^\ord(L)=J_m^{0,\ord}(L)
\end{equation}
for every $m\geq0$ and every number field $L$. Let now $L/\Q$ be an algebraic field extension and write it as a direct limit $L=\dirlim L_i$ of finite extensions. Since the $J_m(L_i)$ have ordinary parts which are compatible with direct limits, we can define the ordinary part of $J_m(L)$ as
\[ J_m^\ord(L):=\dirlim_i J_m^\ord(L_i). \]
Thanks to \eqref{J-m-ord-eq}, and the fact that $J_m^0(L)=\dirlim J_m^0(L_i)$ because direct limits commute with tensor products, we see that
\[ J_m^\ord(L)=J_m^{0,\ord}(L) \]
for every $m\geq0$ and every extension $L/\Q$. Thus $J_m^\ord(L)$ is a $\B_m^\ord$-module for every $m$ and every extension $L/\Q$, where $\B_m^\ord:=e^\ord\cdot\B_m$.

As above, for every extension $L/\Q$ and every $m\geq 1$ we can define by covariant functoriality maps $\wt\alpha_{m, \ast}:J_m(L)\rightarrow J_{m-1}(L)$ and $\wt\alpha_{m,\ast}:J_m^0(L)\rightarrow J_{m-1}^0(L)$ which preserve the ordinary parts, so we can form the projective limits
\[ J_\infty(L):=\invlim_m J_m(L),\qquad J^0_\infty(L):=\invlim_m J_m^0(L),\qquad J^\ord_\infty(L):=\invlim_m J_m^{\ord}(L) \]
with respect to these maps. Form the $\cl O_F$-modules
\[ \B_\infty:=\invlim_m\B_m,\qquad\B_\infty^\ord:=\invlim_m\B^\ord_m. \]
In particular, $J_\infty^0(L)$ is a left $\B_\infty[G_L]$-module and $J_\infty^\ord(L)$ is a left $\B_\infty^\ord[G_L]$-module. Write $\Ta_p\bigl({\rm Jac}(\wt X_m)\bigr)$ for the $p$-adic Tate module of ${\rm Jac}(\wt X_m)$ and define
\[ T_m:=\Ta_p\bigl({\rm Jac}(\wt X_m)\bigr)\otimes_{\Z_p}\cl O_F,\qquad T_\infty:=\invlim_mT_m \]
where the inverse limit is with respect to the canonical projection maps. Then $T_\infty$ and $T_m$ are $\B_\infty$ and $\B_m$-modules, respectively, and one can define the ordinary parts
\[ T_m^\ord:=e_m^\ord \cdot T_m,\qquad T_\infty^\ord:=e^\ord \cdot T_\infty^\ord, \]
which are left $\B_m^\ord[G_\Q]$ and $\B^\ord_\infty[G_\Q]$-modules, respectively.

\subsection{Jacquet--Langlands correspondence} \label{JL-subsection}

In order to simplify notations, set $\T_\star:=\T_{2,\star}$ and $\T_\star^\ord:=\T_{2,\star}^\ord$ for $\star$ an integer $m\geq0$ or the symbol $\infty$: this is
an extension of the convention introduced at the end of \S \ref{section-p-adic-modular-forms}.

The Jacquet--Langlands correspondence (see \cite[\S 2.4]{Hida-HMFIT}) gives an isomorphism of $\cl O_F$-algebras
\begin{equation} \label{JL-eq}
{\rm JL}_m:\T_m\overset{\simeq}{\longrightarrow}\B_m
\end{equation}
taking $T(\ell)_2$ and $T(\ell,\ell)_2$ to the analogous operators in $\B_m$. We define a continuous structure of $\tilde\Lambda$-algebra on $\B_\infty$ and $\B_\infty^\ord$ as in \cite[\S 3.2.8]{Hida-HMFIT}, and denote $[z]\mapsto\langle z\rangle$ the image of group-like elements of $\tilde\Lambda$. We normalize this action so that if $n$ is an integer coprime with $Np$ then $T(n,n)=\langle n\rangle$ as operators in $\B_m$ (as in the case of elliptic modular forms, we adopt the normalization in \cite{ho} instead of the one usually found in Hida's papers).

\begin{prop}[Jacquet--Langlands] \label{JL1}
There is a canonical isomorphism of $\tilde\Lambda$-algebras
\[ {\rm JL}_\infty^\ord:\T_\infty^\ord\overset{\simeq}{\longrightarrow}\B^\ord_\infty. \]
\end{prop}

\begin{proof} For $m\geq1$ there are commutative diagrams
\[ \xymatrix@C=35pt{\T_m^\ord\ar[r]^{{\rm JL}_m^\ord}\ar[d]& \B_m^{\ord}\ar[d]\\\T_{m-1}^\ord\ar[r]^{{\rm JL}_{m-1}^\ord}& \B_{m-1}^{\ord}} \]
where the vertical arrows are the canonical projections. The claim of the proposition then follows by taking inverse limits and noticing that the two $\tilde\Lambda$-algebra structures agree on the set of integers prime to $Np$, so (by a continuity argument) they must be equal. \end{proof}

In light of \eqref{JL-eq} and Proposition \ref{JL1}, from here on we identify the algebras $\B^\star_\bullet$ with the corresponding $\T^\star_\bullet$, and use the latter symbols to denote both Hecke algebras. Similarly, we identify the maximal ideal $\fr m_f$ and the ring $\cl R$ with their images via ${\rm JL}_\infty^\ord$.

Finally, we write $\T_\infty^{\ord,\dagger}$ and $\T_m^{\ord,\dagger}$ for the twisted $G_\Q$-modules $\T_\infty^\ord$ and $\T_m^\ord$, respectively, where the action of $G_\Q$ is via $\Theta^{-1}$.

\subsection{Galois representations in the indefinite case} \label{galois-indefinite-subsec}

In this subsection, as in \cite{fouquet}, we work under the following assumption, whose analogue for the Hida family $\fr h_{\infty,\wt{\fr m}_f}^\ord$ is true by \cite[Proposition 2.1.2]{ho}.

\begin{assumption} \label{gorenstein-ass}
The $\Lambda$-algebra $\T_{\infty,\fr m_f}^\ord$ is Gorenstein, that is $\T_{\infty,\fr m_f}^\ord\simeq\Hom_\Lambda\bigl(\T_{\infty,\fr m_f}^\ord,\Lambda\bigr)$ as $\T_{\infty,\fr m_f}^\ord$-modules.
\end{assumption}

Define 
\[ T_{\infty,\fr m_f}^\ord :=T_\infty^\ord\otimes_{\T^\ord_\infty}\T_{\infty,\fr m_f}^\ord,\qquad {\bf T}_{\rm Sh}:=T_{\infty,\fr m_f}^\ord \otimes_{\T_{\infty,\fr m_f}^\ord}\cl R,\qquad {\bf T}_{\rm Sh}^\dagger:={\bf T}_{\rm Sh}\otimes_\mathcal R\mathcal R^\dagger. \]
Let $\fr p$ be an arithmetic prime of $\cl R$ and set
\[ {\bf T}_{\rm Sh,\fr p}:={\bf T}_{\rm Sh}\otimes_\cl R\cl R_{\fr p},\qquad {\bf T}_{\rm Sh,\fr p}^\dagger:={\bf T}_{\rm Sh,\fr p}\otimes_\mathcal R\mathcal R^\dagger, \]
\[ V_{\rm Sh,\fr p}:= {\bf T}_{\rm Sh,\fr p}/{\fr p}{\bf T}_{\rm Sh,\fr p},\qquad V_{\rm Sh,\fr p}^\dagger:=V_{\rm Sh,\fr p}\otimes_\mathcal R\mathcal R^\dagger. \]
All these groups are endowed with $G_\Q$ and Heche actions. As before, let $\fr m_\cl R$ be the maximal ideal of the local ring $\cl R$. The next assumption plays the role of \cite[Hypoth{\`e}se 1.4.26]{fouquet}.

\begin{assumption} \label{T-Sh-assumption}
The residual $G_\Q$-representation ${\bf T}_{\rm Sh}/\fr m_\cl R{\bf T}_{\rm Sh}$ is absolutely irreducible.
\end{assumption}

We keep Assumptions \ref{gorenstein-ass} and \ref{T-Sh-assumption} for the rest of this subsection. We first recall the basic properties of the representations ${\bf T}_{\rm Sh}$.

\begin{prop} \label{fouquet}
\begin{enumerate}
\item The $\cl R$-module ${\bf T}_{\rm Sh}$ is free of rank two.
\item The $G_\Q$-representation ${\bf T}_{\rm Sh}$ is unramified outside $Np$ and the arithmetic Frobenius at a prime $\ell\nmid Np$ acts with characteristic polynomial $X^2-T_\ell X+[\ell]\ell$.
\item For any arithmetic prime $\fr p$ of $\cl R$ denote by $V(f_\fr p)$ the $G_\Q$-representation over $F_\fr p$ attached to $f_\fr p$. Then the $G_\Q$-representation $V_{{\rm Sh},\fr p}$ is equivalent to the dual $V^*(f_\fr p)$ of $V(f_\fr p)$, hence to $V(f_\fr p)(r_\fr p-1)\otimes[\zeta_\fr p^{-1}]$.
\end{enumerate}
\end{prop}

\begin{proof} Keeping our assumptions on the form $f$ in mind, it can be checked that the hypotheses made in \cite[\S 1.4.5]{fouquet} and used in the proof of \cite[Th{\'e}or\`eme 1]{fouquet} are verified. We just remark that, in our context, \cite[Hypothèse 1.4.28]{fouquet} is the analogue for Shimura curves of the main result of \cite{mr}, whose generalization to Shimura curves when $N^+=1$ and $N^-=pq$ (with $p,q$ distinct primes) is provided by \cite[Theorem 2]{ri}. Assuming that the representation associated with $f$ is ramified at all primes dividing $N^-$, we expect that this result holds in our more general situation as well (details will be given elsewhere). Since all assumptions are verified, the statements of the proposition follow from \cite[Th{\'e}or\`eme 1]{fouquet}. \end{proof}

Now recall the $\cl R$-module ${\bf T}$ defined in \S \ref{arithmetic-primes-subsec}. The following consequence of Proposition \ref{fouquet} will be crucial for our arguments.

\begin{coro} \label{T-T-Sh}
There are isomorphisms of $G_\Q[\T_{\infty,\fr m_f}^\ord]$-modules ${\bf T}\simeq{\bf T}_{\rm Sh}$ and ${\bf T}^\dagger\simeq {\bf T}_{\rm Sh}^\dagger$.
\end{coro}

\begin{proof} First of all, by \S \ref{arithmetic-primes-subsec} and part (1) in Proposition \ref{fouquet}, both ${\bf T}$ and ${\bf T}_{\rm Sh}$ are free $\cl R$-modules of rank two. Moreover, Proposition \ref{T-residue-prop} and Assumption \ref{T-Sh-assumption} guarantee that the residual $G_\Q$-representations ${\bf T}/\fr m_\cl R{\bf T}$ and ${\bf T}_{\rm Sh}/\fr m_\cl R{\bf T}_{\rm Sh}$ are absolutely irreducible. Finally, by \S \ref{arithmetic-primes-subsec} and part (2) in Proposition \ref{fouquet}, the arithmetic Frobenius at a prime $\ell\nmid Np$ acts on ${\bf T}$ and ${\bf T}_{\rm Sh}$ with the same characteristic polynomial. Putting all these statements together, the isomorphisms of $G_\Q$-modules ${\bf T}\simeq{\bf T}_{\rm Sh}$ and ${\bf T}^\dagger\simeq {\bf T}_{\rm Sh}^\dagger$ follow from, e.g., \cite[\S 5, Corollary]{mazur}. The Hecke equivariance is immediate from the definitions. \end{proof}

Corollary \ref{T-T-Sh} implies that for every arithmetic prime $\fr p$ of $\cl R$ there are isomorphisms of $G_\Q[\T_{\infty,\fr m_f}^\ord]$-modules ${\bf T}_\fr p\simeq {\bf T}_{\rm Sh,{\fr p}}$, ${\bf T}_\fr p^\dagger\simeq {\bf T}_{\rm Sh,\fr p}^\dagger$ $V_\fr p\simeq V_{\rm Sh,\fr p}$, $V_\fr p^\dagger\simeq V_{\rm Sh,\fr p}^\dagger$, so in the following we will unify these notations and write ${\bf T}$ in place of ${\bf T}_{\rm Sh}$, and analogously for the other Galois and Hecke-modules.

\section{Big Heegner points and classes} \label{big-heegner-section}

In this section we introduce big Heegner points and big Heegner classes, and prove their main compatibility properties. Note that the first three subsections apply both to the definite and to the indefinite case. These results generalize the construction of Galois cohomology classes out of Heegner points on classical modular curves achieved by Howard in \cite{ho}. The reader is also referred to the work of Fouquet (\cite{fouquet}, \cite{fouquet2}) for an extension of some of Howard's results to the broader setting of Shimura curves attached to indefinite quaternion algebras over totally real fields.

For every integer $d\geq 0$ we introduce the notation
\[ \mathcal G_d:=\Gal(K^{\rm ab}/H_d) \]
where, with a slight abuse, we set $H_0:=K$, so that $\mathcal G_0=\Gal(K^\ab/K)$ is the abelianization $G_K^{\rm ab}$ of the absolute Galois group $G_K=\Gal(\bar\Q/K)$ of $K$.

Recall the subsets $\wt X_m^{(K)}$ defined in \S \ref{section-definite-shimura} (definite case) and \S \ref{section-indefinite-shimura} (indefinite case), where $K$ is an imaginary quadratic field admitting injections $K\hookrightarrow B$. In both cases, let us denote by 
\[ D_m:=\Div\bigl(\wt X_m^{(K)}\bigr) \] 
the submodule of $\Div(\wt X_m)$ supported on points in $\wt X_m^{(K)}$, endowed with natural Hecke and $G_K^{\rm ab}$-actions. Define 
\[ D_m^\ord:=D_m\otimes_{\T_m}\T_m^\ord. \]
For $m\geq2$, the maps $\wt\alpha_m:\wt X_m^{(K)}\rightarrow\wt X_{m-1}^{(K)}$ induce, by covariant functoriality, maps $\wt\alpha_m:D_m\rightarrow D_{m-1}$ that respect the ordinary parts, so we can define
\[ D_\infty^\ord:=\invlim_mD_m^\ord. \]
This group is naturally endowed with actions of $\T_\infty^\ord$ and $G_K^{\rm ab}$. In the indefinite case, when we want to emphasize the field of rationality $H\subset K^{\rm ab}$ of a divisor or a limit of divisors we write $D_\star^\ord(H)$, for $\star=m$ or $\infty$. (In the definite case, all the points in $D_m$ are rational over $K$, so there is no need to specify fields of rationality.) Observe that $\mathcal R$ is naturally a $\T_\infty^\ord$-module through the composition $\T_\infty^\ord\rightarrow\T_{\infty,\mathfrak m_f}^\ord\rightarrow\mathcal R$. Define 
\[ {\bf D}_m:=D_m^\ord \otimes_{\T_\infty^\ord}\cl R,\qquad{\bf D}_m^\dagger:={\bf D}_m\otimes_\mathcal R\mathcal R^\dagger \]
(the structure of $\T_\infty^\ord$-module on $D_m^\ord$ is via the projection $\T_\infty^\ord\rightarrow \T_m^\ord$) and
\[ {\bf D}:=D_{\infty}^\ord \otimes_{\T_{\infty}^\ord}\cl R,\qquad{\bf D}^\dagger:={\bf D}\otimes_\mathcal R\mathcal R^\dagger. \]
All these groups are endowed with Hecke and $G_K^{\rm ab}$-actions. Of course, in the indefinite case they are more generally endowed with $G_\Q$-actions. 

\subsection{Galois relations}

Fix an integer $m\geq 1$ and let $\sigma\in\Gal(\bar\Q/H_{cp^m})$. The inclusion $\Q(\sqrt{p^\ast})\subset H_{cp^m}$ implies that $\sigma$ is the identity on $\Q(\sqrt{p^\ast})$, so it follows that there exists $\xi_\sigma\in\boldsymbol\mu_{p-1}$ such that $\xi^2_\sigma=\epsilon_\tame(\sigma)$. Hence
\begin{equation} \label{eq4}
\xi_\sigma\epsilon_\wild^{1/2}(\sigma)=\pm\vartheta(\sigma),
\end{equation}
with $\vartheta$ as in \S \ref{compatible-Galois-families}. By definition, $\Theta(\sigma)=\xi^{k+j-2}_\sigma\bigl[\epsilon_\wild^{1/2}(\sigma)\bigr]$. From \eqref{fund-rel-in-Lambda} it follows that
\begin{equation} \label{equat5}
\Theta(\sigma)e_{k+j-2}=\xi^{k+j-2}_\sigma\bigl[\epsilon_\wild^{1/2}(\sigma)\bigr]e_{k+j-2}=e_{k+j-2}\bigl[\xi_\sigma\epsilon_\wild^{1/2}(\sigma)\bigr]
\end{equation}
in $\Lambda$. Equations \eqref{e-T-eq}, \eqref{eq4} and \eqref{equat5} imply that
\begin{equation} \label{eq-Theta-vartheta}
\Theta(\sigma)P=[\pm\vartheta(\sigma)]P=\langle\vartheta(\sigma)\rangle P
\end{equation}
for all $P\in {\bf D}_m$, where $\langle\ell\rangle$ is the diamond operator at $\ell$ as in \S \ref{compatible-Galois-families}.

Recall the point $\wt P_{c,m}\in\wt X_m^{(K)}$ defined in \S \ref{compatible-Heegner-subsec} and write ${\bf P}_{c,m}$ for its image in $\D_m$. For all $\sigma\in\Gal(\bar\Q/H_{cp^m})$, equations \eqref{eq5} and \eqref{eq-Theta-vartheta} give the equality ${\bf P}_{c,m}^\sigma=\Theta(\sigma){\bf P}_{c,m}$ in $\D_m$, from which it follows that
\[ {\bf P}_{c,m}\in H^0\bigl(\cl G_{cp^m},\D_m^\dagger\bigr). \]

\subsection{Hecke relations}

For any integers $s,t\geq1$ write $\cor_{H_{st}/H_s}$ for the corestriction map from $H_{st}$ to $H_s$. Explicitly, for all $\eta\in\Gal(H_{st}/H_{s})$ choose an extension $\tilde\eta\in\Gal(K^{\rm ab}/H_s)$ of $\eta$; if $Q\in H^0\bigl(\cl G_{st},\D_m^\dagger\bigr)$ then
\begin{equation} \label{cores}
\cor_{H_{st}/H_s}(Q)=\sum_{\eta\in\Gal(H_{st}/H_s)}\Theta(\tilde\eta^{-1})Q^{\tilde\eta}.
\end{equation}
As usual, the maps $\wt\alpha_m:\wt X_m\to\wt X_{m-1}$ induce maps
\[ \wt\alpha_m=\wt\alpha_{m,\ast}: H^0\bigl(\mathcal G_{cp^m},
\D_{m}^{\dagger}\bigr)\longrightarrow H^0\bigl(\mathcal G_{cp^m},\D_{m-1}^{\dagger}\bigr) \]
by covariant functoriality.

\begin{prop} \label{prop-Hecke-relations}
The equality
\[ \wt\alpha_m\bigl(\cor_{H_{cp^m}/H_{cp^{m-1}}}({\bf P}_{c,m})\bigr)=U_p({\bf P}_{c,m-1}) \]
holds in $H^0\bigl(\mathcal G_{cp^{m-1}},\D_{m-1}^{\dagger}\bigr)$ for all $m\geq1$.
\end{prop}

\begin{proof} Since $H_{cp^m}$ and $H_{cp^{m-1}}(\boldsymbol\mu_{p^\infty})$ are linearly disjoint over $H_{cp^{m-1}}$, we can fix a finite set $S_m\subset\Gal(K^{\rm ab}/H_{cp^{m-1}})$ of extensions of the elements in $\Gal(H_{cp^m}/H_{cp^{m-1}})$ such that every $\sigma\in S_m$ acts trivially on  $\boldsymbol\mu_{p^\infty}$.
Applying $e_{k+j-2}e^\ord$ to the equation of Proposition \ref{Hecke-relation-for-X} yields the analogous relation
\[ \wt\alpha_m\Big(\sum_{\sigma\in S_m}{\bf P}_{c,m}^\sigma\Big)=U_p({\bf P}_{c,m-1}). \]
Now we calculate the corestriction $\cor_{H_{cp^m}/H_{cp^{m-1}}}$ by choosing the elements $\tilde\eta$ of
\eqref{cores} in $S_m$, and this yields
\[ \wt\alpha_m\bigl(\cor_{H_{cp^m}/H_{cp^{m-1}}}({\bf P}_{c,m})\bigr)=\wt\alpha_m\Big(\sum_{\sigma\in S_m}{\bf P}_{c,m}^{\sigma}\Big). \]
The result follows. \end{proof}

For all $m\geq1$ define
\[ \cl P_{c,m}:=\cor_{H_{cp^m}/H_c}({\bf P}_{c,m})\in H^0\bigl(\mathcal G_c,\D_{m}^{\dagger}\bigr). \]
\begin{coro}[Hecke relations] \label{coro-hecke-relations}
The equality
\[ \wt\alpha_m(\cl P_{c,m})=U_p(\cl P_{c,m-1}).\]
holds in $H^0\bigl(\mathcal G_c,\D_{m-1}^{\dagger}\bigr)$ for all $m\geq1$.
\end{coro}

\begin{proof} Straightforward from Proposition \ref{prop-Hecke-relations} on applying $\cor_{H_{cp^{m-1}}/H_c}$. \end{proof}

\subsection{Big Heegner points}

Thanks to Corollary \ref{coro-hecke-relations} and the isomorphism
\[ \invlim_mH^0\bigl(\mathcal G_c,\D_{m}^{\dagger}\bigr)\simeq H^0\bigl(\mathcal G_c,\D^{\dagger}\bigr), \]
the following definition makes sense.

\begin{defi} \label{big-heegner-point-defi}
The \emph{big Heegner point of conductor $c$} is the element
\[ \cl P_c:=\invlim_m U_p^{-m}\bigl(\cl P_{c,m}\bigr)\in H^0\bigl(\mathcal G_c,\D^{\dagger}\bigr). \]
\end{defi}

\subsection{Big Heegner classes in the indefinite case}

Suppose we are in the \emph{indefinite} case. Let $H_c^{(Np)}$ be the maximal extension of $H_c$ unramified outside $Np$ and  set
\[ G_c^{(Np)}:=\Gal\bigl(H_c^{(Np)}/H_c\bigr). \]
If, with the above notations and conventions, we set 
${\bf J}_m^\dagger:=J_{m}^{\ord}\otimes_{\T_\infty}\mathcal R^\dagger$, then we may 
define the twisted Kummer map
\[ \delta_m:H^0\bigl(H_c,{\bf J}_m^\dagger\bigl(H_{cp^m}(\boldsymbol\mu_{p^m})\bigr)\bigr)\longrightarrow H^1\bigl(G_c^{(Np)},{\bf T}_m^\dagger\bigr) \]
as in \cite[p. 101]{ho}. Write
\[ \bar{\cl P}_{c,m}\in H^0\bigl(H_c,{\bf J}_{m}^{\dagger}\bigl(H_{cp^m}(\boldsymbol\mu_{p^m})\bigr)\bigr) \]
for the image of $\cl P_{c,m}\in H^0\bigl(\mathcal G_c,\D_m^\dagger\bigr)$ under the natural map and set $\kappa_{c,m}:=\delta_m(\bar{\cl P}_{c,m})$. Because of the $U_p$-equivariance of the Kummer map, the Hecke relations of Corollary \ref{coro-hecke-relations} imply the corresponding relations
\begin{equation} \label{Hecke-rel-classes}
\wt\alpha_m(\kappa_{c,m})=U_p(\kappa_{c,m-1}).
\end{equation}
Thanks to \eqref{Hecke-rel-classes} and the isomorphism of $\T_\infty^{\ord,\dagger}$-modules
\begin{equation} \label{invlim-cohom-eq}
\invlim_m H^1\bigl(G_c^{(Np)},{\bf T}_m^\dagger\bigr)\simeq H^1\bigl(G_c^{(Np)},{\bf T}^\dagger\bigr),
\end{equation}
we can give the following

\begin{defi} \label{big-heegner-class-defi}
The \emph{big Heegner class of conductor $c$} is the element
\[ \kappa_c:=\invlim_m U_p^{-m}\bigl(\kappa_{c,m}\bigr)\in H^1\bigl(G_c^{(Np)},{\bf T}^{\dagger}\bigr). \]
\end{defi}
Put $H_{cp^\infty}(\boldsymbol\mu_{p^\infty}):=\cup_{m\geq1}H_{cp^m}(\boldsymbol\mu_{p^m})$ and, 
as above, define 
\[{\bf J}^\dagger:=J_{\infty}^{\ord}\otimes_{\T^\ord_\infty}\mathcal R^\dagger.\] 
By isomorphism \eqref{invlim-cohom-eq}, taking the inverse limit with respect to the maps $\delta_m$ yields a twisted Kummer map
\[ \delta_\infty:H^0\bigl(G_c^{(Np)},{\bf J}^{\dagger}\bigl(H_{cp^\infty}(\boldsymbol\mu_{p^\infty})\bigr)\bigr)
\longrightarrow H^1\bigl(G_c^{(Np)},{\bf T}^{\dagger}\bigr). \]
Write
\[ \bar{\cl P}_c\in H^0\bigl(\mathcal G_c,{\bf J}^{\dagger}
\bigl(H_{cp^\infty}(\boldsymbol\mu_{p^\infty})\bigr)\bigr)\]
for the image of $\cl P_c\in H^0\bigl(\mathcal G_c,{\bf D}^{\dagger}\bigr)$ under the natural map.
The next lemma will be used in the proof of Corollary \ref{coro-euler-system}.

\begin{lemma} \label{C}
$\delta_\infty(\bar{\cl P}_c)=\kappa_c$.
\end{lemma}

\begin{proof} Recall that, by definition, $\delta_m(\bar{\cl P}_{c,m})=\kappa_{c,m}$ and pass to the inverse limit over $m$.\end{proof}

\begin{rem}
In the special case where $N^-=1$ (i.e., when $B\simeq\M_2(\Q)$) we expect that our system of big Heegner classes essentially coincides with the system of big Heegner points considered by Howard in \cite{ho}. On the contrary, we have not investigated the existence of an explicit relation between our \emph{indefinite} cohomology classes and the specialization to the base field $F=\Q$ of the ones introduced by Fouquet in \cite{fouquet}.
\end{rem}

\section{Euler system relations} \label{section-euler-system-relations}

This section is devoted to the proof of the ``Euler system'' relations satisfied by the classes
${\cl P}_{c,m}$ and $\cl P_c$ introduced above. The formulas obtained, which are the counterparts in our definite/indefinite quaternionic setting of the results in \cite[\S 2.3]{ho}, will be used in \S \ref{bounding-selmer-subsec} to control the size of certain Selmer groups.

\subsection{The operator $U_p$} \label{U-p-big-subsec}

We begin with an analysis of the action of the Hecke operator $U_p$.

\begin{prop} \label{Hecke-relation-II}
For all $m\geq1$ the equality
\[ U_p\bigl({\bf P}_{c,m}\bigr)=\cor_{H_{cp^{m+1}}/H_{cp^m}}\bigl({\bf P}_{cp,m}\bigr) \]
holds in $H^0\bigl(\mathcal G_{cp^m},{\bf D}_m^{\dagger}\bigr)$.
\end{prop}

\begin{proof} The proof is similar to that of Proposition \ref{prop-Hecke-relations}. Applying $e_{k+j-2}e^\ord$ to the equation of Proposition \ref{horizontal-hecke-X} yields the analogous relation $\sum_{\sigma\in S_m}{\bf P}_{cp,m}^\sigma=U_p({\bf P}_{c,m})$ in ${\bf D}_m^\dagger$, and calculating corestriction $\cor_{H_{cp^{m+1}}/H_{cp^m}}$ by choosing the elements $\wt\eta$ of \eqref{cores} in $S_{m+1}$ gives the result. \end{proof}

\begin{coro} \label{coro-euler-system}
The following relations hold for all integers $c,m\geq1$:
\begin{enumerate}
\item\label{1} $U_p(\cl P_{c,m})=\cor_{H_{cp}/H_c}(\cl P_{cp,m})$ in $H^0\bigl(\cl G_c,{\bf D}_m^\dagger\bigr)$;
\item\label{2} $U_p(\cl P_c)=\cor_{H_{cp}/H_c}(\cl P_{cp})$ in $H^0\bigl(\cl G_c,{\bf D}^\dagger\bigr)$;
\item\label{3} $U_p(\kappa_c)=\cor_{H_{cp}/H_c}(\kappa_{cp})$ in $H^1\bigl(G_c^{(Np)},{\bf T}^\dagger\bigr)$.
\end{enumerate}
\end{coro}

\begin{proof} Relation (\ref{1}) follows easily from Proposition \ref{Hecke-relation-II} and the equality
\[ \cor_{H_{cp}/H_c}\circ\cor_{H_{cp^{m+1}}/H_{cp}}=\cor_{H_{cp^m}/H_c}\circ\cor_{H_{cp^{m+1}}/H_{cp^m}}. \]
Relation (\ref{2}) follows from (\ref{1}) by passing to the inverse limit over $m$. Finally, relation (\ref{3}) is a consequence of Lemma \ref{C} and the equivariance of the twisted Kummer map with respect to the action of $U_p$. \end{proof}

\subsection{The operators $T_\ell$} \label{T-ell-operator}

Let $c\geq1$ be an integer prime to $N$. Fix an integer $m\geq0$ and a prime number $\ell\nmid Np^mc$ which is inert in $K$. As done before when proving explicit formulas for the operators $T_\ell$, in this subsection we assume for simplicity that $\cl O_{cp^m}^\times=\{\pm1\}$.

\begin{prop} \label{ell} The equality
\[ T_\ell({\bf P}_{c,m})=\cor_{H_{c\ell p^m}/H_{cp^m}}({\bf P}_{c\ell,m}) \]
holds in $H^0\bigl(\mathcal G_c,{\bf D}_m^\dagger\bigr)$.
\end{prop}

\begin{proof} Similar to that of Proposition \ref{Hecke-relation-II}. Choose a set $S\subset \Gal(\C/H_{cp^m})$ of extensions of $\Gal(H_{c\ell p^m}/H_{cp^m})$ such that every $\sigma\in S$ acts trivially on $\boldsymbol\mu_{p^\infty}$. Proposition \ref{T-operator} gives
\begin{equation} \label{T-ell}
\sum_{\sigma\in S}{\bf P}_{c\ell,m}^\sigma=T_\ell\bigl({\bf P}_{c,m}\bigr)
\end{equation}
in ${\bf D}_m^\dagger$. Applying $e_{k+j-2}e^\ord$ and calculating $\cor_{H_{c\ell p^m}/H_{cp^m}}$ by choosing the elements $\wt\eta$ of \eqref{cores} in $S$ yields the desired result. \end{proof}

\begin{coro}
There are equalities
\begin{enumerate}
\item\label{1-bis} $T_\ell(\cl P_{c,m})=\cor_{H_{c\ell}/H_c}(\cl P_{c\ell,m})$ in $H^0\bigl(\mathcal G_c,{\bf D}_m^\dagger\bigr)$;
\item\label{2-bis} $T_\ell(\cl P_c)=\cor_{H_{c\ell}/H_c}(\cl P_{c\ell})$ in $H^0\bigl(\mathcal G_c,{\bf D}^\dagger\bigr)$;
\item\label{3-bis} $T_\ell(\kappa_c)=\cor_{H_{c\ell}/H_c}(\kappa_{c\ell})$ in $H^1\bigl(G_c^{(Np)},{\bf T}^\dagger\bigr)$.
\end{enumerate}
\end{coro}

\begin{proof} Same proof as for Corollary \ref{coro-euler-system}, but this time to obtain relation \eqref{1-bis} one uses the equality
\[ \cor_{H_{c\ell}/H_c}\circ\cor_{H_{c\ell p^{m}}/H_{c\ell}}=\cor_{H_{cp^m}/H_c}\circ\cor_{H_{c\ell p^m}/H_{cp^m}}, \]
and for relation \eqref{3-bis} one uses the equivariance of the twisted Kummer map with respect to the action of $T_\ell$. \end{proof}

\subsection{The Eichler--Shimura congruence relation}

Throughout this subsection we restrict to the \emph{indefinite} case. Let $\ell\nmid Npc$ be a prime which is inert in $K$. By class field theory, $\ell$ splits completely in the extension $H_c/K$. Fix a prime $\lambda$ of $H_c$ above $\ell$. Note that $\lambda$ is totally ramified in $H_{c\ell}$, so $\lambda\cdot\cl O_{H_{c\ell}}=\tilde\lambda^{\ell+1}$ for a prime ideal $\tilde\lambda$ of the ring of integers $\cl O_{H_{c\ell}}$ of $H_{c\ell}$ above $\ell$. For every prime number $q$ and every integer $k\geq1$ denote by $\mathbb F_{q^k}$ the field with $q^k$ elements, and for every number field $H$ and every prime ideal $\fr q$ of the ring of integers of $H$ denote by $\Frob_\fr q$ a Frobenius element at $\fr q$ and by $\mathbb F_{H,\fr q}$ the residue field of $H$ at $\fr q$. Then there are canonical isomorphisms
\[ \mathbb F_{\ell^2}\simeq\mathbb F_{K,\ell}\simeq\mathbb F_{H_c,\lambda}\simeq\mathbb F_{H_{c\ell},\tilde\lambda}. \]
Write $\wt X_{m,\ell}$ for the canonical (smooth, proper) integral model of $\wt X_m$ over $\Z_\ell$. By the valuative criterion of properness, any point $x\in\wt X_m$ extends uniquely to a point in $\wt X_{m,\ell}$, which will be denoted in the same fashion. As in \S \ref{T-ell-operator}, we assume that $\mathcal O_{cp^m}^\times=\{\pm 1\}$.

\begin{lemma} \label{lemma-eichler-shimura}
Let $\ell\nmid Npc$ be a prime number which is inert in $K$. Then in $\wt X_{m,\ell}$ we have:
\[{\bf P}_{c\ell,m}\equiv\Frob_\lambda\bigl({\bf P}_{c,m}\bigr)\pmod{\tilde\lambda}.\]
\end{lemma}

\begin{proof} Choose $S_m$ as in the proof of Proposition \ref{ell}. For any $\sigma\in S$, there is a congruence
\[ {\bf P}_{c\ell,m}^\sigma\equiv{\bf P}_{c\ell,m}\mod{\tilde\lambda} \]
because $\lambda$ is totally ramified in $H_{c\ell}$. Therefore, by \eqref{T-ell}, it follows that
\[ T_\ell({\bf P}_{c,m})\equiv(\ell+1){\bf P}_{c\ell,m}\mod{\tilde\lambda}. \]
The Eichler--Shimura congruence relation (see, e.g., \cite[\S 10.3]{ca}) shows that $T_\ell={\rm Frob}_\ell+{\rm Frob}^\ast_\ell\pmod{\ell}$, hence at least one of the points in the divisor $T_\ell\bigl({\bf P}_{c,m}\bigr)$ is congruent to ${\rm Frob}_\lambda\bigl({\bf P}_{c,m}\bigr)$ modulo $\tilde\lambda$. Thus the same holds for all the points in the divisor $T_\ell\bigl({\bf P}_{c,m}\bigr)$, and in particular for ${\bf P}_{c\ell,m}$. \end{proof}

\begin{rem}
See \cite[Proposition 3.7]{gross} for the same argument applied in the context of Heegner points on (classical) modular curves.
\end{rem}

\begin{prop}
Let $\ell\nmid Npc$ be a prime number which is inert in $K$. Then $\kappa_{c\ell}$ and ${\rm Frob}_\lambda(\kappa_{c\ell})$ have the same image in $H^1\bigl(H_{c\ell,\tilde\lambda},{\bf T}^{\dagger}\bigl)$.
\end{prop}

\begin{proof} Proceed as in the proof of \cite[Proposition 2.3.2]{ho}, using Lemma \ref{lemma-eichler-shimura}. \end{proof}

\section{Arithmetic applications and conjectures: the definite case} \label{Part-III-section-def}

From here to the end of the paper, fix a modular form $f$ of weight $k$ as in \eqref{fixed-modular-form} and let
\[ f_\infty:\fr h_\infty^\ord\longrightarrow\cl R \]
be the primitive morphism associated with $f$. If $\fr p$ is an arithmetic prime of $\cl R$ then $f_\fr p$ is the modular form introduced in \eqref{form-ass-arith-prime}. Recall that $\cl R$ is a complete local noetherian domain which is finitely generated as a $\Lambda$-module (Proposition \ref{A-ring-prop}) and that if $\fr p$ is an arithmetic prime of $\cl R$ and $P:=\fr p\cap\Lambda$ then $\Lambda_P\subset\cl R_\fr p$ is an unramified extension of discrete valuation rings (see \cite[Corollary 1.4]{h-galois} or \cite[\S 12.7.5]{Ne-selmer}).

The purpose of the following sections is to apply our constructions of big Heegner points and classes to various arithmetic situations. While so far we have strived to adopt a uniform approach to the definite and indefinite cases, at this point it is inevitable to distinguish between these two settings. In fact, the philosophy behind the so-called ``parity conjectures'' suggests that the definite case deals with \emph{even} rank (most typically, rank zero) situations while the indefinite case takes care of \emph{odd} rank (most notably, rank one) contexts.

Throughout this section we assume that we are in the \emph{definite} case, i.e. that the quaternion algebra $B$ is definite.

\subsection{Algebraic results} \label{section-structure-J} 

Let $m\geq0$ be an integer. Since $\wt X_m$ is a disjoint union of $\tilde h(m)$ curves of genus $0$, we can fix an isomorphism of $\cl O_F$-modules between $H_0\bigl(\wt X_m(\C),\cl O_F\bigr)$ and $J_m$ where $H_\star$ denotes singular homology. The above isomorphism endows
$H_0\bigl(\wt X_m(\C),\cl O_F\bigr)$ with a canonical Hecke action. Passing to the ordinary parts, one thus obtains an isomorphism of Hecke modules
\begin{equation} \label{homology-J-eq}
H_0^\ord\bigl(\wt X_m(\C),\cl O_F\bigr):=e_m^\ord\cdot H_0\bigl(\wt X_m(\C),\cl O_F\bigr)\simeq J_m^\ord.
\end{equation}
The cohomology module $H^0\bigl(\wt X_m(\C),F/\cl O_F\bigr)$ with coefficients in $F/\cl O_F$ is also equipped with a canonical Hecke action, and its ordinary part is defined in the usual way. For any $\cl O_F$-module $M$ let
\[ M^\ast:=\Hom_{\cl O_F}(M,F/\cl O_F) \]
denote its Pontryagin dual, with induced Hecke action whenever $M$ is a module over the Hecke algebra. Then (see, e.g., \cite[\S 1.9]{hida-duke}) there is a canonical isomorphism of Hecke modules
\[ H^0\bigl(\wt X_m(\C),F/\cl O_F\bigr)^\ast\simeq H_0\bigl(\wt X_m(\C),\cl O_F\bigr) \]
which induces an isomorphism of $\cl O_F$-modules
\begin{equation} \label{isom-dual-eq}
H^0_\ord\bigl(\wt X_m(\C),F/\cl O_F\bigr)^\ast\simeq H_0^\ord\bigl(\wt X_m(\C),\cl O_F\bigr).
\end{equation}
Following \cite[Definition 8.5]{Hida-annals}, set
\[ \cl V:=\dirlim_mH^0_\ord\bigl(\wt X_m(\C),F/\cl O_F\bigr),\qquad V:=\cl V^*. \]
Then \eqref{homology-J-eq} and \eqref{isom-dual-eq} yield isomorphisms of $\cl O_F$-modules
\begin{equation} \label{V-isom-eq}
\begin{split}
V&=\Hom_{\cl O_F}\Big(\textstyle{\varinjlim_m} H^0_\ord\bigl(\wt X_m(\C),F/\cl O_F\bigr),F/\cl O_F\Big) \\[2mm]
&\simeq\invlim_m H^0_\ord\bigl(\wt X_m(\C),F/\cl O_F\bigr)^\ast\\[2mm]
&\simeq\invlim_m H_0^\ord\bigl(\wt X_m(\C),\cl O_F\bigr)\simeq\invlim_mJ_m^\ord=J_\infty^\ord.
\end{split}
\end{equation}
Recall that $\Gamma:=1+p\Z_p$ and define $\Gamma_m:=1+p^m\Z_p$; in particular,
\[ \Lambda:=\cl O_F[\![\Gamma]\!]=\invlim_m\cl O_F[\Gamma/\Gamma_m]. \]
The group $\Gamma$ acts on $\widehat B^\times$ via multiplication on the $p$-component, and this induces an $\cl O_F$-linear action of $\Gamma/\Gamma_m$ on $J_m$. Thus $J_\infty^\ord$ is endowed with an action of $\Lambda$ which is, of course, the one induced by its $\T_\infty^\ord$-module structure. Furthermore, isomorphisms \eqref{V-isom-eq} are $\Lambda$-equivariant, the structure of $\Lambda$-module of $V$ being defined as in \cite[\S 9]{Hida-annals}. By \cite[Corollary 10.4]{Hida-annals}, the $\Lambda$-modules $V$ and $J_\infty^\ord$ are free of finite rank. It follows immediately that $J_\infty^\ord$ is a finitely generated $\T_\infty^\ord$-module, hence \[{\bf J}:=J^\ord_{\infty}\otimes_{\T_\infty^\ord}\mathcal R\] is a finitely generated $\cl R$-module. If ${\fr p}$ (respectively, $P$) is an arithmetic prime of $\cl R$ (respectively, $\Lambda$) and $M$ is an $\cl R$-module (respectively, a $\Lambda$-module) then we set $M_{\fr p}:=M\otimes_\cl R\cl R_\fr p$ (respectively, $M_P:=M\otimes_\Lambda\Lambda_P$), where $\cl R_\fr p$ (respectively, $\Lambda_P$) is the localization of $\cl R$ at $\fr p$ (respectively, of $\Lambda$ at $P$). To lighten the notation, put also
\[\T:=\T_\infty^\ord.\]

\begin{prop} \label{prop-hida}
Let $\fr p$ be an arithmetic prime of $\cl R$. The $\cl R_\fr p$-module ${\bf J}_\fr p$ is free  of rank one.
\end{prop}

\begin{proof} By \cite[Theorem 12.1]{Hida-annals}, there are isomorphisms of $\T_P$-modules $V_P\simeq\T_P$ for all arithmetic primes $P$ of $\Lambda$. Since there are isomorphisms of $\cl R_P$-modules ${\bf J}_P\simeq V_P\otimes_\T\cl R$ and $\cl R_P\simeq\T_P\otimes_\T\cl R$, we conclude that
\begin{equation}\label{eq-hida-I}
{\bf J}_P\simeq\cl R_P
\end{equation}
as $\cl R_P$-modules. Fix an arithmetic prime $\fr p$ of $\cl R$ and let $P:=\fr p\cap\Lambda$ be the arithmetic prime of $\Lambda$ which lies below $\fr p$. There is a canonical map of rings $\cl R_P\rightarrow\cl R_\fr p$ defined by the composition
\[ \cl R_P:=\cl R\otimes_\Lambda\Lambda_P\longrightarrow\cl R\otimes_\Lambda\cl R_\fr p\longrightarrow\cl R\otimes_{\cl R}\cl R_\fr p=\cl R_\fr p. \]
There are isomorphisms of $\cl R_\fr p$-modules
\begin{equation} \label{J-R-loc-eq}
{\bf J}_P\otimes_{\cl R_P}\cl R_\fr p\simeq V\otimes_\T\cl R_\fr p\simeq(V\otimes_\T\cl
R)\otimes_{\cl R}\cl R_\fr p={\bf J}\otimes_{\cl R}\cl R_\fr p={\bf J}_\fr p.
\end{equation}
Furthermore, thanks to \eqref{eq-hida-I}, ${\bf J}_P\otimes_{\cl R_P}\cl R_\fr p\simeq \cl R_\fr p$ as $\cl R_\fr p$-modules. Comparing this with \eqref{J-R-loc-eq} yields the result. \end{proof}
From here until the end of the section we make the following assumption, which is coherent with
the result proved in Proposition \ref{prop-hida}.

\begin{assumption} \label{ass-def}
Let $\fr m_\cl R$ be the maximal ideal of the local ring $\cl R$ and let $\mathbb F_\cl R:=\cl R/\fr m_\cl R$ be its residue field. The $\mathbb F_\cl R$-vector space ${\bf J}/\fr m_\cl R{\bf J}$ has dimension one.
\end{assumption}
With this condition in force, we can prove

\begin{prop} \label{prop-appunti-1}
The $\cl R$-module $\bf J$ is free of rank one.
\end{prop}

\begin{proof} Since ${\bf J}$ is finitely generated over $\cl R$ and Assumption \ref{ass-def} holds, Nakayama's lemma ensures that there is a surjective homomorphism $\cl R\twoheadrightarrow{\bf J}$ of $\cl R$-modules. If this map is not
an isomorphism, there is a non-zero ideal $I\subset\cl R$ such that $\cl R/I\simeq{\bf J}$ as $\cl R$-modules. By \cite[Theorem 6.5]{mat}, the localization $(\cl R/I)_\fr p$ is non-zero only for a finite number of arithmetic primes $\fr p$ of $\cl R$. Hence ${\bf J}_\fr p=0$ for almost all arithmetic primes $\fr p$, contradicting Proposition \ref{prop-hida} (of course, the local vanishing at just one such prime $\fr p$ would suffice). Thus $I$ is the zero ideal, and the proposition is proved. \end{proof}

In light of Proposition \ref{prop-appunti-1}, fix an isomorphism
\begin{equation} \label{iso-J-R}
{\bf J}\overset{\simeq}{\longrightarrow}\cl R.
\end{equation}
of $\cl R$-modules. If $H/K$ is a finite abelian extension then composing the canonical map $H^0\bigl(\Gal(K^{\rm ab}/H),{\bf D}^\dagger\bigr)\rightarrow{\bf D}$ with the surjection ${\bf D}\twoheadrightarrow{\bf J}$ and isomorphism \eqref{iso-J-R} produces a map
\begin{equation} \label{eta-H-map-eq}
\eta_H:H^0\bigl(\Gal(K^{\rm ab}/H),{\bf D}^\dagger\bigr)\longrightarrow\cl R.
\end{equation}
To simplify notations, for every integer $d\geq0$ set $\eta_d:=\eta_{H_d}$, with the convention (introduced at the beginning of Section \ref{big-heegner-section}) that $H_0=K$. These maps will be used in the next subsection to state our results on Selmer groups. In particular, \S \ref{section-bounding-def} and \S \ref{section-iwasawa-def} are motivated by \cite[Theorems A and B]{BD-annals-146} and \cite[Corollary 4]{bd*}, respectively, where classical Heegner (or, better, Gross) points on definite Shimura curves are used to control certain Selmer groups.

\subsection{Root numbers and bounds on Selmer groups} \label{section-bounding-def}

To begin with, for later reference we point out the following algebraic result.

\begin{prop} \label{R-prop}
If $x\in\cl R$ is non-zero then there are only finitely many prime ideals $\fr p$ of $\cl R$ such that $x\in\fr p$, i.e., such that $\pi_\fr p(x)=0$.
\end{prop}

\begin{proof} Since it is an integral extension of $\Lambda$, the local domain $\cl R$ has Krull dimension $2$. It follows that the height of a prime ideal $\fr p\not=0$ of $\cl R$ is either $1$ or $2$, the latter possibility occurring if and only if $\fr p$ is the maximal ideal of $\cl R$. To prove the proposition it thus suffices to show that an intersection $I:=\cap_{\fr p\in\cl S}\fr p$ of infinitely many height one prime ideals of $\cl R$ (indexed by a set $\cl S$) is necessarily trivial. If this were not the case then every $\fr p\in\cl S$, having height one, would be minimal among the prime ideals of $\cl R$ containing $I$. But the set of prime ideals of $\cl R$ containing $I$ has only finitely many minimal elements by \cite[Exercise 4.12]{mat}, and we are done. \end{proof}

Recall the class $\cl P_c\in H^0\bigl(\mathcal G_c,{\bf D}^\dagger\bigr)$ introduced in Definition  \ref{big-heegner-point-defi}. For every integer $d\geq0$ define
\[ \tilde G_d:=\Gal(H_d/K) \]
with the convention that $H_0=K$ as before. Set
\[ \cl J_0:=\sum_{\sigma\in\tilde G_1}\eta_1(\cl P^\sigma_1)\in\mathcal R,\qquad \cl J_c:=\sum_{\sigma\in\tilde G_c}\eta_c(\cl P^\sigma_c)\otimes\sigma^{-1}\in\cl R\bigl[\tilde G_c\bigr]\quad\text{if $c\geq1$}. \]
Fix a character $\chi:\tilde G_c\rightarrow\cl O^\times$ where $\cl O$ is a finite extension of $\cl O_F$
and $c\geq 1$ is an integer. After enlarging $F$ if necessary, without loss of generality we can assume that $\cl O=\cl O_F$. Extend $\chi$ to an $\cl R$-linear homomorphism $\chi:\cl R\bigl[\tilde G_c\bigr]\rightarrow{\cl R}$, then define
\begin{equation} \label{L-chi-eq}
\cl L(f_\infty/K,\chi):=\chi(\cl J_c)\in\cl R,\qquad\cl L(f_\infty/K,\chi,\fr p):=\pi_\fr p\bigl(\cl L(f_\infty/K,\chi)\bigr)\in F_\fr p
\end{equation}
where $\fr p$ is an arithmetic prime of $\cl R$. In particular, if $\chi={\bf 1}$ is the trivial character of $\tilde G_1$ then $\cl L(f_\infty/K,{\bf 1})=\mathcal J_0$.

Since the critical character $\Theta$ introduced in \S \ref{critical-subsection} is trivial on $\Gal(\bar\Q/\Q(\boldsymbol{\mu}_{p^\infty}))$, there is an induced map $\Theta:\Gal(\Q(\boldsymbol{\mu}_{p^\infty})/\Q)\rightarrow\cl R^\times$. As in \cite[Section 2]{ho2}, for every arithmetic prime $\fr p$ of $\cl R$ denote by $\theta_\fr p$ the composition
\[ \theta_\fr p:\Z_p^\times\xrightarrow{\epsilon_{\text{cyc}}^{-1}}\Gal(\Q(\boldsymbol{\mu}_{p^\infty})/\Q)\overset\Theta\longrightarrow\cl R^\times\longrightarrow F_\fr p^\times, \]  
then set
\begin{equation} \label{twisted-f-eq}
f_\fr p^\dagger:=f_\fr p\otimes\theta_\fr p^{-1}. 
\end{equation}
The form $f_\fr p^\dagger$ has trivial nebentype, and the twisted representation $V_\fr p^\dagger$ is the representation attached to $f_\fr p^\dagger$ by Deligne. 

We make two conjectures, the first of which concerns the non-vanishing of $\cl L(f_\infty/K,\chi)$. To formulate them, let $\fr p$ be an arithmetic prime of $\cl R$ and write $w_\fr p$ for the root number (i.e., the sign in the functional equation) of the $L$-function of the modular form $f_\fr p^\dagger$. Except for finitely many primes $\fr p$ (which were explicitly determined by Mazur, Tate and Teitelbaum in \cite{mtt} and correspond to the \emph{exceptional} primes of \S \ref{section-Selmer-groups}), the number $w_\fr p$ is constant as $\fr p$ varies (see, e.g., \cite[\S 3.4.4]{np} for details); we denote this common root number by $w$.

\begin{conj} \label{conj-def-1}
If $w=1$ then $\mathcal L(f_\infty/K,\chi)\not=0$. In particular, if $w=1$ then $\cl J_0\not=0$.
\end{conj}

In light of Proposition \ref{R-prop} and Conjecture \ref{conj-def-1}, we expect that if $w=1$ then $\cl L(f_\infty/K,\chi,\fr p)$ does not vanish for all but finitely many $\fr p$. As anticipated in the introduction, Conjecture \ref{conj-def-1} can be justified as follows. Suppose for simplicity that $c=1$ and $\chi={\bf 1}$ is the trivial character. The element $\cl J_0\in\cl R$ is the analogue in our Hida setting of the divisor $c_{\boldsymbol 1}$ introduced by Gross in \cite[\S 11]{gross-2}, hence it is natural to expect that it encodes, via the ``specialization'' maps $\pi_\fr p$, the (non-)vanishing of the special values of the classical $L$-functions of the modular forms $f_\fr p^\dagger$ in the family. When $w=1$ the functional equations of the $L$-functions suggest that these special values are non-zero for almost all arithmetic primes $\fr p$, so in this analytic situation it is natural to predict (cf. Proposition \ref{R-prop}) the non-triviality of $\cl J_0$. We refer the reader to \S  \ref{section-vanishing} for conjectures \emph{{\`a} la} Greenberg on the vanishing of the special values of twisted $L$-functions over $K$ of the forms in the Hida branch of $f$.

The next conjectural statement is about the size of Nekov\'a\v{r}'s Selmer groups $\wt H^1_f\bigl(K,V_\fr p^\dagger\bigr)$ and their $\chi$-twists. If $\epsilon$ is the quadratic character of the extension $K/\Q$ then the generic root number of the twisted forms $f_\fr p^\dagger\otimes\epsilon$ is $-\epsilon(N)w$, hence is $w$ in the definite case and $-w$ in the indefinite case (see, e.g., \cite[Ch. IV]{GZ}). In particular, if we are in the definite case (which is the situation considered in this section) and $w=-1$ then we expect that
\[ \dim_{F_\fr p}\wt H^1_f\bigl(\Q,V_\fr p^\dagger\bigr)=\dim_{F_\fr p}\wt H^1_f\bigl(\Q,V_\fr p^\dagger\otimes\epsilon\bigr)=1 \]
with only finitely many exceptions. In light of the factorization
\begin{equation} \label{factorization-eq}
\wt H^1_f\bigl(K,V_\fr p^\dagger\bigr)\simeq\wt H^1_f\bigl(\Q,V_\fr p^\dagger\bigr)\oplus\wt H^1_f\bigl(\Q,V_\fr p^\dagger\otimes\epsilon\bigr),
\end{equation}
it follows that the equality
\[ \dim_{F_\fr p}\wt H^1_f\bigl(K,V_\fr p^\dagger\bigr)=\dim_{F_\fr p}\wt H^1_f\bigl(\Q,V_\fr p^\dagger\bigr)+\dim_{F_\fr p}\wt H^1_f\bigl(\Q,V_\fr p^\dagger\otimes\epsilon\bigr)=2 \]
should hold for all but finitely many arithmetic primes $\fr p$. For analogous reasons, when $w=1$ it is expected instead that
\[ \wt H^1_f\bigl(K,V_\fr p^\dagger\bigr)=0 \]
for almost all arithmetic primes $\fr p$.

As a piece of notation, if $M$ is a $\Z\bigl[\tilde G_c\bigr]$-module then set
\[ M^\chi:=M\otimes_{\Z[\tilde G_c]}\cl O_F \]
where the tensor product is taken with respect to $\chi:\Z\bigl[\tilde G_c\bigr]\rightarrow\cl O_F$.

In accordance with the above discussion, we can thus state the following
\begin{conj} \label{conj-def-2}\label{conj-def}
Let $d\geq0$ be an integer and fix a character $\chi:\tilde G_d\rightarrow\mathcal O_F^\times$.
\begin{enumerate}
\item If $w=1$ then
\[ \dim_{F_\fr p}\wt H^1_f\bigl(H_d,V_\fr p^\dagger\bigr)^\chi=0 \]
for all but finitely many arithmetic primes $\fr p$ of $\cl R$. In particular, if $w=1$ then $\wt H^1_f\bigl(H_d,V_\fr p^\dagger\bigr)^\chi=0$ for almost all arithmetic primes $\fr p$ such that $\cl L(f_\infty/K,\chi,\fr p)\not=0$.
\item If $w=-1$ then 
\[ \dim_{F_\fr p}\wt H^1_f\bigl(K,V_\fr p^\dagger\bigr)=2 \]
for all but finitely many arithmetic primes $\fr p$ of $\cl R$.
\item If $w=1$ then $\mathrm{rank}_\cl R\wt H^1_f\bigl(H_d,{\bf T}^\dagger\bigr)=0$.
\end{enumerate}
\end{conj}

\begin{rem}
1. We expect that when $w=1$ and $d=0$ part (1) of Conjecture \ref{conj-def} for $\fr p$ of weight $2$ can be proved by extending the techniques and the results of \cite{BD}. Similarly, if $w=1$ then the general weight $2$ case may be dealt with by extending the techniques of \cite{lv}. 

2. Let $w=-1$. For most (non-trivial) characters $\chi$ of $\tilde G_d$ there is no factorization of $\wt H^1_f\bigl(H_d,V_\fr p^\dagger\bigr)^\chi$ analogous to \eqref{factorization-eq}; in this situation, the behaviour of $\wt H^1_f\bigl(H_d,V_\fr p^\dagger\bigr)^\chi$ is not, at present, sufficiently clear to us to formulate a conjecture on its $F_\fr p$-dimension (however, the referee has pointed out that it is reasonable to expect that these Selmer groups should have generic dimension $0$ as $\fr p$ varies, and that Conjecture \ref{conj-def-1} is probably true for most $\chi$ even if $w=-1$).  
\end{rem}

In the case of root number $w=1$ and $d=0$ we can prove that if $\cl J_0\not=0$ then part (1) of Conjecture \ref{conj-def} implies part (3) of the same conjecture.

\begin{teo} \label{teo-def}
Suppose that $w=1$ and assume part (1) of Conjecture \ref{conj-def}. If $\cl J_0\not=0$ then the $\cl R$-module $\wt H^1_f\bigl(K,{\bf T}^\dagger\bigr)$ is torsion.
\end{teo}

\begin{proof} Since $\cl J_0 $ is non-zero, Proposition \ref{R-prop} implies that $\pi_{\fr p}(\cl J_0 )\neq0$ for all but finitely many arithmetic primes $\fr p$ of $\cl R$. Thus, by part (1) of Conjecture \ref{conj-def}, we get that $\wt H^1_f\bigl(K,V_\fr p^\dagger\bigr)=0$ for almost all arithmetic primes $\fr p$.

If $\fr p$ is an arithmetic prime of $\cl R$ then there is a short exact sequence
\[ 0\longrightarrow \wt H^1_f\bigl(K,{\bf T}^\dagger\bigr)_\fr p\big/\fr p\wt H^1_f\bigl(K,{\bf T}^\dagger\bigr)_\fr p\longrightarrow\wt H^1_f\bigl(K,V_\fr p^\dagger\bigr)\longrightarrow\wt H^2_f\bigl(K,{\bf T}^\dagger\bigr)_\fr p[\fr p]\longrightarrow 0 \]
(see the proof of \cite[Corollary 3.4.3]{ho}), which shows that
\begin{equation}\label{vanishing-cohomology-eq}
\wt H^1_f\bigl(K,{\bf T}^\dagger\bigr)_\fr p\big/\fr p\wt H^1_f\bigl(K,{\bf T}^\dagger\bigr)_\fr p=0
\end{equation}
for all but finitely many arithmetic primes $\fr p$ of $\cl R$. As pointed out at the beginning of the proof of \emph{loc. cit.}, the $\cl R$-module $\wt H^1_f\bigl(K,{\bf T}^\dagger\bigr)$ is finitely generated, hence if some $x\in\wt H^1_f\bigl(K,{\bf T}^\dagger\bigr)$ were non-torsion then, by \cite[Lemma 2.1.7]{ho}, we would have that $x\not\in\fr p\wt H^1_f\bigl(K,{\bf T}^\dagger\bigr)_\fr p$ for all but finitely many arithmetic primes $\fr p$. This contradicts \eqref{vanishing-cohomology-eq}, whence $\wt H^1_f\bigl(K,{\bf T}^\dagger\bigr)$ is $\cl R$-torsion. \end{proof}

\begin{rem}
Taking the first part of Conjecture \ref{conj-def-2} for granted when $w=1$ and $d\geq 1$, one could presumably derive the second part by using arguments which are similar to those employed in the proof of Theorem \ref{teo-def}
for $d=0$.\end{rem}

\subsection{Iwasawa theory} \label{section-iwasawa-def}

The goal of this subsection is to formulate a ``main conjecture'' of Iwasawa theory for Hida families (Conjecture \ref{main-def-conj}) in our definite setting.

Set $H_{p^\infty}:=\cup_{m\geq1}H_{p^m}$, denote by $K_\infty\subset H_{p^\infty}$ the anticyclotomic $\Z_p$-extension of $K$ and for every integer $n\geq0$ let $K_n$ be the $n$-th layer of $K_\infty$, i.e. the (unique) subfield of $K_\infty$ such that
\[ G_n:=\Gal(K_n/K)\simeq\Z/p^n\Z. \]
For every integer $n\geq1$ set
\[ d(n):=\min\bigl\{m\in\N\mid K_n\subset H_{p^m}\bigr\}. \]
For example, if $p$ does not divide the class number of $K$ then $d(n)=n+1$ for all $n\geq1$. Let $G_\infty:=\Gal(K_\infty/K)$ (so that $G_\infty\simeq\Z_p$, the isomorphism depending on the choice of a topological generator of $G_\infty$) and define the completed group algebra
\[ \cl R_\infty:=\invlim_n \cl R[G_n]=\cl R[\![G_\infty]\!], \]
where the inverse limit is computed with respect to the canonical maps. Throughout this subsection we make the following
\begin{assumption} \label{ass-regular-def}
The local ring $\cl R$ is regular.
\end{assumption}
In our Iwasawa-theoretic context, this simplifying hypothesis is a natural condition to require (see, e.g., \cite[\S 3.3]{ho} and \cite[Ch. X]{delbourgo}) and gives us some control on the behaviour of $\cl R$ and $\cl R_\infty$ under localizations. For any finitely generated $\cl R_\infty$-module $M$ let
\[ M^\vee:=\Hom_{\Z_p}(M,\Q_p/\Z_p) \]
be its Pontryagin dual, $M_{\rm tors}$ its torsion submodule and ${\rm Char}_{\cl R_\infty}(M)$ its characteristic ideal. Recall that, by definition, ${\rm Char}_{\cl R_\infty}(M)$ is the ideal of $\cl R_\infty$ given by
\[ {\rm Char}_{\cl R_\infty}(M):=\begin{cases}\prod_{\text{ht}(\fr P)=1}\fr P^{{\rm length}(M_\fr P)} & \text{if $M=M_{\rm tors}$}\\[3mm] \{0\} & \text{otherwise}\end{cases} \]
where the product is made over all height one prime ideals of $\cl R_\infty$. Note that, thanks to the assumption that $\cl R$ is regular, the localization $\cl R_{\infty,\fr P}$ is a discrete valuation ring for every prime ideal $\fr P$ of height one in $\cl R_\infty$.

Finally, define the $\cl R_\infty$-module
\[ \wt H^1_{f,{\rm Iw}}\bigl(K_\infty,{\bf T}^\dagger\bigr):=\invlim_n\wt H^1_f\bigl(K_n,{\bf T}^\dagger\bigr), \]
where the inverse limit is taken with respect to the corestriction maps, and the $\cl R_\infty$-module
\[ \wt H^1_{f,{\rm Iw}}\bigl(K_\infty,{\bf A}^\dagger\bigr):=\dirlim_n\wt H^1_f\bigl(K_n,{\bf A}^\dagger\bigr), \]
where the direct limit is taken with respect to the restriction maps.

As before, for all integers $n\geq1$ take the element $\cl P_{p^n}\in H^0\bigl(\mathcal G_{p^n},{\bf D}^\dagger\bigr) $ and set
\[ \cl Q_n:=\cor_{H_{p^{d(n)}}/K_n}\bigl(\cl P_{p^{d(n)}}\bigr)\in H^0\bigl(\Gal(K^{\rm ab}/K_n),{\bf D}^\dagger\bigr). \]
In other words, consider the classes $\cl P_{c,\cl R}$ with $c$ varying in the set of powers of the prime $p$ and take their traces on the anticyclotomic $\Z_p$-extension $K_\infty$ of $K$.
For every integer $n\geq 1$ we introduce the theta-element
\[ \theta_n:=\alpha_p^{-n}\sum_{\sigma\in G_n}\eta_{K_n}\bigl(\cl Q^\sigma_{n}\bigr)\otimes\sigma^{-1}\in\cl R[G_n]. \]
Here $\eta_{K_n}$ is the map of \eqref{eta-H-map-eq} with $H=K_n$ and $\alpha_p\in\cl R^\times$ is the image of the Hecke operator $U_p$ under the morphism
$f_\infty:\fr h_\infty^\ord\rightarrow\cl R$.
Thanks to the compatibility relations enjoyed by big
Heegner points (see \S \ref{U-p-big-subsec}), for all integers $m\geq n\geq1$ one has
$\nu_{m,n}(\theta_m)=\theta_n$, where $\nu_{m,n}:\cl R[G_m]\rightarrow\cl R[G_n]$ is the map
induced by the natural surjection $G_m\twoheadrightarrow G_n$, so one can define
\[ \theta_\infty:=\invlim_n\theta_n\in\cl R_\infty. \]
Note that the element $\theta_\infty$ is not entirely canonical, since it is independent of the choice of the compatible system of big Heegner points $\bigl\{\cl P_{p^n}\bigr\}_{n\geq1}$ only up to multiplication by an element of $G_\infty$. To get rid of this ambiguity, we proceed as follows. Denote by $x\mapsto x^\ast$ the canonical involution of $\cl R_\infty$ acting as $\sigma\mapsto\sigma^{-1}$ on group-like elements. We associate a two-variable $p$-adic $L$-function with the primitive morphism $f_\infty$ and the imaginary quadratic field $K$.
\begin{defi}
The \emph{two-variable $p$-adic $L$-function} attached to $f_\infty$ and $K$ is the element
\[ \cl L_p(f_\infty/K):=\theta_\infty\cdot\theta_\infty^*\in\cl R_\infty. \]
\end{defi}
Always assuming that $\cl R$ is regular, now we formulate our ``main conjecture'' relating $\cl L_p(f_\infty/K)$ to the characteristic ideal of the Pontryagin dual of $\wt H^1_{f,{\rm Iw}}\bigl(K_\infty,{\bf A}^\dagger\bigr)$.
\begin{conj} \label{main-def-conj}
The group $\wt H^1_{f,{\rm Iw}}\bigl(K_\infty,{\bf A}^\dagger\bigr)$ is a finitely generated torsion module over $\cl R_\infty$ and there is an equality
\[ \bigl(\cl L_p(f_\infty/K)\bigr)={\rm Char}_{\cl R_\infty}\Big(\wt H^1_{f,{\rm Iw}}\bigl(K_\infty,{\bf A}^\dagger\bigr)^\vee\Big) \]
of ideals of $\cl R_\infty$.
\end{conj}
The reader should compare Conjecture \ref{main-def-conj} with the Main Conjecture of Iwasawa theory for elliptic curves in the ordinary and anticyclotomic setting that was partially proved by Bertolini and Darmon in \cite{BD} and with the main conjectures over the weight space formulated by Delbourgo in \cite[\S 10.5]{delbourgo}.

\subsection{Vanishing of special values in Hida families} \label{section-vanishing}

The aim of this subsection is to formulate two conjectures on the vanishing at the critical points of the $L$-functions over $K$ of (twists of) the modular forms in the Hida family of $f$ living on the same branch as $f$. As in the previous section, we work in the \emph{definite} case. In fact, our conjectures will involve the elements $\cl L(f_\infty/K,\chi,\fr p)\in F_\fr p$ and the $p$-adic $L$-function $\cl L_p(f_\infty/K)\in\cl R_\infty$ introduced in \S \ref{section-bounding-def} and \S \ref{section-iwasawa-def}, respectively. We remark that results on the vanishing of special values were obtained by Howard in \cite{ho2}, where it is shown that if there exists a weight $2$ form in a Hida family whose $L$-function vanishes to exact order one at $s=1$ then all but finitely many weight $2$ forms in the family enjoy this same property (see \cite[Theorem 8]{ho2}; see also \cite[Theorem 7]{ho2} for the analogous result for order of vanishing zero, which is a consequence of work of Kato, Kitagawa and Mazur).

To begin with, recall the fixed isomorphism $\C\simeq\C_p$ which induces an embedding $\bar\Q_p\hookrightarrow\C$ of an algebraic closure of $\Q_p$ into the complex field, and choose embeddings $F_\fr p\hookrightarrow\bar\Q_p$ for all arithmetic primes $\fr p$ of $\cl R$, so that we can view the $q$-expansion coefficients of the forms $f_\fr p^\dagger$ (introduced in \eqref{twisted-f-eq}) as (algebraic) complex numbers. Next, fix a character $\chi:\tilde G_c\rightarrow\cl O^\times\hookrightarrow\C^\times$, where $c\geq 1$ is an integer and $\cl O$ is a finite extension of $\cl O_F$, and for every arithmetic prime $\fr p$ let $L_K\bigl(f_\fr p^\dagger,\chi,s\bigr)$ be the $L$-function of $f_\fr p^\dagger$ over $K$ twisted by $\chi$ (see, e.g., \cite[p. 268]{GZ} for the definition). Finally, recall the element $\cl L(f_\infty/K,\chi,\fr p)\in F_\fr p$ defined in \eqref{L-chi-eq} and, as in \S \ref{section-bounding-def}, denote by $w$ the common root number of the $L$-functions of (almost all) the modular forms $f_\fr p^\dagger$. As explained in \cite[Ch. IV]{GZ}, one can check that $w$ is also the root number of the twisted $L$-functions $L_K\bigl(f_\fr p^\dagger,\chi,s\bigr)$ for all but finitely many arithmetic primes $\fr p$.

Motivated by \cite[Theorem 1.11]{cv2} and \cite[Theorem 1.3.2]{Z} (which extend \cite[Proposition 11.2]{gross-2} and \cite[Theorem 1.1]{BD-annals-146}), we propose the following
\begin{conj} \label{conj-1}
Let $\fr p$ be a non-exceptional arithmetic prime of $\cl R$ of weight $k_\fr p\geq2$ and let $\chi$ be as above. Assume that $w=1$. The special value $L_K\bigl(f_\fr p^\dagger,\chi,k_\fr p/2\bigr)$ is non-zero if and only if $\cl L(f_\infty/K,\chi,\fr p)$ is non-zero.
\end{conj}
In other words, we conjecture that $L_K\bigl(f_\fr p^\dagger,\chi,s\bigr)$ vanishes at the critical point $s=k_\fr p/2$ precisely when the element $\cl L(f_\infty/K,\chi)\in\cl R$ introduced in \eqref{L-chi-eq} lies in $\fr p$. In particular, the $L$-function $L_K\bigl(f_\fr p^\dagger,\chi,s\bigr)$ is expected not to vanish at $s=k_\fr p/2$ for all but finitely many $\fr p$. The Bloch--Kato conjectures (\cite{BK}) predict that the $L$-function of the form $f_\fr p$ is related to the Selmer group of the associated representation $V_\fr p$, and the $L$-function of $f_\fr p^\dagger$ should be related to the Selmer group of $V_\fr p^\dagger$; in this sense, Conjecture \ref{conj-1} is consistent with Conjecture \ref{conj-def-2}. The reader is suggested to compare the above statement with the conjectures on the generic analytic rank of the forms $f_\fr p$ made by Greenberg in \cite{greenberg}, of which Conjecture \ref{conj-1} can be viewed as a refinement.

Now we want to formulate an analogous conjecture for twists by characters of the Galois group $G_\infty\simeq\Z_p$. Thus let $\chi:G_\infty\rightarrow\cl O^\times$ be a finite (i.e., $p$-power) order character of $G_\infty$, where $\cl O$ is a finite extension of $\cl O_F$. If $\fr p$ is an arithmetic prime of $\cl R$ then the canonical map $\cl R\rightarrow F_\fr p$ gives a map $\cl R_\infty\rightarrow F_\fr p[\![G_\infty]\!]$; composing this with the map $F_\fr p[\![G_\infty]\!]\rightarrow\bar\Q_p$ induced by $\chi$ yields a map
\[ \chi_\fr p:\cl R_\infty\longrightarrow\bar\Q_p. \]
The analogue of Conjecture \ref{conj-1} in this Iwasawa-theoretic context is the following
\begin{conj} \label{conj-2}
Let $\fr p$ be a non-exceptional arithmetic prime of $\cl R$ of weight $k_\fr p\geq2$ and let $\chi$ be as above. Assume that $w=1$. The special value $L_K\bigl(f_\fr p^\dagger,\chi,k_\fr p/2\bigr)$ is non-zero if and only if $\chi_\fr p\bigl(\cl L_p(f_\infty/K)\bigr)$ is non-zero.
\end{conj}

\section{Arithmetic applications and conjectures: the indefinite case} \label{section-indefinite}

In this section we assume that we are in the \emph{indefinite} case, i.e. that the quaternion algebra $B$ is indefinite.

Results in the vein of some of those which follow were also obtained by Fouquet in \cite{fouquet} and \cite{fouquet2} over totally real fields. However, our perspective is different, and (as apparent in the previous sections) the Jacquet--Langlands correspondence plays a much more prominent role in our paper than in the work of Fouquet.

Throughout this section we suppose that Assumptions \ref{gorenstein-ass} and \ref{T-Sh-assumption} hold.

\subsection{Galois representations}

Let $\ell|N^-$ be a prime number. Since $\ell$ is inert in $K$, the completion $K_\ell$ of $K$ at the prime $(\ell)$ is the (unique, up to isomorphism) unramified quadratic extension of $\Q_\ell$. By \cite[Proposition 4.2.3]{Ne-selmer}, the group $H^1(K_\ell,{\bf T})$ is a finitely generated $\cl R$-module, hence (since $\cl R$ is noetherian) the $\cl R$-torsion submodule $H^1\bigl(K_\ell,{\bf T}^\dagger\bigr)_{\rm tors}$ of $H^1\bigl(K_\ell,{\bf T}^\dagger\bigr)$ is a finitely generated $\cl R$-module too. Define $\fr a_\cl R$ to be the annihilator in $\cl R$ of the finitely generated torsion $\cl R$-module $\prod_{\ell|N^-}H^1\bigl(K_\ell,{\bf T}^\dagger\bigr)_{\rm tors}$. Recall the big Heegner class $\kappa_c\in H^1\bigl(G_c^{(Np)},{\bf T}^\dagger\bigr)$ introduced in Definition \ref{big-heegner-class-defi} and denote by the same symbol its image in $H^1\bigl(H_c,{\bf T}^\dagger\bigr)$ under inflation. The next result is a variant of \cite[Proposition 2.4.5]{ho}, to the proof of which we refer for the details we omit.

\begin{prop} \label{prop-classes-selmer}
If $\lambda\in\fr a_\cl R$ and $c$ is prime to $N$ then $\lambda\cdot\kappa_c\in\Sel_{\rm Gr}\bigl(H_c,{\bf T}^\dagger\bigr)$.
\end{prop}

\begin{proof} For any place $v$ of $H_c$ and any $\Gal(\bar\Q/H_c)$-module $M$ let us denote by
\[ \res_v:H^1(H_c,M)\longrightarrow H^1(H_{c,v},M) \]
the restriction map. Fix an integer $c\geq1$ prime to $N$. If $v\nmid Np$ then $\kappa_c$ satisfies the Greenberg local condition at $v$ because of its unramifiedness at $v$.

Now let us assume that $v|Np$ and choose a place $w$ of $\bar\Q$ above $v$. Let $\fr p$ be an arithmetic prime of $\cl R$ of weight $2$ and recall the integer $m:=m_\fr p$ defined in \eqref{m-integer-eq}. Then the natural map ${\bf Ta}^\ord\rightarrow V_{\fr p}$ factors through ${\rm Ta}_p^\ord\bigl({\rm Jac}(\wt X_m)\bigr)$. Let $\kappa_{c,\fr p} $ denote the image of $\kappa_c$ in $H^1\bigl(H_c,V_{\fr p}^\dagger\bigr)$. After restriction to $H_{cp^m}(\boldsymbol\mu_{p^m})$, we see that $V_{\fr p}\simeq V_{\fr p}^\dagger$. Furthermore, the restriction of $\kappa_{c,\fr p} $ to $H^1\bigl(H_{cp^m}(\boldsymbol\mu_{p^m}),V_{\fr p}\bigr)$ is contained in the image of the classical (untwisted) Kummer map
\[ {\rm Jac}(\wt X_m)\bigl(H_{cp^m}(\boldsymbol\mu_{p^m})\bigr)
\longrightarrow H^1\bigl(\wt H_{cp^m},{\rm Ta}_p\bigl({\rm Jac}(\wt X_m)\bigr)\bigr)\longrightarrow H^1\bigl(H_{cp^m}(\boldsymbol\mu_{p^m}),V_{\fr p}\bigr). \] Therefore, by \cite[Example 3.11]{BK}, the restriction of $\kappa_{c,\fr p} $ to $H^1\bigl(H_{cp^m}(\boldsymbol\mu_{p^m}),V_{\fr p}\bigr)$ lies in the Bloch--Kato Selmer group $H^1_f\bigl(H_{cp^m}(\boldsymbol\mu_{p^m}),V_{\fr p}\bigr)$ of $V_{\fr p}  $. Thus, by Proposition \ref{prop-Selmer}, the isomorphic image in $H^1\bigl(H_{cp^m}(\boldsymbol\mu_{p^m}),V_{\fr p}\bigr)$ of the restriction of $\kappa_{c,\fr p} $ to $H^1\bigl(H_{cp^m}(\boldsymbol\mu_{p^m}),V_{\fr p}  \bigr)$ belongs to $\Sel_{\rm Gr}(H_{cp^m}(\boldsymbol\mu_{p^m}),V_\fr p^\dagger)$. Following the arguments in the proof of \cite[Proposition 2.4.5]{ho}, we thus conclude that
\[ \kappa_{c,\fr p} \in\Sel_{\rm Gr}\bigl(H_c,V_{\fr p}  ^\dagger\bigr) \]
for all arithmetic primes $\fr p$ of $\cl R$ of weight $2$.

Once again by \cite[Proposition 2.4.5]{ho}, if $v|pN^+$ then $\res_v(\kappa_{c})$ belongs to $H^1_{\rm Gr}\bigl(H_{c,v},{\bf T}^\dagger\bigr)$, while if $v|N^-$ one can only show that $\res_v(\kappa_c)$ is an $\cl R$-torsion element in $H^1\bigl(H_{c,v},{\bf T}^\dagger\bigr)$. In the latter case, let $\ell$ be the rational prime below $v$. As $\ell$ is inert in $K$, the prime $(\ell)$ of $K$ splits completely in $H_c$, so $H_{c,v}=K_\ell$ and $H^1\bigl(H_{c,v},{\bf T}^\dagger\bigr)=H^1\bigl(K_\ell,{\bf T}^\dagger\bigr)$. Since $\lambda\in\fr a_\cl R$, the result follows. \end{proof}

\begin{rem}
As clear in the proof of Proposition \ref{prop-classes-selmer}, the obstacle towards proving that $\kappa_c$ belongs to $\Sel_{\rm Gr}\bigl(H_c,{\bf T}^\dagger\bigr)$ is the lack of control on the restriction of $\kappa_c$ at places dividing $N^-$.
\end{rem}

With notation as above, fix once and for all a non-zero $\lambda\in\fr a_\cl R$. Thanks to \cite[(21)]{ho}, for every integer $c\geq1$ prime to $N$ the class $\lambda\cdot\kappa_c$ defines a class
\begin{equation} \label{X-c-heegner-eq}
\fr X_c:=\lambda\cdot\kappa_{c}\in\Sel_{\rm Gr}\bigl(H_c,{\bf T}^\dagger\bigr)\simeq\wt H^1_f\bigl(H_c,{\bf T}^\dagger\bigr).
\end{equation}
These cohomology classes are the arithmetic objects in terms of which we will formulate our results and conjectures in this indefinite setting.

\subsection{Bounding Selmer groups} \label{bounding-selmer-subsec}

Define the two cohomology classes
\[ \kappa_0:=\cor_{H_1/K}(\kappa_1)\in H^1\bigl(K,{\bf T}^\dagger\bigr),
\qquad\fr Z_0:=\cor_{H_1/K}(\fr X_1)=\lambda\cdot\kappa_0\in \wt H^1_f\bigl(K,{\bf T}^\dagger\bigr). \]
The following conjecture is the counterpart of \cite[Conjecture 3.4.1]{ho}.

\begin{conj} \label{conj-indef}
The class $\fr Z_0$ is not $\cl R$-torsion.
\end{conj}

Note that Conjecture \ref{conj-indef} is equivalent to the assertion that $\kappa_0$ is not $\cl R$-torsion. The Euler system relations satisfied by the classes $\kappa_c$ (proved in Section \ref{section-euler-system-relations}) yield a proof of the following

\begin{teo} \label{teo-final-indefinite}
Let $\fr p$ be a non-exceptional arithmetic prime of $\cl R$ with trivial character and even weight. If $\fr Z_0$ has non-trivial image in $\wt H^1_f\bigl(K,V_\fr p^\dagger\bigr)$ then $\dim_{F_\fr p}\wt H^1_f\bigl(K,V_\fr p^\dagger\bigr)=1$.
\end{teo}

\begin{proof} Our Euler system of big Heegner classes specializes to an Euler system for $V_\fr p^\dagger$ which enjoys the same properties as the system of cohomology classes considered, in a different arithmetic context, by Nekov\'a\v{r} in \cite{nk-inv}. Then, as in the proof of \cite[Theorem 3.4.2]{ho}, the results in \cite[\S\S 6--13]{nk-inv} yield the theorem. \end{proof}

\begin{rem}
The definition of the class $\fr Z_{0}$ depends on the choice of $\lambda\in\fr a_\cl R$, which is not made explicit in the notation. It might be possible that for different $\lambda_1$ and $\lambda_2$ in $\fr a_\cl R$ the class $\lambda_1\cdot\kappa_0$ is trivial in $\wt H^1_f\bigl(K,V_\fr p^\dagger\bigr)$ while the class $\lambda_2\cdot\kappa_0$ is not. However, since $\fr a_\cl R$ is contained in only \emph{finitely many} arithmetic primes $\fr p$, this occurrence can happen only for a finite number of $\fr p$. Furthermore, if Conjecture \ref{conj-indef} is true then for any choice of $\lambda\in\fr a_\cl R$ the class $\lambda\cdot\kappa_0$ has non-trivial image in $\wt H^1_f\bigl(K,V_\fr p^\dagger\bigr)$ for all but finitely many primes $\fr p$, by \cite[Lemma 2.1.7]{ho}. Thus, under Conjecture \ref{conj-indef}, the different choices of $\lambda\in\fr a_\cl R$ are essentially equivalent.
\end{rem}

The next result is a consequence of Theorem \ref{teo-final-indefinite}.

\begin{teo} \label{teo-final-indefinite2}
Assume Conjecture \ref{conj-indef}. The $\cl R$-module $\wt H^1_f\bigl(K,{\bf T}^\dagger\bigr)$ has rank one.
\end{teo}

\begin{proof} Mimic the arguments in the proof of \cite[Corollary 3.4.3]{ho}, replacing \cite[Theorem 3.4.2]{ho} with Theorem \ref{teo-final-indefinite}. \end{proof}

\subsection{Iwasawa theory} \label{Iwasawa-indefinite}

We formulate an Iwasawa-theoretic ``main conjecture'' (Conjecture \ref{main-indef-conj}) which is the counterpart of Conjecture \ref{main-def-conj} in the indefinite setting. The reader is referred to \cite{ochiai} for results of Ochiai on the cyclotomic Iwasawa main conjecture for Hida families.

Resume the notation of \S \ref{section-iwasawa-def}; in particular, for every integer $n\geq1$ the field $K_n$ is the $n$-th layer of the anticyclotomic $\Z_p$-extension $K_\infty$ of $K$ and $d(n)$ is the smallest natural number such that $K_n$ is a subfield of $H_{p^{d(n)}}$. As in Assumption \ref{ass-regular-def}, we suppose that the local ring $\cl R$ is regular. For every integer $n\geq1$ define the cohomology class
\[ \fr Z_n:=\cor_{H_{p^{d(n)}}/K_n}\bigl(U_p^{1-d(n)}\fr X_{p^{d(n)}}\bigr)\in\wt H^1_f\bigl(K_n,{\bf T}^\dagger\bigr). \]
Since the classes $\fr Z_{n}$ are compatible with respect to corestriction, we can give the following
\begin{defi}
The \emph{two-variable $p$-adic $L$-function} attached to the family $\bigl\{\fr Z_{n}\bigr\}_{n\geq1}$ is the element
\[ \fr Z_{\infty}:=\invlim_n\fr Z_n\in\wt H^1_{f,{\rm Iw}}\bigl(K_\infty,{\bf T}^\dagger\bigr). \]
\end{defi}
Recall that if $M$ is a finitely generated $\cl R_\infty$-module then $M^\vee$ is the Pontryagin dual of $M$. Now we propose our two-variable ``main conjecture''. Since the class $\fr Z_{\infty,\cl R}$ depends on the element $\lambda\in\fr a_\cl R$ appearing in \eqref{X-c-heegner-eq}, in order to state our conjecture we need to assume an additional condition.

\begin{conj} \label{main-indef-conj}
The group $\wt H^1_{f,{\rm Iw}}\bigl(K_\infty,{\bf T}^\dagger\bigr)\big/(\fr Z_\infty)$ is a finitely generated $\cl R_\infty$-module. Moreover, suppose that $\kappa_{p^m}$ belongs to $\wt H^1_f\bigl(H_{p^m},{\bf T}^\dagger\bigr)$ and set $\fr X_{p^m}:=\kappa_{p^m}$ for all $m\geq0$. There is an equality
\begin{equation} \label{main-indef-eq}
{\rm Char}_{\cl R_\infty}\Big(\wt H^1_{f,{\rm Iw}}\bigl(K_\infty,{\bf T}^\dagger\bigr)\big/(\fr Z_{\infty})\Big)^2={\rm Char}_{\cl R_\infty}\Big(\wt H^1_{f,{\rm Iw}}\bigl(K_\infty,{\bf A}^\dagger\bigr)^\vee_{\rm tors}\Big)
\end{equation}
of ideals of $\cl R_\infty$.
\end{conj}
Conjecture \ref{main-indef-conj} extends both \cite[Conjecture 3.3.1]{ho} and the classical Heegner point main conjecture for elliptic curves formulated by Perrin-Riou in \cite{p-r}. Observe that in the special case where $N^-=1$ (or, more generally, for quaternion algebras over totally real number fields satisfying suitable conditions) Fouquet shows in \cite[Theorem A]{fouquet2} that the right-hand side divides the left-hand side in \eqref{main-indef-eq}.

\end{document}